\documentclass[11pt]{article}
\usepackage[a4paper,margin=1in]{geometry}
\usepackage{amsmath,amssymb,amsthm,mathtools}
\usepackage{graphicx}
\usepackage{booktabs}
\usepackage{hyperref}
\usepackage{enumitem}
\usepackage{microtype}
\usepackage{caption}
\usepackage{subcaption}
\usepackage{float}
\usepackage{xcolor}
\usepackage{array}
\hypersetup{
  colorlinks=true,
  linkcolor=blue!50!black,
  citecolor=blue!50!black,
  urlcolor=blue!50!black,
  pdftitle={From a Scalar Parabolic Oscillator to Topological Thermostats: Selective Feedback Control of Harmonic Flow Modes},
  pdfauthor={Sandro Merino},
  pdfsubject={Hodge-theoretic feedback control of harmonic flow modes},
  pdfkeywords={topological thermostat, Hodge decomposition, harmonic cycle flows, finite-channel feedback, swing equations, sensor-actuator placement, power-network loop flows}
}

\theoremstyle{plain}
\newtheorem{theorem}{Theorem}[section]

\newtheorem{lemma}[theorem]{Lemma}
\newtheorem{corollary}[theorem]{Corollary}
\theoremstyle{definition}
\newtheorem{definition}[theorem]{Definition}
\newtheorem{remark}[theorem]{Remark}
\newtheorem{example}[theorem]{Example}

\newcommand{\R}{\mathbb{R}}
\newcommand{\im}{\operatorname{im}}
\newcommand{\rank}{\operatorname{rank}}
\newcommand{\Ker}{\operatorname{ker}}
\newcommand{\cH}{\mathcal{H}}
\newcommand{\cC}{\mathcal{C}}
\newcommand{\dd}{\mathrm{d}}
\newcommand{\T}{\mathbb{T}}
\newcommand{\Z}{\mathbb{Z}}
\newcommand{\norm}[1]{\left\lVert #1\right\rVert}
\newcommand{\Lin}{\mathcal{L}}

\title{From a Scalar Parabolic Oscillator to Topological Thermostats\\[0.3em]
\large Selective Feedback Control of Harmonic Flow Modes}
\author{Sandro Merino\\[0.4em]
\small Basler Kantonalbank\\
\small Correspondence: \href{mailto:smerino@bluewin.ch}{\texttt{smerino@bluewin.ch}}\\
\small ORCID: \href{https://orcid.org/0000-0003-3980-8819}{\texttt{0000-0003-3980-8819}}}
\date{}

\begin{document}
\maketitle

\begin{abstract}
This paper develops a Hodge-theoretic feedback framework for flow-valued states, motivated by scalar parabolic thermostat problems with localized sensors and dual sources.  On a connected graph, the edge space decomposes into a cut space and a harmonic cycle space.  The ideal topological thermostat regulates the cut component toward a prescribed cycle-free transfer target and damps the harmonic component.  For realizations through selected physical edge channels, the first Betti number gives the minimal sensing and actuation ranks required for exact compression of the full harmonic sector.  Minimal rank does not imply local realization of the nonlocal Hodge projector; the remaining cut--harmonic blocks quantify spillover and forcing.

For transmission networks, we formulate ideal and finite-channel nonlinear line-actuated swing closures and linearize the complete closed systems at an angle-cohesive synchronous equilibrium.  With the actuator fixed, passive tangent swing motion preserves any pre-existing harmonic line-flow component.  The ideal closure replaces this conservation law by exponential decay and appends a stable harmonic block without changing the reduced passive nodal--cut spectrum.  Under the channel-rank conditions, finite realizations reproduce the harmonic compression, whereas autonomous harmonic decay and full reduced-state recovery require additional decoupling and Hurwitz conditions.  A target-centred Bregman balance yields nonlinear dissipation for the ideal metric-compatible controller and for canonical colocated finite-channel feedback.  Explicit theta-network, microgrid, ring, and synthetic IEEE 14-bus examples illustrate transfer regulation, cycle-flow damping, and localization-induced transients.  Realistic device dynamics, constraints, and large-scale validation remain open.
\end{abstract}


\section{Introduction}
The scalar heat equation is the canonical dissipative parabolic equation.  In a thermostat problem, a trace sensor and a localized dual source form a feedback pair; in the work of Guidotti and Merino this architecture produces nonlocal perturbations of the Laplacian and may generate oscillatory dynamics \cite{GuidottiMerino1997,GuidottiMerino2000,GuidottiMerino2020,GuidottiMerino2021}.  The present paper asks which part of that mechanism remains meaningful when the state is a flow on a graph or a differential form on a manifold.

Hodge theory supplies the organizing geometry.  Passive Hodge heat damps positive spectral modes while leaving the harmonic kernel invariant.  On a connected graph, fix the Euclidean inner product on the edge space.  Then
\begin{equation}
        \R^m=\im B^\top\oplus\Ker B
\end{equation}
is an orthogonal decomposition.  Let $P_C$ and $P_H$ denote the corresponding orthogonal projections.  Given $f_*\in\im B^\top$ and relaxation rates $\alpha,\kappa>0$, the ideal normal form
\begin{equation}
        \dot f=-\alpha P_C(f-f_*)-\kappa P_Hf
        \label{eq:introduction-topological-thermostat}
\end{equation}
regulates transfer error while damping internal circulation.  The stabilizing signs are intrinsic to the projections: $P_C$ and $P_H$ act as the identity on their respective ranges, so the two projected errors decay at rates $\alpha$ and $\kappa$.  Their matrices need not be entrywise positive.

The harmonic projector is intrinsic once the edge metric is fixed, but its rank is not its spatial support.  If sensing and actuation are restricted to physical edge coordinates, at least $b_1$ independent sensor and actuator channels are required for the full harmonic sector.  The minimal square realization can satisfy
\begin{equation}
        H^\top CKSH=-\kappa I_{b_1},
        \label{eq:introduction-compression}
\end{equation}
but generally not $CKS=-\kappa P_H$.  The coordinate transformation to harmonic variables is a global mixing of edge quantities and does not relocate physical devices.  Consequently, a parsimonious local realization generally creates cut-space spillover and cut-to-harmonic forcing; exact ideal action may require distributed sensing and actuation.

The transmission-grid part uses one nonlinear swing plant with two target-compatible closed-loop realizations.  A line-actuated swing system is closed either by the ideal weighted harmonic projector or by finite-channel feedback.  The target flow $f_*=f(\theta_*,u_*)$ satisfies the static balance equation $p=Bf_*$ and is chosen cycle-free.  Each complete closed vector field is linearized at the same angle-cohesive synchronous target.  The derivative of the constitutive line-flow map enters both tangent systems through the line-flow deviation output
\begin{equation}
        r(\vartheta,w)
        =\mathrm Df(\theta_*,u_*)[\vartheta,w]
        =W_*(B^\top\vartheta+w).
        \label{eq:introduction-linearized-output}
\end{equation}
With the line actuator held fixed, $\dot r=W_*B^\top\nu$ belongs to the weighted cut space, so the passive tangent system conserves $P_H^{W_*}r$.  The ideal line controller replaces this conservation law by exponential harmonic decay without changing the reduced passive swing spectrum.  Finite channels can reproduce the harmonic compression but generally couple the cut and harmonic output sectors; exact autonomy, full-system stability, and transient control are therefore separate requirements.  In the tangent model, harmonic disturbances enter through independent edge variables, whereas the nonlinear model also admits finite winding configurations carrying harmonic flow.

At the nonlinear level, a target-centred Bregman storage is defined on the effective edge-phase variables $\eta=B^\top\theta+u$ and $\omega$.  Along every full closed-loop trajectory it satisfies the controller-independent balance
\begin{equation}
        \dot V=-\omega^\top D\omega+\bigl(f(\eta)-f_*\bigr)^\top\dot u.
        \label{eq:introduction-bregman-balance}
\end{equation}
This closes the conceptual circle to the parabolic prototype.  Hodge geometry selects the regulated flow component; a dual or colocated sensor--actuator pairing turns negative feedback into an additional negative quadratic term.  The ideal weighted projector and the canonical colocated finite-channel controller are therefore dissipative, although the latter may still create transfer and frequency transients.

The paper proceeds from the continuous Hodge motivation to the graph normal form, finite-channel realizations, linear edge-flow examples, and nonlinear swing closures with tangent and energy analysis.  A theta network, meshed microgrids, an exactly solvable ring trajectory, and a synthetic IEEE 14-bus experiment illustrate the successive modelling levels.  Existing PST, FACTS, series-compensation, HVDC, and inverter technologies motivate line-level actuation, but the models remain idealized \cite{Sadikovic2006,ENTSOEPST,SiemensPST,HitachiTCSC,ENTSOEHVDC}.

Related work uses Hodge Laplacians in higher-order network dynamics and synchronization, while co-tree constructions have been used in loop-flow and PST-allocation problems \cite{MuhammadEgerstedtHigherOrder,ArnaudonSakaguchiKuramoto,NurissoSimplicialKuramoto,HuangPST2002}.  Weighted Hodge flows can also improve decay away from the harmonic kernel \cite{ZieglerBalancedHodge,deBadynSummers2026}.  The contribution here is not the underlying Hodge or co-tree facts themselves, but their organization into one control framework: an ideal harmonic-projector benchmark, separate sensing and actuation rank conditions, exact finite-channel compression, explicit cut--harmonic spillover blocks, and stability analysis of the complete closed system.

\section{Trace sensors, dual sources, and the parabolic prototype}
The scalar thermostat problem contains the basic sensor--actuator principle used throughout this paper.  Localized readout is a trace operation.  Let $\Omega\subset\R^n$ be a smooth bounded domain and let the Hilbert Sobolev scale be used.  For
\begin{equation}
        s>\frac n2,
\end{equation}
the Sobolev embedding
\begin{equation}
        H^s(\Omega)\hookrightarrow C^0(\overline\Omega)
\end{equation}
makes point evaluation a bounded linear trace map
\begin{equation}
        \gamma_{x_0}:H^s(\Omega)\longrightarrow\R,
        \qquad
        \gamma_{x_0}u=u(x_0).
\end{equation}
The dual map
\begin{equation}
        \gamma_{x_0}':\R\longrightarrow H^{-s}(\Omega)
\end{equation}
sends $1$ to the Dirac source at $x_0$:
\begin{equation}
        \gamma_{x_0}'1=\delta_{x_0},
        \qquad
        \langle\delta_{x_0},\varphi\rangle
        =
        \gamma_{x_0}\varphi
        =
        \varphi(x_0),
        \quad \varphi\in H^s(\Omega).
\end{equation}
Thus the trace sensor and the Dirac source are not two unrelated formal objects.  They are adjoint objects in the positive and negative Sobolev scales.  This is the trace/source duality appearing in the scalar thermostat analysis of Guidotti--Merino \cite{GuidottiMerino2000}; the surrounding functional-analytic framework is the standard interpolation and extrapolation-space theory for sectorial parabolic generators in the sense of Amann \cite{AmannI,AmannII}.

In abstract form, the scalar feedback perturbation has the rank-one structure
\begin{equation}
        \dot u=Au+b\,k\,\ell(u),
\end{equation}
where $A$ is the parabolic generator, $\ell$ is the trace sensor, and $b$ is the corresponding source in the appropriate negative Sobolev or extrapolation space.  In the colocated point case this is written schematically as
\begin{equation}
        \gamma_{x_0}'\,k\,\gamma_{x_0}u
        =
        k\,u(x_0)\delta_{x_0}.
\end{equation}
The expression is a mnemonic shorthand for the trace/source operator; it becomes mathematically rigorous only after the appropriate function spaces have been specified.  The source is placed in the dual negative space, and the parabolic evolution regularizes it for positive time.

The finite-dimensional network model is the discrete analogue of this trace/source pair.  For finite-dimensional real vector spaces $U$ and $V$ we write
\begin{equation}
        \Lin(U,V)
\end{equation}
for the space of linear maps from $U$ to $V$, and $\Lin(U):=\Lin(U,U)$.  When $U=\R^n$ and $V=\R^m$, canonical coordinates identify these maps with $m\times n$ matrices.  We use the operator notation below and repeat coordinate matrix dimensions only when they add information.

Let $X:=\R^E$ be the edge space.  For every physical edge $e\in E$, let $\delta_e\in X$ denote the canonical edge-coordinate vector,
\begin{equation}
        (\delta_e)_f
        =
        \begin{cases}
        1,&f=e,\\
        0,&f\ne e.
        \end{cases}
        \label{eq:canonical-edge-coordinate-vector}
\end{equation}
Thus $\{\delta_e:e\in E\}$ is the standard basis of $X$, and every edge flow has the expansion
\begin{equation}
        f=\sum_{e\in E}f_e\delta_e.
\end{equation}
A line-flow sensor is the coordinate functional
\begin{equation}
        s_e\in\Lin(X,\R),
        \qquad
        s_e(f)=f_e=\langle\delta_e,f\rangle_X,
\end{equation}
and its Euclidean adjoint is the coordinate injection
\begin{equation}
        s_e^*\in\Lin(\R,X),
        \qquad
        s_e^*1=\delta_e.
\end{equation}
For selected edges $e_1,\ldots,e_r$, this gives a sensor map $S\in\Lin(X,\R^r)$ and a colocated actuator map $C=S^*\in\Lin(\R^r,X)$,
\begin{equation}
        Sf=(f_{e_1},\ldots,f_{e_r})^\top,
        \qquad
        Cu=\sum_{j=1}^r u_j\delta_{e_j}.
\end{equation}
In Euclidean coordinates the adjoint is represented by the transpose, $C=S^\top$; with a weighted edge metric, the corresponding Riesz map gives the dual injection.  Hence the finite-dimensional feedback
\begin{equation}
        \dot f=Af+CKSf
\end{equation}
is the edge-space counterpart of the scalar trace/source perturbation $b\,k\,\ell$.

In the physical network interpretation, the coordinate functional $s_e(f)=f_e$ is a line-flow measurement.  Depending on the modelling level, $f_e$ may represent active power flow, line current, or a derived line-flow signal.  The dual vector $\delta_e$ is not a literal injection of power into the middle of a line.  It is the idealized edge-flow direction generated by a controllable branch device.  Such a device modifies the line-flow law on edge $e$ and, after linearization, produces an input in the corresponding edge coordinate.  Thus the colocated choice $C=S^*$ should be read as follows: measure selected line-flow coordinates and actuate the same line-flow channels.  The mathematical edge injection is therefore a Riesz-dual object in edge space and, physically, a first-order abstraction of line-flow control.

The Hodge-theoretic step is to apply this duality not to arbitrary directions, but to the harmonic sector.  If $H$ is a coordinate map with $\im H=\Ker B$, then $SH$ records how the selected sensors distinguish harmonic cycle modes, while $H^\top C$ records how the actuator directions excite them.  Section~\ref{sec:finite-edge-channel-realization} develops the corresponding full-rank conditions and compressed feedback law.

The geometric generalization is immediate.  A point sensor for a $k$-form $\omega$ on a manifold evaluates the form at a point $p$ on an oriented $k$-vector $\xi\in\Lambda^kT_pM$:
\begin{equation}
\omega\longmapsto \omega_p(\xi).
\end{equation}
The corresponding actuator is a Dirac $k$-current concentrated at $p$ and oriented by $\xi$.  Thus scalar Dirac sources become Dirac currents.  The language of currents is classical in de Rham's theory and geometric measure theory \cite{DeRham,Federer}.  This observation leads first to a continuous Hodge-theoretic motivation; the finite-dimensional network theorem is then obtained by replacing differential forms by graph edge flows.

\section{Hodge heat on manifolds}
We first recall the continuous geometric analogue that motivated the network construction.  Let $M$ be a compact oriented Riemannian manifold without boundary.  Let $\Omega^k(M)$ denote smooth $k$-forms and let
\begin{equation}
\Delta_k=\dd\dd^*+\dd^*\dd
\end{equation}
be the Hodge Laplacian.  Hodge theory gives
\begin{equation}
\cH^k(M)=\Ker\Delta_k\cong H^k_{\rm dR}(M).
\end{equation}
Classical references include Warner, de Rham, and Schwarz \cite{Warner,DeRham,SchwarzHodge}.

We recall the following standard consequence of Hodge theory and the spectral theorem.  Its role is only to fix the sign convention and to identify the finite-dimensional persistent sector that motivates the control construction below.  A brief version of the classical spectral argument is included for completeness.

\begin{theorem}[Classical Hodge heat decomposition]
Let $\omega(t)$ solve
\begin{equation}
\partial_t\omega=-\Delta_k\omega
\end{equation}
on $k$-forms.  Then the harmonic projection of $\omega(t)$ is constant in time, and all positive spectral modes decay exponentially.
\end{theorem}

\begin{proof}[Proof sketch]
The Hodge Laplacian $\Delta_k$ is a nonnegative self-adjoint elliptic operator with compact resolvent.  Hence its $L^2$-spectrum consists of the eigenvalue $0$, whose eigenspace is $\cH^k(M)=\Ker\Delta_k$, and a sequence of positive eigenvalues.  If $P_{\cH}$ denotes the orthogonal projection onto $\cH^k(M)$, the spectral representation of the heat semigroup gives
\begin{equation}
        e^{-t\Delta_k}
        =P_{\cH}+\sum_{\lambda_j>0}e^{-\lambda_jt}P_j,
\end{equation}
where $P_j$ is the spectral projection associated with $\lambda_j$.  Consequently,
\begin{equation}
        P_{\cH}\omega(t)=P_{\cH}\omega(0),
\end{equation}
whereas every positive spectral component is multiplied by $e^{-\lambda_jt}$.  In particular, if $\lambda_1>0$ is the first positive eigenvalue, then
\begin{equation}
        \|(I-P_{\cH})\omega(t)\|_{L^2}
        \leq
        e^{-\lambda_1t}\|(I-P_{\cH})\omega(0)\|_{L^2}.
\end{equation}
This is the standard spectral proof; see, for example, \cite{Warner,DeRham,SchwarzHodge}.
\end{proof}

For scalar heat, $k=0$.  On a connected compact manifold, the harmonic zero-forms are constants, and the harmonic projection is the average value.  Thus average temperature is the degree-zero version of a harmonic projection.  In higher degree, the persistent component is no longer a scalar average but a cohomological class represented by a harmonic form.  The preceding classical result is recorded because its kernel-versus-positive-spectrum structure is the continuous prototype for the finite-dimensional thermostat construction below.

\begin{remark}[Analytic setting on manifolds and bundles]
The preceding account is deliberately condensed.  The analytic theory underlying Hodge Laplacians and Hodge heat equations on differential forms is classical and has been developed and refined rigorously within elliptic and parabolic frameworks.  The short spectral argument above records only the consequence needed for the present construction.

A Hodge heat equation on differential forms is a parabolic equation on sections of the exterior bundle $\Lambda^kT^*M$.  Amann's theory of parabolic equations on Riemannian manifolds and tensor bundles provides a broad functional-analytic setting for such problems, including Sobolev--Slobodeckii and H{\"o}lder scales on uniformly regular Riemannian manifolds and maximal-regularity results for parabolic equations on sections of tensor bundles \cite{AmannParabolicManifolds,AmannCauchyManifolds}.  On compact manifolds this framework is compatible with the familiar elliptic--parabolic theory.  Its relevance here is conceptual: it places scalar heat equations, tensor-field equations, and equations for differential forms within a common functional-analytic setting.
\end{remark}

\section{Finite-rank feedback on harmonic forms}
The preceding preservation result isolates the harmonic sector as a finite-dimensional state space.  Let $\cH^k(M)$ have dimension $r$ and choose an orthonormal harmonic basis $h_1,\dots,h_r$.  Let $H:\R^r\to\cH^k(M)$ denote the coordinate map, and let $H^*$ be its Hilbert-space adjoint.

Let $S$ be a finite family of sensors on $k$-forms and $C$ a finite family of actuators, possibly distributional currents.  Consider the formal perturbed Hodge heat equation
\begin{equation}
\partial_t\omega=-\Delta_k\omega+CKS\omega.
\end{equation}
Write
\begin{equation}
        \omega=Hc+\omega_\perp,
        \qquad
        H^*\omega_\perp=0.
\end{equation}
Since $H^*\Delta_k=0$ on smooth forms, projection to harmonic coordinates gives the complete equation
\begin{equation}
        \dot c
        =
        H^*CKSHc
        +
        H^*CKS\omega_\perp.
        \label{eq:continuous-harmonic-coordinate-equation}
\end{equation}
The matrix
\begin{equation}
        \Gamma_{\cH^k}(CKS):=H^*CKSH
\end{equation}
is the harmonic compression of the feedback operator.  It governs the autonomous harmonic-coordinate dynamics when the state is restricted to $\cH^k(M)$, or more generally when the coupling term $H^*CKS\omega_\perp$ vanishes.  Thus the finite-dimensional rank principle is an exact statement about harmonic compression: if the sensors separate and the actuators excite the harmonic sector with full rank, then a prescribed matrix can be realized as $\Gamma_{\cH^k}(CKS)$.  The full perturbed Hodge heat equation may still couple non-harmonic and harmonic components unless an additional structural condition removes the second term in \eqref{eq:continuous-harmonic-coordinate-equation}.

The analytic treatment of point or current actuators requires appropriate Sobolev or extrapolation spaces.  This is precisely where the scalar semigroup viewpoint is useful.  Dirac sources are legitimate mathematical objects when interpreted in the correct negative Sobolev or extrapolation spaces.  For example, a point source dual to a trace belongs to a space of negative order, and boundary point sources are treated through the corresponding boundary trace and dual boundary Sobolev scale.  The parabolic evolution then regularizes such data for positive time.  A complete analytic theory of current feedback on manifolds is beyond the scope of this draft; here it is presented as the continuous motivation for the finite-dimensional network theorem below.

\section{Forced topological ringing on the flat torus}
This example is included because it is the most geometrically intuitive continuous model of the mechanism.  The flat torus has two independent one-dimensional cycles, hence a two-dimensional harmonic one-form space.  In this geometry the complete mechanism can be written explicitly, without any numerical discretization.

Let
\begin{equation}
\T^2=\R^2/\Z^2
\end{equation}
be the flat torus with coordinates $x,y$.  With the positive Hodge-Laplacian convention used in this paper, the Hodge Laplacian on one-forms acts componentwise on the flat torus:
\begin{equation}
\Delta_1(f\,dx+g\,dy)
=
(-\partial_x^2-\partial_y^2)f\,dx
+
(-\partial_x^2-\partial_y^2)g\,dy .
\end{equation}
The harmonic one-forms are
\begin{equation}
dx,\qquad dy,
\end{equation}
and a harmonic one-form has the form
\begin{equation}
\omega_H=a\,dx+b\,dy.
\end{equation}
The coefficients are period coordinates:
\begin{equation}
a=\int_{\gamma_x}\omega,
        \qquad
        b=\int_{\gamma_y}\omega,
\end{equation}
where $\gamma_x$ and $\gamma_y$ are the two fundamental cycles.

Choose the non-harmonic eigenform
\begin{equation}
        \eta=\sin(2\pi x)\,dx .
\end{equation}
It has zero harmonic projection, since its coefficient has mean zero, and
\begin{equation}
        \Delta_1\eta=4\pi^2\eta .
\end{equation}
Therefore the Hodge heat equation
\begin{equation}
        \partial_t\omega=-\Delta_1\omega
\end{equation}
with initial value
\begin{equation}
        \omega_0=a_0dx+b_0dy+\varepsilon\eta
\end{equation}
has the explicit solution
\begin{equation}
        \omega(t)=a_0dx+b_0dy+
        \varepsilon e^{-4\pi^2t}\sin(2\pi x)\,dx .
        \label{eq:torus-hodge-heat-explicit}
\end{equation}
Thus the positive spectral mode is damped exponentially, while the harmonic periods are unchanged.

\begin{remark}[Relation with the topological thermostat normal form]
\label{rem:torus-normal-form}
The Hodge Laplacian should not be identified with an orthogonal
projector.  It vanishes on the harmonic sector but assigns its
positive eigenvalues as decay rates to the non-harmonic spectral modes.
For the exact eigenform
\[
        \eta=\sin(2\pi x)\,dx
        =\dd\!\left(-\frac{\cos(2\pi x)}{2\pi}\right)
\]
used above, however,
\[
        \Delta_1\eta=4\pi^2\eta.
\]
Consequently, on the invariant subspace
\[
        \mathcal Y
        =
        \operatorname{span}\{\eta,dx,dy\},
\]
one has
\[
        \Delta_1=4\pi^2P_\eta,
\]
where $P_\eta$ is the $L^2$-orthogonal projection onto
$\operatorname{span}\{\eta\}$ and $P_{\cH}$ is the $L^2$-orthogonal
projection onto $\cH^1(\T^2)=\operatorname{span}\{dx,dy\}$.

For an exact target $\omega_*=\varepsilon_*\eta$, the continuous
counterpart of the normal form
\eqref{eq:introduction-topological-thermostat} is therefore
\begin{equation}
        \partial_t\omega
        =
        -\Delta_1(\omega-\omega_*)
        -\kappa P_{\cH}\omega
        =
        -4\pi^2P_\eta(\omega-\omega_*)
        -\kappa P_{\cH}\omega,
        \qquad
        \omega\in\mathcal Y.
        \label{eq:torus-topological-thermostat-normal-form}
\end{equation}
Thus, on the selected exact mode, the Hodge heat term coincides with
the cut-space relaxation part of the projector normal form at rate
$4\pi^2$.  On the full non-harmonic space, however, $\Delta_1$ damps
different eigenmodes at their natural spectral rates.  The
Hodge heat equation displayed above is the special case $\omega_*=0$
and $\kappa=0$.
\end{remark}

Forced topological ringing uses the same sectorwise architecture but
replaces the dissipative harmonic block $-\kappa I_2$ by a
skew-symmetric generator.  If the feedback induces
\begin{equation}
        \dot a=-\Omega b,
        \qquad
        \dot b=\Omega a,
\end{equation}
then, for $a_0=1$ and $b_0=0$,
\begin{equation}
        \omega(t)=
        \cos(\Omega t)\,dx
        +\sin(\Omega t)\,dy
        +\varepsilon e^{-4\pi^2t}\sin(2\pi x)\,dx .
        \label{eq:torus-forced-ringing-explicit}
\end{equation}
This formula is the intended picture.  The topology does not oscillate by itself.  Hodge heat removes the positive spectral mode, while the feedback rotates the two-dimensional harmonic zero-mode sector.

\begin{example}[Point-current realization on the flat torus]
Choose a point sensor $p\in\T^2$ and read the two directional components
\begin{equation}
S_x(\omega)=\omega_p(\partial_x),
        \qquad
        S_y(\omega)=\omega_p(\partial_y).
\end{equation}
Choose a point actuator at $q\in\T^2$ with one-current directions $\partial_x$ and $\partial_y$.  The cross-coupled feedback
\begin{equation}
-\Omega S_y(\omega)\delta_{q,\partial_x}
        +\Omega S_x(\omega)\delta_{q,\partial_y}
\end{equation}
induces, on the harmonic sector, the rotation matrix
\begin{equation}
\begin{pmatrix}0&-\Omega\\ \Omega&0\end{pmatrix},
\end{equation}
up to the normalization determined by the harmonic projection of the Dirac currents.
\end{example}

On the flat torus, the harmonic forms $dx$ and $dy$ are constant, so evaluation at any point separates the two harmonic components:
\begin{equation}
S_x(dx)=1,
\quad S_x(dy)=0,
\qquad
S_y(dx)=0,
\quad S_y(dy)=1.
\end{equation}
Similarly, the harmonic projection of a Dirac one-current oriented in the $x$-direction is nonzero in the $dx$ direction, and the harmonic projection of a Dirac one-current oriented in the $y$-direction is nonzero in the $dy$ direction.  Thus the sensor matrix and the actuator matrix both have rank two on the harmonic sector $\operatorname{span}\{dx,dy\}$.  This verifies the rank principle in the simplest continuous case.

\section{Edge spaces and the Hodge decomposition of a network}
We now translate the preceding continuous motivation into the finite-dimensional setting of network edge flows.  This is not merely a discretization convenience.  It is also the setting in which, in our view, the most immediate applied model problems occur: line currents in inverter microgrids, line flows in meshed transmission networks, and more generally flows on networks with nontrivial cycles.  In such graph models the distinction between useful source-to-load transfer and internal circulation becomes especially transparent.

We use the standard cochain viewpoint of graph Hodge theory.  The underlying graph is undirected, but we fix once and for all a reference orientation for each edge in order to write coordinates and matrices.  For the incidence matrix itself we adopt below the power-flow sign convention, which differs from the usual coboundary convention by a global minus sign.  Vertex functions form the space of $0$-cochains
\begin{equation}
C^0(G)=L^2(V)\cong\R^n,
\end{equation}
and skew-symmetric edge functions form the space of $1$-cochains
\begin{equation}
C^1(G)=L^2_{\wedge}(E)\cong\R^m.
\end{equation}
We identify $\R^m$ with $X=\R^E$ through the canonical edge-coordinate basis $\{\delta_e:e\in E\}$ defined in \eqref{eq:canonical-edge-coordinate-vector}.  The vector $\delta_e$ represents a unit flow in the chosen reference orientation of edge $e$.  Relabelling the edges merely permutes this basis, while reversing an edge orientation changes the sign convention for its coordinate; statements about which physical edges occur in the support are unaffected.

Elements of $C^1(G)$ are edge flows.  If the reference orientation of an edge $e$ is $i\to j$, then a coordinate vector $f\in\R^m$ represents the skew-symmetric edge function $X_f$ with
\begin{equation}
X_f(i,j)=f_e,\qquad X_f(j,i)=-f_e.
\end{equation}
Thus a positive coordinate means flow in the chosen reference orientation, while a negative coordinate means flow in the opposite physical direction.  The reference orientation is bookkeeping; the actual directed weighted graph at a given instant is encoded by the signs and magnitudes of the edge-flow $1$-cochain.  This is the convention emphasized by Lim: the underlying graph may remain undirected while its edge flows encode directed weighted data \cite{LimHodgeGraphs}.

Let $G=(V,E)$ be connected, with $n=|V|$ and $m=|E|$, and let
\begin{equation}
        B\in\Lin(X,\R^n)
        \label{eq:incidence-operator-space}
\end{equation}
be the incidence operator associated with the chosen reference orientation.  We use the power-flow sign convention: the column associated with a reference-oriented edge has entry $+1$ at the tail and $-1$ at the head.  Thus a positive edge coordinate represents flow from tail to head, and $Bf$ is the corresponding vector of net nodal outflows or injections.  The transpose
\begin{equation}
        B^\top\in\Lin(\R^n,X)
        \label{eq:transpose-incidence-operator-space}
\end{equation}
maps nodal potentials to oriented edge drops, and
\begin{equation}
        (B^\top v)_e
        =
        v_{\mathrm{tail}(e)}-v_{\mathrm{head}(e)}
\end{equation}
is the oriented potential drop along the reference direction.  With the usual cochain convention the coboundary is $d_0=-B^\top$; this global sign has no effect on $\im B^\top$, $\Ker B$, or the Hodge projections used below.

\begin{definition}[Cut and harmonic cycle spaces]
The cut space is
\begin{equation}
\cC:=\im B^\top\subset\R^m.
\end{equation}
The harmonic cycle space is
\begin{equation}
\cH:=\Ker B\subset\R^m.
\end{equation}
We denote by $P_C,P_H\in\Lin(X)$ the orthogonal projections with ranges $\cC$ and $\cH$, respectively.
\end{definition}

\begin{definition}[Edge support and matrix support]
Let $Y\subseteq X=\R^E$ be a linear subspace.  Its physical edge support is
\begin{equation}
        \operatorname{esupp}(Y)
        :=
        \left\{
        e\in E:\text{there exists }y\in Y\text{ with }y_e\ne0
        \right\}.
        \label{eq:edge-support-subspace}
\end{equation}
If $P_Y$ denotes the Euclidean orthogonal projection onto $Y$, then
\begin{equation}
\begin{aligned}
        e\in\operatorname{esupp}(Y)
        &\iff P_Y\delta_e\ne0\\
        &\iff
        \langle\delta_e,P_Y\delta_e\rangle_X
        =\norm{P_Y\delta_e}^2>0.
\end{aligned}
        \label{eq:edge-support-projection-characterization}
\end{equation}
We therefore also write $\operatorname{esupp}(P_Y):=\operatorname{esupp}(Y)$.  This notion depends on the canonical physical edge coordinates, but not on a chosen basis of the subspace $Y$.

For $L\in\Lin(X)$, its matrix support in the canonical edge basis is
\begin{equation}
        \operatorname{msupp}(L)
        :=
        \left\{
        (e,f)\in E\times E:
        \langle\delta_e,L\delta_f\rangle_X\ne0
        \right\}.
        \label{eq:matrix-support-edge-operator}
\end{equation}
Thus $\operatorname{msupp}(L)$ records exactly the nonzero matrix entries of $L$ in physical edge coordinates.  A larger set such as $A_0\times R$ below describes instead the positions that may potentially be nonzero when only sensor and actuator supports are prescribed.
\end{definition}

Since $G$ is connected, $\rank B=n-1$, and the fundamental orthogonal decomposition is
\begin{equation}
\R^m=\im B^\top\oplus\Ker B.
\end{equation}
The dimension of $\cH$ is the first Betti number
\begin{equation}
b_1=m-n+1.
\end{equation}
The edge support of the harmonic space is
\begin{equation}
        E_{\rm cyc}
        :=\operatorname{esupp}(\cH)
        =\operatorname{esupp}(P_H).
        \label{eq:cyclic-edge-support}
\end{equation}
It is exactly the set of edges belonging to at least one cycle, equivalently the set of non-bridge edges.  Indeed, an edge belongs to a cycle precisely when some vector in $\Ker B$ has a nonzero coordinate on that edge.  Formula \eqref{eq:edge-support-projection-characterization} gives the harmonic-basis-independent test
\begin{equation}
        e\in E_{\rm cyc}
        \iff
        \langle\delta_e,P_H\delta_e\rangle_X
        =\norm{P_H\delta_e}^2>0.
        \label{eq:cyclic-edge-projector-test}
\end{equation}
The set $E_{\rm cyc}$ is independent of the orthonormal basis used to represent $\Ker B$.  Positive diagonal changes of edge metric alter the numerical projector entries but not the underlying harmonic subspace or its physical edge support.
This is the one-dimensional graph Hodge decomposition for edge flows on the network.  The adjective one-dimensional is important.  In the finite-dimensional network part of this paper the graph is not completed by filled faces or higher-dimensional cells; even a triangle is treated as a loop of three edges, not as a filled $2$-simplex.  Thus the cochain complex stops at degree one,
\begin{equation}
C^0(G)\xrightarrow{\,-B^\top\,}C^1(G)\longrightarrow 0,
\end{equation}
and the next coboundary is absent:
\begin{equation}
d_1=0,
        \qquad
        \im d_1^*=\{0\}.
\end{equation}
Consequently Lim's three-term Hodge decomposition for $1$-cochains on a higher-dimensional complex,
\begin{equation}
C^1=\im d_0\oplus \im d_1^*\oplus \mathcal H^1,
\end{equation}
collapses here to
\begin{equation}
C^1(G)=\im B^\top\oplus \Ker B.
\end{equation}
In this specialization, every divergence-free edge flow belongs to the harmonic cycle-flow space.  Hence $P_H$ projects onto the whole graph cycle space $\Ker B$, whose dimension is the first Betti number $b_1=m-n+1$.  Small local cycles are included in this space.  They would be separated into a coexact or curl-type sector only if the graph were augmented by filled $2$-cells.  For graph Hodge theory and its relation to elementary linear algebra and cohomology, see Lim \cite{LimHodgeGraphs}.

\section{From Hodge coordinates to the topological thermostat}
The Hodge decomposition of edge space is not, by itself, a dynamical statement.  It is a coordinate decomposition.  It identifies which part of an edge-flow vector is a transfer flow and which part is a cycle flow, but it does not say how these components evolve under a given physical or controlled dynamics.

Let $G=(V,E)$ be connected, with the fixed reference orientation chosen above.  The edge-flow space admits the orthogonal splitting
\begin{equation}
\R^m=\im B^\top\oplus\Ker B.
\end{equation}
We write
\begin{equation}
P_C:\R^m\to\im B^\top,
        \qquad
        P_H:\R^m\to\Ker B
\end{equation}
for the corresponding projections.  Thus every edge flow $f\in\R^m$ decomposes uniquely as
\begin{equation}
f=f_C+f_H,
        \qquad
        f_C=P_Cf,
        \qquad
        f_H=P_Hf.
\end{equation}
The component $f_C$ is the cut, or transfer, component.  The component $f_H$ is the harmonic, or cycle-flow, component.  Throughout this section, $P_H$ refers to the one-dimensional graph-network projector onto $\Ker B$ described above; it includes all independent graph cycles, including local loops, because no $2$-cochain space has been introduced.

Suppose now that an autonomous edge-flow dynamics is given abstractly by
\begin{equation}
\dot f=F(f).
\end{equation}
Then the Hodge coordinates satisfy
\begin{equation}
\dot f_C=P_CF(f_C+f_H),
\end{equation}
and
\begin{equation}
\dot f_H=P_HF(f_C+f_H).
\end{equation}
These equations are always true, but they need not be decoupled.  In general, the vector field $F$ may mix the cut and harmonic subspaces.  Thus Hodge theory supplies coordinates; it does not automatically produce invariant dynamical subsystems.  Invariance of the two summands requires additional structure, for instance a symmetry of the dynamics or a vector field specifically adapted to the decomposition.

This is the point at which the thermostat idea enters.  Motivated by the scalar sensor--actuator thermostat, we do not try to derive a passive network law from the Hodge decomposition alone.  Instead, we propose a closed-loop vector field whose action is prescribed in Hodge coordinates.  The design requirement is deliberately simple.  First, the useful transfer component should relax toward a prescribed transfer flow $f_*\in\im B^\top$:
\begin{equation}
\dot f_C=-\alpha(f_C-f_*),
        \qquad
        \alpha>0.
\end{equation}
Second, the harmonic cycle-flow component should be damped:
\begin{equation}
\dot f_H=-\kappa f_H,
        \qquad
        \kappa>0.
\end{equation}
Adding these two equations and using the decomposition $f=f_C+f_H$ gives
\begin{equation}
\dot f
=
-\alpha P_C(f-f_*)
-\kappa P_Hf.
\end{equation}
This is the ideal topological thermostat.

The equation should therefore be read as an ideal closed-loop normal form, not as the passive dynamics of a power network.  The passive grid model may be algebraic, as in DC power flow, or second order, as in the swing equations.  The thermostat equation instead specifies what a feedback layer is meant to accomplish on the edge-flow output: track the prescribed transfer flow and damp the cohomological component.

It is also useful to record the sensor--actuator interpretation.  The ideal harmonic sensor is the projection
\begin{equation}
f\longmapsto P_Hf,
\end{equation}
which reads out the cycle-flow part of the line-flow vector.  The ideal actuator applies the opposite vector field
\begin{equation}
P_Hf\longmapsto -\kappa P_Hf.
\end{equation}
Similarly, the cut-space regulator reads
\begin{equation}
f\longmapsto P_C(f-f_*)
\end{equation}
and actuates
\begin{equation}
P_C(f-f_*)\longmapsto -\alpha P_C(f-f_*).
\end{equation}
Thus the topological thermostat is a feedback law whose measured variable is not the whole line-flow vector but its harmonic projection.  This is the small control-theoretic twist: the Hodge coordinates are classical, but the harmonic coordinate is made into the regulated output.

\begin{theorem}[Topological thermostat]
\label{thm:topological-thermostat}
Let $G$ be connected, equip $\R^m$ with the Euclidean edge inner product, and let $\R^m=\cC\oplus\cH$ be the resulting orthogonal cut-cycle decomposition.  Let $f_*\in\cC$ be a prescribed transfer flow.  Consider
\begin{equation}
\dot f=-\alpha P_C(f-f_*)-\kappa P_Hf,
        \qquad \alpha,\kappa>0.
        \label{eq:topological-thermostat}
\end{equation}
Write $f=f_C+f_H$ with $f_C\in\cC$ and $f_H\in\cH$.  Then
\begin{equation}
f_C(t)-f_* = e^{-\alpha t}(f_C(0)-f_*),
\end{equation}
and
\begin{equation}
f_H(t)=e^{-\kappa t}f_H(0).
\end{equation}
In particular, the harmonic loop-flow component is exponentially damped without disturbing the prescribed cut-space target.
\end{theorem}

\begin{proof}
Apply $P_C$ and $P_H$ to the equation.  Since $P_CP_H=P_HP_C=0$ and $P_Cf_*=f_*$, the cut component satisfies
\begin{equation}
\dot f_C=-\alpha(f_C-f_*),
\end{equation}
while the harmonic component satisfies
\begin{equation}
\dot f_H=-\kappa f_H.
\end{equation}
Equivalently, with $e=f-f_*$, orthogonality gives the dissipation identity
\[
\frac12\frac{\dd}{\dd t}\norm{e(t)}^2
=
-\alpha\norm{P_Ce(t)}^2
-\kappa\norm{P_He(t)}^2.
\]
The same calculation holds, with the induced norm, for any fixed edge inner product with respect to which $P_C$ and $P_H$ are the corresponding orthogonal projections.  The two projected equations give the stated formulas.
\end{proof}

\begin{remark}
The sign convention in \eqref{eq:topological-thermostat} is coordinate independent: the projectors act as $+I$ on their ranges, while their matrices in oriented edge coordinates may contain entries of either sign.  The theorem is deliberately elementary.  Its purpose is to identify the correct control target.  A generic damping law acts on the whole edge-flow vector.  The topological thermostat acts on the component in $\Ker B$.  In applications this is the loop-flow or cycle-current component.  The ideal law uses the exact Hodge projector; the finite-channel realization in the next section replaces this global projector by a sensor--actuator operator that is exact after compression to harmonic coordinates.
\end{remark}

\paragraph{Betti number as controller-rank count.}
This formulation gives a useful control-theoretic reading of a familiar Hodge-theoretic fact.  Of course, it adds no new fact to Hodge theory: the identity
\begin{equation}
\dim \cH^k=b_k
\end{equation}
is one of its basic consequences.  Likewise, in graph and circuit theory, $b_1$ is the circuit rank, or cyclomatic number, and counts the independent cycle or loop-current variables of the network; see, for instance, Harary \cite{HararyGraphTheory} and the classical circuit-theory treatment of Chua, Desoer and Kuh \cite{ChuaDesoerKuh}.  Closely related uses of Hodge Laplacians, homological information, and higher-order Laplacians in network dynamics and control have also been studied; see Lim \cite{LimHodgeGraphs} and Muhammad--Egerstedt \cite{MuhammadEgerstedtHigherOrder}.  The point made here is more specific: once the topological thermostat is defined as a feedback law acting on the harmonic sector, the Betti number becomes a concrete controller-rank count.  It is the minimum number of independent harmonic channels needed to observe and actuate the full topological sector.

In the graph case, $\cH^1=\Ker B$ and, for a connected graph,
\begin{equation}
\begin{aligned}
        b_1&=|E|-|V|+1,\\
        b_1(\Theta)&=3-2+1=2,\\
        b_1(\mathrm{IEEE\ 14})&=20-14+1=7,\\
        b_1(\T^2)&=2.
\end{aligned}
\end{equation}
On the two-torus the harmonic basis is represented by $dx$ and $dy$, so two independent period channels suffice for the full harmonic one-form sector.  The useful observation is not the Betti-number identity itself, but the way the thermostat equation makes it operational: topology prescribes the minimal rank of a full topological controller.

\subsection{Higher-order Hodge-thermostat normal forms}
\label{subsec:higher-order-hodge-thermostats}

The finite-dimensional network part of this paper uses only vertex functions and edge-flow $1$-cochains.  Nevertheless, the normal-form idea is not restricted to one-dimensional graphs.  We record the higher-order version as an explanatory extension, because it shows that the edge-flow thermostat is the $k=1$ member of a more general Hodge-coordinate feedback principle.

Let $\mathcal K$ be a finite simplicial or cellular complex, and let
\begin{equation}
C^k(\mathcal K)
\end{equation}
be the real vector space of $k$-cochains, equipped with a fixed inner product.  Denote the coboundary maps by
\begin{equation}
d_{k-1}:C^{k-1}(\mathcal K)\to C^k(\mathcal K),
        \qquad
 d_k:C^k(\mathcal K)\to C^{k+1}(\mathcal K),
\end{equation}
and let $d_{k-1}^*$ and $d_k^*$ be their adjoints.  The $k$-Hodge Laplacian is
\begin{equation}
\Delta_k=d_{k-1}d_{k-1}^*+d_k^*d_k.
\end{equation}
The corresponding finite-dimensional Hodge decomposition is
\begin{equation}
C^k(\mathcal K)
=
\im d_{k-1}
\oplus
\cH^k
\oplus
\im d_k^*,
\end{equation}
where
\begin{equation}
\cH^k=\Ker\Delta_k=\Ker d_k\cap \Ker d_{k-1}^*
\end{equation}
is the $k$-th harmonic cochain space.  Its dimension is the $k$-th Betti number $b_k$.

Let
\begin{equation}
P_{\mathrm{ex}}^k,\qquad P_{\cH}^k,\qquad P_{\mathrm{coex}}^k
\end{equation}
denote the orthogonal projections onto
\begin{equation}
\im d_{k-1},
        \qquad
\cH^k,
        \qquad
\im d_k^*,
\end{equation}
respectively.  A higher-order analogue of the topological thermostat is the closed-loop normal form
\begin{equation}
\dot x
=
-\alpha P_{\mathrm{ex}}^k(x-x_*)
-\beta P_{\mathrm{coex}}^k x
-\kappa P_{\cH}^k x,
        \qquad
        x(t)\in C^k(\mathcal K),
\end{equation}
where
\begin{equation}
x_*\in \im d_{k-1},
        \qquad
        \alpha,\beta,\kappa>0.
\end{equation}
Writing
\begin{equation}
x=x_{\mathrm{ex}}+x_{\cH}+x_{\mathrm{coex}}
\end{equation}
according to the Hodge decomposition, the three components satisfy
\begin{equation}
\frac{d}{dt}P_{\mathrm{ex}}^k(x-x_*)
=
-\alpha P_{\mathrm{ex}}^k(x-x_*),
\end{equation}
\begin{equation}
\frac{d}{dt}P_{\mathrm{coex}}^k x
=
-\beta P_{\mathrm{coex}}^k x,
\end{equation}
and
\begin{equation}
\frac{d}{dt}P_{\cH}^k x
=
-\kappa P_{\cH}^k x.
\end{equation}
Thus the exact component tracks the prescribed exact target, the coexact component is damped, and the cohomological component is damped at the prescribed rate $\kappa$.

This formulation also clarifies the relation with passive Hodge heat.  The uncontrolled equation
\begin{equation}
\dot x=-\Delta_kx
\end{equation}
damps the positive Hodge-Laplacian modes but leaves
\begin{equation}
\cH^k=\Ker\Delta_k
\end{equation}
unchanged.  Adding a feedback term of the form
\begin{equation}
-\kappa P_{\cH}^k x
\end{equation}
is therefore the direct higher-order analogue of the topological thermostat: it acts precisely on the harmonic zero-mode sector left uncontrolled by passive Hodge heat.

The graph theorem above is the one-dimensional specialization of this normal form.  In the network part of the present paper the graph is treated as a one-dimensional complex.  Hence
\begin{equation}
C^0(G)\xrightarrow{\,d_0\,}C^1(G)\longrightarrow 0,
        \qquad
        d_0=-B^\top,
\end{equation}
and there are no $2$-cochains:
\begin{equation}
d_1=0,
        \qquad
        \im d_1^*=0.
\end{equation}
Consequently,
\begin{equation}
C^1(G)=\im B^\top\oplus\Ker B,
        \qquad
        \cH^1=\Ker B.
\end{equation}
The higher-order thermostat therefore reduces exactly to
\begin{equation}
\dot f
=
-\alpha P_C(f-f_*)
-\kappa P_Hf,
\end{equation}
which is Theorem~\ref{thm:topological-thermostat}.

For $k=0$, the state is a vertex function.  On a connected graph or connected complex, the harmonic $0$-cochains are the constant functions.  Thus the degree-zero case recovers the familiar scalar phenomenon that passive heat flow damps nonconstant components while preserving the average.  The present paper uses this scalar case as motivation.  Its finite-dimensional network control results are specialized to $k=1$, where the controlled variables are edge-flow $1$-cochains.

\section{Finite-channel sensor--actuator realizations}
\label{sec:finite-edge-channel-realization}
The ideal law uses the global Hodge projector $P_H$.  On a finite graph this projector is already finite rank, with $\rank P_H=b_1$.  Thus the distinction made in this section is not one of dimension, but between the ideal Hodge operator and a realization through selected physical edge channels.

Recall that $X=\R^E\cong\R^m$ is the edge space, and put
\begin{equation}
        \cC=\im B^\top,
        \qquad
        \cH=\Ker B,
        \qquad
        X=\cC\oplus\cH.
\end{equation}
Let $P_C$ and $P_H$ be the corresponding orthogonal projections.  Choose an isometric coordinate map
\begin{equation}
        H\in\Lin(\R^{b_1},X)
\end{equation}
whose columns form an orthonormal basis of $\cH$, so that $P_H=HH^\top$.  The harmonic-compression map is
\begin{equation}
        \Gamma_{\cH}\in\Lin\bigl(\Lin(X),\Lin(\R^{b_1})\bigr),
        \qquad
        \Gamma_{\cH}(L):=H^\top L H.
        \label{eq:harmonic-compression-definition}
\end{equation}
It restricts $L$ to harmonic inputs and projects the result back to harmonic coordinates.  If $H$ is replaced by $HQ$, $Q\in O(b_1)$, then $\Gamma_{\cH}(L)$ is replaced by $Q^\top\Gamma_{\cH}(L)Q$; hence $\Gamma_{\cH}(L)=-\kappa I_{b_1}$ is basis independent.

The ideal harmonic damping operator is
\begin{equation}
        L_{\rm id}:=-\kappa P_H.
\end{equation}
It has the block identities
\begin{equation}
        P_CL_{\rm id}P_C=P_CL_{\rm id}P_H=P_HL_{\rm id}P_C=0,
        \qquad
        P_HL_{\rm id}P_H=-\kappa P_H,
\end{equation}
and therefore, relative to $X=\cC\oplus\cH$,
\begin{equation}
        L_{\rm id}
        =
        \begin{pmatrix}
        0&0\\
        0&-\kappa I_{b_1}
        \end{pmatrix}.
        \label{eq:ideal-harmonic-block-form}
\end{equation}

A finite-channel feedback has the form
\begin{equation}
        L_{\rm ch}=CKS.
        \label{eq:finite-channel-feedback-operator}
\end{equation}
The channel maps have domains
\begin{equation}
        S\in\Lin(X,\R^{n_s}),
        \qquad
        K\in\Lin(\R^{n_s},\R^{n_a}),
        \qquad
        C\in\Lin(\R^{n_a},X),
        \label{eq:finite-channel-map-spaces}
\end{equation}
where $n_s$ and $n_a$ are the sensor and actuator channel counts.  The sensor map is typically coordinate readout on selected edges, while $C$ is coordinate injection or a matrix of device-sensitivity directions.  Full harmonic-sector control requires at least $b_1$ independent channels of each type; the minimal square count is $n_s=n_a=b_1$, subject to the rank conditions below.  The finite-channel objective is not, in general, $CKS=-\kappa P_H$, but the compression identity
\begin{equation}
        \Gamma_{\cH}(CKS)=H^\top CKSH=-\kappa I_{b_1}.
        \label{eq:finite-channel-compression-target}
\end{equation}

The elementary matrix equation underlying all finite-channel constructions is recorded separately.  This keeps the later feedback proofs focused on the Hodge geometry rather than on generalized-inverse calculations.  The result is standard in matrix analysis and generalized-inverse theory \cite{HornJohnsonMatrixAnalysis,BenIsraelGreville}; the detailed proof is included for clarity and later reuse.

\begin{lemma}[One-sided inverses and prescribed matrix compression]
\label{lem:matrix-compression}
Let $d\ge1$ and let
\begin{equation}
        G_C\in\Lin(\R^{n_a},\R^d),
        \qquad
        G_S\in\Lin(\R^d,\R^{n_s}).
        \label{eq:matrix-compression-operator-spaces}
\end{equation}
The following assertions are equivalent:
\begin{enumerate}[label=\textup{(\roman*)}]
\item
\begin{equation}
        \rank G_C=\rank G_S=d;
\end{equation}
\item $G_C$ admits a right inverse $G_C^R\in\Lin(\R^d,\R^{n_a})$ and $G_S$ admits a left inverse $G_S^L\in\Lin(\R^{n_s},\R^d)$, that is,
\begin{equation}
        G_CG_C^R=I_d,
        \qquad
        G_S^LG_S=I_d;
\end{equation}
\item for every $Q\in\Lin(\R^d)$ there exists $K\in\Lin(\R^{n_s},\R^{n_a})$ satisfying
\begin{equation}
        G_CKG_S=Q;
        \label{eq:general-matrix-compression}
\end{equation}
\item equation \eqref{eq:general-matrix-compression} is solvable for at least one invertible $Q\in\Lin(\R^d)$.
\end{enumerate}
If these conditions hold, then
\begin{equation}
        K_0=G_C^RQG_S^L
        \label{eq:general-compression-particular-solution}
\end{equation}
is a solution for every choice of one-sided inverses.  The complete solution set is
\begin{equation}
        K_0+
        \left\{
        Z\in\Lin(\R^{n_s},\R^{n_a}):G_CZG_S=0
        \right\},
        \label{eq:general-compression-affine-space}
\end{equation}
whose dimension is $n_an_s-d^2$.

Let $G_C^\dagger$ and $G_S^\dagger$ denote the Moore--Penrose inverses.  Then
\begin{equation}
        K_{\rm MP}=G_C^\dagger QG_S^\dagger
        \label{eq:general-compression-moore-penrose}
\end{equation}
is the unique solution of minimum Frobenius norm.  In the minimal square case $n_a=n_s=d$, the matrices $G_C$ and $G_S$ are invertible and
\begin{equation}
        K=G_C^{-1}QG_S^{-1}.
\end{equation}
\end{lemma}

\begin{proof}
A matrix $G_C\in\Lin(\R^{n_a},\R^d)$ admits a right inverse if and only if it has full row rank $d$.  Indeed, if $G_CG_C^R=I_d$, then $\rank G_C=d$; conversely, full row rank makes $G_CG_C^\top$ positive definite and
\begin{equation}
        G_C^R=G_C^\top(G_CG_C^\top)^{-1}
\end{equation}
is a right inverse.  Likewise, $G_S$ admits a left inverse if and only if it has full column rank $d$, in which case
\begin{equation}
        G_S^L=(G_S^\top G_S)^{-1}G_S^\top.
\end{equation}
This proves the equivalence of \textup{(i)} and \textup{(ii)}.

If the one-sided inverses exist, then
\begin{equation}
        G_C(G_C^RQG_S^L)G_S=Q,
\end{equation}
so \textup{(ii)} implies \textup{(iii)}, and \textup{(iii)} implies \textup{(iv)}.  Conversely, if $G_CKG_S=Q$ for an invertible $Q$, then
\begin{equation}
        d=\rank Q\le\rank G_C\le d,
        \qquad
        d=\rank Q\le\rank G_S\le d,
\end{equation}
which gives \textup{(i)}.

The linear map
\begin{equation}
        \mathcal T:\Lin(\R^{n_s},\R^{n_a})\longrightarrow\Lin(\R^d),
        \qquad
        \mathcal T(K)=G_CKG_S,
\end{equation}
is surjective under these conditions.  Rank--nullity therefore gives
\begin{equation}
        \dim\Ker\mathcal T=n_an_s-d^2,
\end{equation}
and the affine description follows by subtracting one particular solution.

Under the full-rank assumptions,
\begin{equation}
        G_C^\dagger=G_C^\top(G_CG_C^\top)^{-1},
        \qquad
        G_S^\dagger=(G_S^\top G_S)^{-1}G_S^\top.
\end{equation}
If $K$ is any solution and $Z=K-K_{\rm MP}$, then $G_CZG_S=0$.  With the Frobenius inner product and cyclicity of the trace,
\begin{align}
        \langle K_{\rm MP},Z\rangle_{\rm F}
        &=\operatorname{tr}\!\left(
        G_S(G_S^\top G_S)^{-1}Q^\top
        (G_CG_C^\top)^{-1}G_CZ
        \right)\notag\\
        &=\operatorname{tr}\!\left(
        (G_S^\top G_S)^{-1}Q^\top
        (G_CG_C^\top)^{-1}G_CZG_S
        \right)=0.
\end{align}
Hence
\begin{equation}
        \norm K_{\rm F}^2
        =\norm{K_{\rm MP}}_{\rm F}^2+\norm Z_{\rm F}^2,
\end{equation}
which proves uniqueness of the minimum-norm solution.  The square formula is immediate.
\end{proof}

\begin{remark}[Metrics on controller coordinates]
The minimum-norm assertion in
Lemma~\ref{lem:matrix-compression} uses the Euclidean inner products on
the sensor and actuator coordinate spaces and the induced Frobenius norm
on gains.  If positive-definite channel metrics are prescribed instead,
the corresponding weighted generalized inverse gives the natural
minimum-norm representative.
\end{remark}

\begin{theorem}[Exact harmonic compression by finite-channel feedback]
\label{thm:finite-rank-feedback}
Assume
\begin{equation}
        \rank(SH)=b_1,
        \qquad
        \rank(H^\top C)=b_1.
        \label{eq:finite-channel-rank-conditions}
\end{equation}
Then, for every $\kappa>0$, there exists $K\in\Lin(\R^{n_s},\R^{n_a})$ such that
\begin{equation}
        H^\top CKSH=-\kappa I_{b_1}.
        \label{eq:finite-channel-exact-compression}
\end{equation}
Put
\begin{equation}
        G_C:=H^\top C\in\Lin(\R^{n_a},\R^{b_1}),
        \qquad
        G_S:=SH\in\Lin(\R^{b_1},\R^{n_s}).
\end{equation}
If $G_C^R$ is any right inverse of $G_C$ and $G_S^L$ any left inverse of $G_S$, then
\begin{equation}
        K_0=-\kappa G_C^RG_S^L
        \label{eq:finite-channel-general-gain}
\end{equation}
is one solution.  All solutions have the form
\begin{equation}
        K=K_0+Z,
        \qquad
        G_CZG_S=0.
        \label{eq:finite-channel-affine-gain-space}
\end{equation}
This affine space has dimension $n_an_s-b_1^2$, and the Moore--Penrose choice $K_{\min}=-\kappa G_C^\dagger G_S^\dagger$ is the minimum-Frobenius-norm representative for Euclidean command metrics.  In the minimal square case $n_a=n_s=b_1$,
\begin{equation}
        K=-\kappa(H^\top C)^{-1}(SH)^{-1}.
\end{equation}
\end{theorem}

\begin{proof}
Apply Lemma~\ref{lem:matrix-compression} with
\begin{equation}
        d=b_1,
        \qquad
        G_C=H^\top C,
        \qquad
        G_S=SH,
        \qquad
        Q=-\kappa I_{b_1}.
\end{equation}
The conclusions follow directly from the lemma.
\end{proof}

\subsection*{Full edge-space dynamics under finite-channel feedback}

Theorem~\ref{thm:finite-rank-feedback} is an algebraic compression statement: it prescribes the harmonic-to-harmonic block of the finite-channel operator, but not its complete action on the edge space.  To derive the corresponding full dynamics, fix a constant cycle-free target $f_*\in\cC$ and put
\begin{equation}
        e=f-f_*.
\end{equation}
Since $P_Hf_*=0$, the ideal thermostat can equivalently be written as
\begin{equation}
        \dot f
        =
        -\alpha P_C(f-f_*)
        +L_{\rm id}(f-f_*),
        \qquad
        L_{\rm id}:=-\kappa P_H.
        \label{eq:ideal-target-centred-normal-form}
\end{equation}
The finite-channel counterpart is obtained by replacing $L_{\rm id}$ with
$L_{\rm ch}:=CKS$ on the same target error:
\begin{equation}
        \dot f
        =
        -\alpha P_C(f-f_*)
        +CKS(f-f_*).
        \label{eq:finite-channel-target-centred-law}
\end{equation}
Because $f_*$ is constant, the error therefore satisfies
\begin{equation}
        \dot e
        =
        -\alpha P_Ce+CKSe.
        \label{eq:finite-channel-full-error-ode}
\end{equation}

\begin{remark}[Target centering in finite-channel feedback]
\label{rem:finite-channel-target-centering}
The use of $f-f_*$ in \eqref{eq:finite-channel-target-centred-law} is not an arbitrary modification of the ideal normal form.  It preserves a property that the ideal harmonic projector has automatically: since $f_*\in\cC$,
\begin{equation}
        -\kappa P_Hf
        =
        -\kappa P_H(f-f_*).
\end{equation}
By contrast, the compression condition
\begin{equation}
        H^\top CKSH=-\kappa I_{b_1}
\end{equation}
does not imply $CKSf_*=0$, because it constrains only the harmonic-to-harmonic action of $CKS$.  Applying the finite-channel operator to the target error ensures that $f=f_*$ remains an equilibrium.  Equivalently, if $y=Sf$ and $y_*=Sf_*$ are the measured output and its target value, then the feedback signal is the output error
\begin{equation}
        CK(y-y_*)=CKS(f-f_*).
\end{equation}
If $CKS$ were instead applied to the raw flow $f$, the error equation would contain the generally nonzero constant forcing term $CKSf_*$.  Avoiding that term without target centering would require the additional condition $CKSf_*=0$, which is not part of Theorem~\ref{thm:finite-rank-feedback}.
\end{remark}

Write
\begin{equation}
        e=e_C+Hc,
        \qquad
        e_C=P_Ce\in\cC,
        \qquad
        c=H^\top e\in\R^{b_1}.
\end{equation}
Let $L=L_{\rm ch}=CKS\in\Lin(X)$ and define the blocks
\begin{equation}
\begin{aligned}
        L_{CC}&:=P_CL|_{\cC}\in\Lin(\cC),\\
        L_{CH}&:=P_CLH\in\Lin(\R^{b_1},\cC),\\
        L_{HC}&:=H^\top L|_{\cC}\in\Lin(\cC,\R^{b_1}),\\
        L_{HH}&:=H^\top LH\in\Lin(\R^{b_1}).
\end{aligned}
        \label{eq:finite-channel-block-definitions}
\end{equation}
If \eqref{eq:finite-channel-exact-compression} holds, then $L_{HH}=-\kappa I_{b_1}$ and \eqref{eq:finite-channel-full-error-ode} is equivalent to
\begin{equation}
        \dot e_C
        =
        (-\alpha I_{\cC}+L_{CC})e_C
        +L_{CH}c,
        \label{eq:finite-channel-cut-block}
\end{equation}
\begin{equation}
        \dot c
        =
        L_{HC}e_C
        -\kappa c.
        \label{eq:finite-channel-harmonic-block}
\end{equation}
In block form,
\begin{equation}
        \frac{d}{dt}
        \begin{pmatrix}e_C\\ c\end{pmatrix}
        =
        \begin{pmatrix}
        -\alpha I_{\cC}+L_{CC}&L_{CH}\\
        L_{HC}&-\kappa I_{b_1}
        \end{pmatrix}
        \begin{pmatrix}e_C\\ c\end{pmatrix}.
        \label{eq:finite-channel-block-matrix}
\end{equation}
The ideal thermostat is the special case $L_{CC}=L_{CH}=L_{HC}=0$.  A sparse finite-channel realization generally prescribes only the lower-right block exactly.  The remaining blocks measure cut-space spillover and cut-to-harmonic coupling.

\begin{remark}[Channel count and the full-edge limit]
\label{rem:channel-count-full-edge-limit}
The integers $n_s$ and $n_a$ are the numbers of sensor and actuator channels, respectively.  Since
\begin{equation}
        SH\in\Lin(\R^{b_1},\R^{n_s}),
        \qquad
        H^\top C\in\Lin(\R^{n_a},\R^{b_1}),
\end{equation}
the rank conditions in Theorem~\ref{thm:finite-rank-feedback} force $n_s\ge b_1$ and $n_a\ge b_1$.  Hence $n_a=n_s=b_1$ is the minimal square channel count for damping the full harmonic sector after compression.

At the opposite extreme, suppose all edge coordinates are both measured and actuated, so that $n_a=n_s=m$ and, in coordinate form, $S=I_X$ and $C=I_X$.  Then the finite-channel feedback operator is just $K$, and the compression condition becomes
\begin{equation}
        H^\top K H=-\kappa I_{b_1}.
\end{equation}
This condition still does not determine the ideal Hodge damping operator uniquely.  Indeed, every matrix of the form
\begin{equation}
        K=-\kappa P_H+Z,
        \qquad
        H^\top ZH=0,
\end{equation}
has the same harmonic compression.  The ideal operator is the particular choice
\begin{equation}
        K_{\rm id}=-\kappa P_H.
\end{equation}
Thus observing and actuating all edge coordinates makes the ideal Hodge operator available, but the compression condition alone does not select it.  With the Euclidean edge metric, $K_{\rm id}$ is the minimum-Frobenius-norm solution of the compression constraint.  This gives a useful hierarchy: the minimal square case $n_a=n_s=b_1$ gives the smallest channel count compatible with full harmonic compression, redundant cases leave design freedom, and the full-edge case contains the ideal projector as a canonical choice only after an additional selection principle is imposed.
\end{remark}

\begin{remark}[Colocated and noncolocated edge channels]
The theorem does not require sensors and actuators to be placed on the same edges.  If $R\subset E$ is a set of sensed edges, coordinate restriction is a map
\begin{equation}
        S_R\in\Lin(X,\R^{|R|}).
\end{equation}
If $A_0\subset E$ is a set of actuated edges, coordinate injection is a map
\begin{equation}
        C_{A_0}\in\Lin(\R^{|A_0|},X),
        \qquad
        \im C_{A_0}
        \subseteq
        X_{A_0}:=\operatorname{span}\{\delta_e:e\in A_0\}.
\end{equation}
The general rank conditions are
\begin{equation}
        \rank(S_RH)=b_1,
        \qquad
        \rank(H^\top C_{A_0})=b_1.
\end{equation}
The colocated Euclidean case is the special case $A_0=R$ and
\begin{equation}
        C_R=S_R^*,
\end{equation}
whose coordinate matrix is $S_R^\top$.  Then the two rank conditions reduce to $\rank(S_RH)=b_1$.  The ideal Hodge realization, by contrast, uses the global harmonic sensor $S=H^\top$ and the distributed actuator $C=H$, giving $CKS=-\kappa P_H$ with $K=-\kappa I_{b_1}$.
\end{remark}

\begin{remark}[Minimal harmonic rank versus local realization]
\label{rem:minimal-rank-local-realization}
The minimal channel count $b_1$ concerns the dimension of the harmonic compression, not the spatial support of an ideal realization.  The distinction is particularly transparent on the consistently oriented ring $C_N$, where every edge is cyclic and therefore
\begin{equation}
        \operatorname{esupp}(\cH)=E,
\end{equation}
although the harmonic space is only one-dimensional:
\begin{equation}
        b_1=1,
        \qquad
        \cH=\operatorname{span}\{h\},
        \qquad
        h=\frac{1}{\sqrt N}\mathbf 1_E,
        \qquad
        P_H=hh^\top.
\end{equation}
The ideal harmonic damping operator has the rank-one factorization
\begin{equation}
        -\kappa P_H
        =
        h(-\kappa)h^\top.
        \label{eq:ring-distributed-ideal-factorization}
\end{equation}
Thus one scalar harmonic measurement $h^\top f$ and one scalar actuator command suffice, but both channels are distributed over all ring edges.

If instead sensing and actuation are restricted to one physical edge $e$, put
\begin{equation}
        S_e=\delta_e^\top,
        \qquad
        C_e=\delta_e.
\end{equation}
Since $\delta_e^\top h=1/\sqrt N$, the rank conditions of Theorem~\ref{thm:finite-rank-feedback} hold, and the unique minimal square gain is
\begin{equation}
        K_e=-\kappa N.
\end{equation}
It gives exact harmonic compression,
\begin{equation}
        h^\top C_eK_eS_eh=-\kappa,
\end{equation}
but its full edge-space action is
\begin{equation}
        C_eK_eS_e=-\kappa N \delta_e\delta_e^\top\ne-\kappa P_H.
\end{equation}
In particular, a harmonic input produces the nonzero cut-space spillover
\begin{equation}
        P_CC_eK_eS_eh
        =
        -\kappa\sqrt N\,P_C\delta_e
        \ne0.
\end{equation}
Hence, if the channel budget is restricted to the minimal count $n_a=n_s=b_1=1$ and the channels must be colocated on one physical edge, exact local realization of the distributed projector is impossible on the ring.  Theorem~\ref{thm:finite-rank-feedback} still gives the desired harmonic compression for the linear edge-space normal form, but it does not preserve the rotationally symmetric nonlinear invariant trajectory.  The exactly solvable controlled ring trajectory in Subsection~\ref{subsec:exact-ring-trajectory} uses the distributed ideal factorization \eqref{eq:ring-distributed-ideal-factorization}.  Thus minimal harmonic rank, minimal local support, and exact realization of the ideal Hodge operator are distinct notions.
\end{remark}

\begin{remark}[Support obstruction for exact local realization]
Let $R,A_0\subseteq E$, let $S_R\in\Lin(X,\R^{|R|})$ be coordinate sensing on $R$, and let $C_{A_0}\in\Lin(\R^{|A_0|},X)$ be coordinate actuation on $A_0$.  Put
\begin{equation}
        L:=C_{A_0}KS_R\in\Lin(X).
\end{equation}
For canonical edge vectors $\delta_e,\delta_f$, one has
\begin{equation}
        \langle\delta_e,L\delta_f\rangle_X=0
        \qquad
        \text{whenever }e\notin A_0\text{ or }f\notin R.
        \label{eq:local-controller-matrix-support}
\end{equation}
Equivalently,
\begin{equation}
        \operatorname{msupp}(L)\subseteq A_0\times R.
        \label{eq:local-controller-support-inclusion}
\end{equation}
This is the precise meaning of saying that the coordinate matrix of $L$ has no nonzero output row associated with an unactuated edge and no nonzero input column associated with an unsensed edge.

Suppose now that the full operator identity
\begin{equation}
        C_{A_0}KS_R=-\kappa P_H
        \label{eq:exact-local-projector-identity}
\end{equation}
were to hold.  If $e\in E_{\rm cyc}=\operatorname{esupp}(P_H)$, then
\begin{equation}
        \langle\delta_e,P_H\delta_e\rangle_X
        =\norm{P_H\delta_e}^2>0
\end{equation}
by \eqref{eq:cyclic-edge-projector-test}.  If $e\notin A_0$, the left-hand side of \eqref{eq:exact-local-projector-identity} has zero $e$-th output coordinate for every input, contradicting the nonzero scalar above.  If $e\notin R$, it annihilates the input $\delta_e$, giving the same contradiction.  Therefore
\begin{equation}
        E_{\rm cyc}\subseteq A_0\cap R.
        \label{eq:cyclic-edges-must-be-sensed-actuated}
\end{equation}
Hence an exact realization of the global Hodge projector by local coordinate channels requires every cyclic edge to be both sensed and actuated.  Selecting only $b_1$ physical edges may realize the exact compression identity \eqref{eq:finite-channel-compression-target}, but generally not the full projector identity.  The obstruction disappears only when the channels themselves are distributed Hodge patterns rather than local edge coordinates.
\end{remark}

\begin{remark}[From existence to design]
The rank condition gives the minimal topological requirement for controlling the full harmonic sector.  It should not be read as saying that exactly $b_1$ physical channels are always the best engineering choice.  If $n_a>b_1$ or $n_s>b_1$, then the affine solution space \eqref{eq:finite-channel-affine-gain-space} contains additional degrees of freedom.  These may be used to minimize the non-ideal blocks
\begin{equation}
        P_CL_{\rm ch}P_C,
        \qquad
        P_CL_{\rm ch}H,
        \qquad
        H^\top L_{\rm ch}P_C,
\end{equation}
or to reduce gain size, improve conditioning, and increase robustness.  A finite-dimensional stability check is then performed on the full block matrix in \eqref{eq:finite-channel-block-matrix}.  If this matrix is Hurwitz, all cut-space disturbances caused by spillover are transient.  Thus topology gives the minimal rank requirement, while design may use redundant channels to approach the ideal Hodge thermostat as closely as the available sensor and actuator architecture permits.
\end{remark}

\begin{remark}[Minimal colocated channel sets and co-trees]
\label{rem:minimal-colocated-cotrees}
Let $G=(V,E)$ be connected, let
\[
        b_1=|E|-|V|+1,
\]
and let
\[
        H\in\Lin(\R^{b_1},X),
        \qquad
        \im H=\Ker B,
\]
be any coordinate map for the harmonic edge-flow space.  Consider a
minimal colocated coordinate architecture on an edge set
$R\subseteq E$ with $|R|=b_1$, with coordinate restriction $S_R$ and
coordinate injection $C_R=S_R^*$.  Put
\[
        T:=E\setminus R.
\]
Then the following statements are equivalent:
\begin{enumerate}[label=\textup{(\roman*)}]
\item
\[
        \rank(S_RH)=b_1;
\]
\item the selected edge coordinates determine every harmonic flow
uniquely;
\item no nonzero harmonic flow is supported entirely on $T$;
\item $T$ is a spanning tree of $G$.
\end{enumerate}

Indeed, the kernel of $S_RH$ represents precisely those flows in
$\Ker B$ that vanish on $R$ and are therefore supported on $T$.
Such a nonzero flow exists exactly when the subgraph $(V,T)$ contains
a cycle.  Since
\[
        |T|=|E|-b_1=|V|-1,
\]
acyclicity of $T$ is equivalent to $T$ being a spanning tree.  For
colocated channels, the actuator-rank condition is the same, because
\[
        H^\top C_R=(S_RH)^\top.
\]

Thus the minimal full-rank colocated channel sets are exactly the
co-trees, that is, the complements of spanning trees.  Denoting by
$N_{\mathrm{st}}(G)$ the number of spanning trees, the number of
admissible minimal architectures is therefore $N_{\mathrm{st}}(G)$
rather than the full combinatorial count
\[
        \binom{|E|}{b_1}.
\]
By the Matrix--Tree Theorem,
\[
        N_{\mathrm{st}}(G)=\det L_G^{(i)},
\]
where $L_G^{(i)}$ is any principal minor obtained by deleting one row
and column from the unweighted graph Laplacian.

The underlying co-tree characterization is classical in
cycle--cocycle theory and has also been used in PST-allocation studies
\cite{HuangPST2002}.  Its role here is to identify the minimal
colocated specialization of the finite-channel harmonic-compression
conditions.  This feasibility characterization is purely topological
and independent of the chosen harmonic basis or edge metric.  The
operating-point metric enters only at the subsequent design stage,
where admissible co-trees are compared through conditioning,
spillover, stability, or distance from the ideal thermostat.
\end{remark}

\begin{remark}[Reduced and operationally ranked harmonic subspaces]
\label{rem:reduced-operational-harmonic-subspaces}
One may control only a prescribed subspace
$\cH_k\subset\cH$ of dimension $k\le b_1$.  In a transmission-grid
interpretation, such a subspace may be selected from historically
observed and scenario-generated harmonic flow deviations, weighted by
thermal margins, losses, contingency severity, or other operational
risk measures.  If
\[
        H_k\in\Lin(\R^k,X),
        \qquad
        H_k^\top H_k=I_k,
        \qquad
        \im H_k=\cH_k,
\]
the finite-channel design objective becomes
\begin{equation}
        H_k^\top CKSH_k=-\kappa I_k,
\end{equation}
with the corresponding rank conditions
\begin{equation}
        \rank(SH_k)=k,
        \qquad
        \rank(H_k^\top C)=k.
\end{equation}
The target dimension must therefore be adapted to the effective sensor
and actuator ranks, rather than merely to the number of physical
devices.  An asymmetric architecture with $n_s>n_a\ge k$ may be
advantageous when sensing channels are less costly: additional sensors
can improve conditioning, noise robustness, and redundancy, although
they cannot compensate for insufficient actuator rank.

Strategic PST or FACTS placement and modal channel selection are
established engineering topics \cite{HuangPST2002,Sadikovic2006}.  The
contribution here is not the placement problem per se, but its
basis-independent organization as the successive selection of an
operationally relevant harmonic target subspace, verification of
sensor and actuator sufficiency on that subspace, and design of a
compressed damping law together with spillover and full-system
stability diagnostics.
\end{remark}

\begin{remark}[The ideal thermostat as a design benchmark]
\label{rem:ideal-thermostat-design-benchmark}
The distributed ideal operator
\[
        L_{\rm id}=-\kappa P_H
\]
provides a canonical benchmark even when its global realization is
physically or economically unavailable.  Let
$\cH_k\subseteq\cH$ be a selected harmonic target subspace, let
$P_{\cH_k}$ be its orthogonal projector, and put
$L_{\rm ch}=CKS$.  Then
\begin{equation}
        L_{\rm ch}-L_{\rm id}
        =
        \bigl(L_{\rm ch}+\kappa P_{\cH_k}\bigr)
        +
        \kappa\bigl(P_H-P_{\cH_k}\bigr).
        \label{eq:finite-channel-ideal-benchmark-gap}
\end{equation}
The first term is the realization gap relative to the selected partial
ideal: it contains compression error and non-ideal Hodge blocks.  The
second is the truncation gap created by leaving the complementary
harmonic directions uncontrolled.

Nested target subspaces and expanding sensor--actuator architectures
may therefore be compared against the same ideal reference.  Such
improvement requires sufficient sensor and actuator rank and should be
assessed in an application-dependent energy- or risk-weighted metric.
Indeed, if $\cH_k\subsetneq\cH$, then
$\norm{P_H-P_{\cH_k}}_{\rm op}=1$, so the uniform operator norm does
not express gradual coverage of increasingly important modes.  The
ideal projector is not presumed affordable; it supplies a fixed
reference against which omission, spillover, stability margin,
robustness, and cost can be quantified.
\end{remark}

\begin{remark}[Higher-order finite-channel harmonic feedback]
\label{rem:higher-order-finite-rank}
The compression principle is not specific to edge-flow $1$-cochains.  Let $\mathcal K$ be a finite simplicial or cellular complex, let $C^k(\mathcal K)$ be the space of real $k$-cochains, and let
\begin{equation}
\cH^k=\Ker\Delta_k
\end{equation}
be the harmonic $k$-cochain space, with
\begin{equation}
\dim \cH^k=b_k.
\end{equation}
Choose an orthonormal harmonic basis
\begin{equation}
H_k\in\Lin(\R^{b_k},\R^{N_k}),
        \qquad
N_k=\dim C^k(\mathcal K).
\end{equation}
Let
\begin{equation}
\mathsf S_k:C^k(\mathcal K)\to\R^{n_s}
\end{equation}
be a finite family of $k$-cochain sensors, and let
\begin{equation}
\mathsf C_k:\R^{n_a}\to C^k(\mathcal K)
\end{equation}
be a finite family of $k$-cochain actuators.  The full harmonic $k$-sector is observable and controllable through these channels if
\begin{equation}
\rank(\mathsf S_kH_k)=b_k,
        \qquad
\rank(H_k^*\mathsf C_k)=b_k.
\end{equation}
Under these rank conditions, there exists $\mathsf K_k\in\Lin(\R^{n_s},\R^{n_a})$ such that
\begin{equation}
        H_k^*\mathsf C_k\mathsf K_k\mathsf S_kH_k=-\kappa I_{b_k}.
\end{equation}
Thus the $k$-th Betti number is again a controller-rank count: to damp the full $k$-th cohomological sector, the sensors must separate all harmonic $k$-modes and the actuators must excite all harmonic $k$-directions.  If only a prescribed subspace is to be controlled, the same construction applies with $H_k$ replaced by a basis of that subspace.
\end{remark}

\section{A hand-computable linear theta-network edge-flow model}
This section records a normalized linear edge-flow example before electrical state dynamics are introduced.  Its role is algebraic: it displays, by hand, how useful through-flow and internal circulation are separated by the Hodge decomposition, and how an ideal or finite-channel controller acts on those coordinates.

\begin{example}[Theta network: useful transfer and internal circulation]
\label{ex:theta-edgeflow}
The theta graph is the smallest algebraic model in which useful transfer and internal circulation separate visibly.  It may be read as three parallel transfer corridors between a source and a load.  With all edges reference-oriented from the source to the load, the scalar $i_1+i_2+i_3$ is the net source--load transfer.  The target $i_*=(1,1,1)^\top$ used below is therefore the symmetric no-preference transfer of total size $3$, shared equally by the three lines.  The zero-sum plane consists of internal redistributions that do not change the source--load balance.  These are precisely the loop-flow components that the thermostat is meant to remove.

The theta graph consists of two vertices joined by three parallel edges.  Its edge space is $\R^3$.  With all edges reference-oriented from the left vertex to the right vertex, a flow is
\begin{equation}
i=(i_1,i_2,i_3)^\top.
\end{equation}
The incidence matrix is
\begin{equation}
        B=
        \begin{pmatrix}
        1&1&1\\
        -1&-1&-1
        \end{pmatrix}.
\end{equation}
Thus $\rank B=1$ and the harmonic space is the right nullspace
\begin{equation}
\cH=\Ker B=\{i\in\R^3:i_1+i_2+i_3=0\}.
\end{equation}
The singular value decomposition gives a canonical rank-revealing linear algebra decomposition of this incidence matrix.  One convenient choice is
\begin{equation}
        B=U\Sigma_BV^\top,
\end{equation}
where
\begin{equation}
U=\frac1{\sqrt2}
\begin{pmatrix}
1&1\\
-1&1
\end{pmatrix},
\qquad
\Sigma_B=
\begin{pmatrix}
\sqrt6&0&0\\
0&0&0
\end{pmatrix},
\end{equation}
and
\begin{equation}
V^\top=
\begin{pmatrix}
\frac1{\sqrt3}&\frac1{\sqrt3}&\frac1{\sqrt3}\\[1mm]
\frac1{\sqrt2}&-\frac1{\sqrt2}&0\\[1mm]
\frac1{\sqrt6}&\frac1{\sqrt6}&-\frac2{\sqrt6}
\end{pmatrix}.
\end{equation}
The first row of $V^\top$ is the normalized through-flow coordinate functional
\begin{equation}
g^\top=\frac1{\sqrt3}(1,1,1).
\end{equation}
The corresponding column vector
\begin{equation}
g=\frac1{\sqrt3}(1,1,1)^\top
\end{equation}
is the equal-sharing through-flow direction in the edge-flow space $\R^3$, whose three coordinate axes correspond to the three reference-oriented edges.  Thus this is not a spatial direction in the drawing of the graph.  Applying the direction $g$ means placing equal oriented flow on all three edges.  The last two rows of $V^\top$ give harmonic coordinate functionals.  Equivalently,
\begin{equation}
        V^\top=
        \begin{pmatrix}
        g^\top\\ H^\top
        \end{pmatrix},
        \qquad
        H=
        \begin{pmatrix}
        \frac1{\sqrt2}&\frac1{\sqrt6}\\[1mm]
        -\frac1{\sqrt2}&\frac1{\sqrt6}\\[1mm]
        0&-\frac2{\sqrt6}
        \end{pmatrix}.
\end{equation}
Thus
\begin{equation}
        H^\top H=I_2,
        \qquad
        \im H=\Ker B.
\end{equation}
The two column vectors of $H$ are
\begin{equation}
h_1=\frac{1}{\sqrt2}(1,-1,0)^\top,
        \qquad
        h_2=\frac{1}{\sqrt6}(1,1,-2)^\top,
\end{equation}
and therefore
\begin{equation}
\R^3=\operatorname{span}\{g\}\oplus\operatorname{span}\{h_1,h_2\}.
\end{equation}
Here the word ``direction'' always refers to a direction in edge-flow coordinate space.  In the graph picture, $h_1$ means equal and opposite flows on edges $1$ and $2$, with no flow on edge $3$.  Likewise, $h_2$ means equal forward flows on edges $1$ and $2$, compensated by twice the backward flow on edge $3$.  Both patterns have zero coordinate sum and hence no net source--load transfer.
The corresponding projections are explicit:
\begin{equation}
        P_g=gg^\top=\frac13
        \begin{pmatrix}
        1&1&1\\1&1&1\\1&1&1
        \end{pmatrix},
        \qquad
        P_H=HH^\top=I-P_g
        =\frac13
        \begin{pmatrix}
        2&-1&-1\\-1&2&-1\\-1&-1&2
        \end{pmatrix}.
\end{equation}
The particular SVD basis in $\Ker B$ is not unique, because the zero singular value has multiplicity two.  Replacing $H$ by $HQ$, $Q\in O(2)$, changes the harmonic coordinates but leaves the subspace $\Ker B$, the projection $P_H=HH^\top$, and the rank tests below unchanged.

\paragraph{Reading off sensors and actuators.}
The SVD makes the finite-channel construction completely explicit.  The ideal full-information harmonic sensor and actuator are
\begin{equation}
        S_{\rm id}=H^\top,
        \qquad
        C_{\rm id}=H,
        \qquad
        K_{\rm id}=-\kappa I_2.
\end{equation}
Then
\begin{equation}
        S_{\rm id}H=I_2,
        \qquad
        H^\top C_{\rm id}=I_2,
\end{equation}
so the finite-channel rank conditions are automatically satisfied and the induced harmonic coordinate equation is $\dot c=-\kappa c$.

A sparse physical realization can be obtained by measuring and actuating only two of the three edges.  For instance, choose
\begin{equation}
        S_{12}=R_{12}=
        \begin{pmatrix}
        1&0&0\\0&1&0
        \end{pmatrix},
        \qquad
        C_{12}=E_{12}=
        \begin{pmatrix}
        1&0\\0&1\\0&0
        \end{pmatrix}.
\end{equation}
Then
\begin{equation}
        S_{12}H=
        \begin{pmatrix}
        \frac1{\sqrt2}&\frac1{\sqrt6}\\[1mm]
        -\frac1{\sqrt2}&\frac1{\sqrt6}
        \end{pmatrix},
        \qquad
        H^\top C_{12}=
        \begin{pmatrix}
        \frac1{\sqrt2}&-\frac1{\sqrt2}\\[1mm]
        \frac1{\sqrt6}&\frac1{\sqrt6}
        \end{pmatrix}.
\end{equation}
Both matrices have determinant $1/\sqrt3$ in absolute value.  Hence
\begin{equation}
        \rank(S_{12}H)=2,
        \qquad
        \rank(H^\top C_{12})=2.
\end{equation}
The gain prescribed by Theorem~\ref{thm:finite-rank-feedback} is
\begin{equation}
        K_{12}
        =-
        \kappa
        (H^\top C_{12})^{-1}(S_{12}H)^{-1}
        =-
        \kappa
        \begin{pmatrix}
        2&1\\1&2
        \end{pmatrix}.
\end{equation}
Consequently,
\begin{equation}
        H^\top C_{12}K_{12}S_{12}H=-\kappa I_2.
\end{equation}
Thus the two measured edge currents separate the two harmonic coordinates, and the two actuated edge channels excite the two harmonic directions.  By symmetry, any two of the three edges give an equivalent minimal sparse realization; one edge alone would have rank one and cannot control both loop-flow coordinates.

The topological thermostat is
\begin{equation}
\dot i=-\alpha P_g(i-i_*)-\kappa P_Hi.
\end{equation}
It preserves the through-flow target and damps the two loop-current coordinates.  This is the three-dimensional version of the general theorem.

The unprojected comparison dynamics are exactly analogous to the IEEE 14-bus benchmark.  With $E_1=\operatorname{diag}(1,0,0)$, local damping is
\begin{equation}
\dot i=-\alpha P_g(i-i_*)-\kappa E_1i,
\end{equation}
global damping is
\begin{equation}
\dot i=-\alpha P_g(i-i_*)-\kappa i,
\end{equation}
and the topological thermostat is
\begin{equation}
\dot i=-\alpha P_g(i-i_*)-\kappa P_Hi.
\end{equation}
The point of this example is that all projections are explicit.  The local edge damper is not aligned with the decomposition.  The global damper is aligned with neither the objective nor the useful transfer target.  The topological thermostat is exactly aligned with the cycle-space component.

\paragraph{Closed-form current evolution.}
The theta comparison is deliberately an elementary, exactly solvable linear model.  Put
\begin{equation}
        h_0=i(0)-i_*\in\cH
\end{equation}
and define
\begin{equation}
        D_{\rm none}=0,
        \qquad
        D_{\rm loc}=E_1,
        \qquad
        D_{\rm glob}=I,
        \qquad
        D_{\rm top}=P_H.
\end{equation}
For each comparison law the current satisfies the affine linear equation
\begin{equation}
        \dot i=A_j i+b,
        \qquad
        A_j=-\alpha P_g-\kappa D_j,
        \qquad
        b=\alpha i_*.
\end{equation}
Consequently, all four trajectories are available without time discretization:
\begin{equation}
        \begin{pmatrix}i_j(t)\\1\end{pmatrix}
        =
        \exp\!\left[
        t
        \begin{pmatrix}
        A_j&b\\
        0&0
        \end{pmatrix}
        \right]
        \begin{pmatrix}i(0)\\1\end{pmatrix}.
\end{equation}
For three of the four laws this formula reduces immediately to
\begin{align}
        i_{\rm none}(t)
        &=i_*+h_0,\\
        i_{\rm top}(t)
        &=i_*+e^{-\kappa t}h_0,\\
        i_{\rm glob}(t)
        &=
        \left(
        \frac{\alpha}{\alpha+\kappa}
        +\frac{\kappa}{\alpha+\kappa}
        e^{-(\alpha+\kappa)t}
        \right)i_*
        +e^{-\kappa t}h_0.
\end{align}
The local-damping trajectory is given by the same matrix-exponential formula.  It preserves the uncontrolled difference
\begin{equation}
        i_2(t)-i_3(t)=d_0:=i_2(0)-i_3(0)
\end{equation}
and converges to
\begin{equation}
        i_{\rm loc,\infty}
        =
        \left(
        0,
        \frac{3+d_0}{2},
        \frac{3-d_0}{2}
        \right)^\top.
\end{equation}
Thus no loop control leaves the initial circulation unchanged, local damping leaves an uncontrolled harmonic imbalance, and global damping equalizes the currents only after attenuating the prescribed transfer.  The topological thermostat alone gives
\begin{equation}
        P_gi_{\rm top}(t)=i_*,
        \qquad
        P_Hi_{\rm top}(t)=e^{-\kappa t}h_0,
\end{equation}
so that the target current is preserved while every cyclic component is removed.  The plots below are evaluations of these exact affine flows.  Their purpose is to display the control objective in the smallest nontrivial network, before introducing the linearized and nonlinear electrical state equations.
\end{example}

\begin{remark}[Controller wiring in the theta graph]
In the colocated sparse realization above, the gain matrix $K_{12}$ can be read as a weighted wiring graph between the two sensor channels and the two actuator channels.  The entry $(K_{12})_{\ell j}$ is the gain from sensor $j$ to actuator $\ell$.  Thus the diagonal entries describe self-feedback on the colocated edge channels, while the off-diagonal entries describe cross-coupling between the two measured lines.  In explicit signal form,
\begin{equation}
        u_1=-\kappa(2i_1+i_2),
        \qquad
        u_2=-\kappa(i_1+2i_2).
\end{equation}
Hence each actuator uses both measured edge currents.

This wiring is not ad hoc.  For a colocated selected-edge set $R$ with coordinate sensor $S_R$ and actuator $C_R=S_R^\top$, the minimal square gain is
\begin{equation}
        K_R=-\kappa\bigl(S_RP_HS_R^\top\bigr)^{-1},
\end{equation}
whenever $S_RH$ is invertible.  The matrix $S_RP_HS_R^\top$ is the Gram matrix of the harmonic edge fingerprints $P_H\delta_e$, $e\in R$.  Thus the controller wiring is determined by the harmonic geometry of the selected edge channels.

For the theta network and $R=\{1,2\}$,
\begin{equation}
        S_{12}P_HS_{12}^\top
        =
        \frac13
        \begin{pmatrix}
        2&-1\\
        -1&2
        \end{pmatrix},
\end{equation}
and therefore
\begin{equation}
        K_{12}
        =
        -\kappa
        \begin{pmatrix}
        2&1\\
        1&2
        \end{pmatrix}.
\end{equation}
Equivalently, if edge $3$ is used as the tree edge, the selected edges define the two fundamental cycle vectors
\begin{equation}
        z_1=(1,0,-1)^\top,
        \qquad
        z_2=(0,1,-1)^\top,
\end{equation}
whose Gram matrix is
\begin{equation}
        Z^\top Z=
        \begin{pmatrix}
        2&1\\
        1&2
        \end{pmatrix}.
\end{equation}
The off-diagonal coupling in $K_{12}$ is therefore the algebraic trace of the fact that the two selected fundamental cycles share the unactuated edge.  Similar Gram-matrix interpretations are available for larger colocated realizations, but here the observation is used only to clarify the elementary theta example.
\end{remark}

The closed-form comparison in Figure~\ref{fig:theta-comparison} uses the normalized edge-flow model of Example~\ref{ex:theta-edgeflow}.  The theta-network reproducibility script evaluates the explicit formulas above and the augmented matrix exponential for local damping; it does not use a time-stepping ODE solver.

\begin{figure}[H]
\centering
\begin{subfigure}{0.48\textwidth}
\centering
\includegraphics[width=\textwidth]{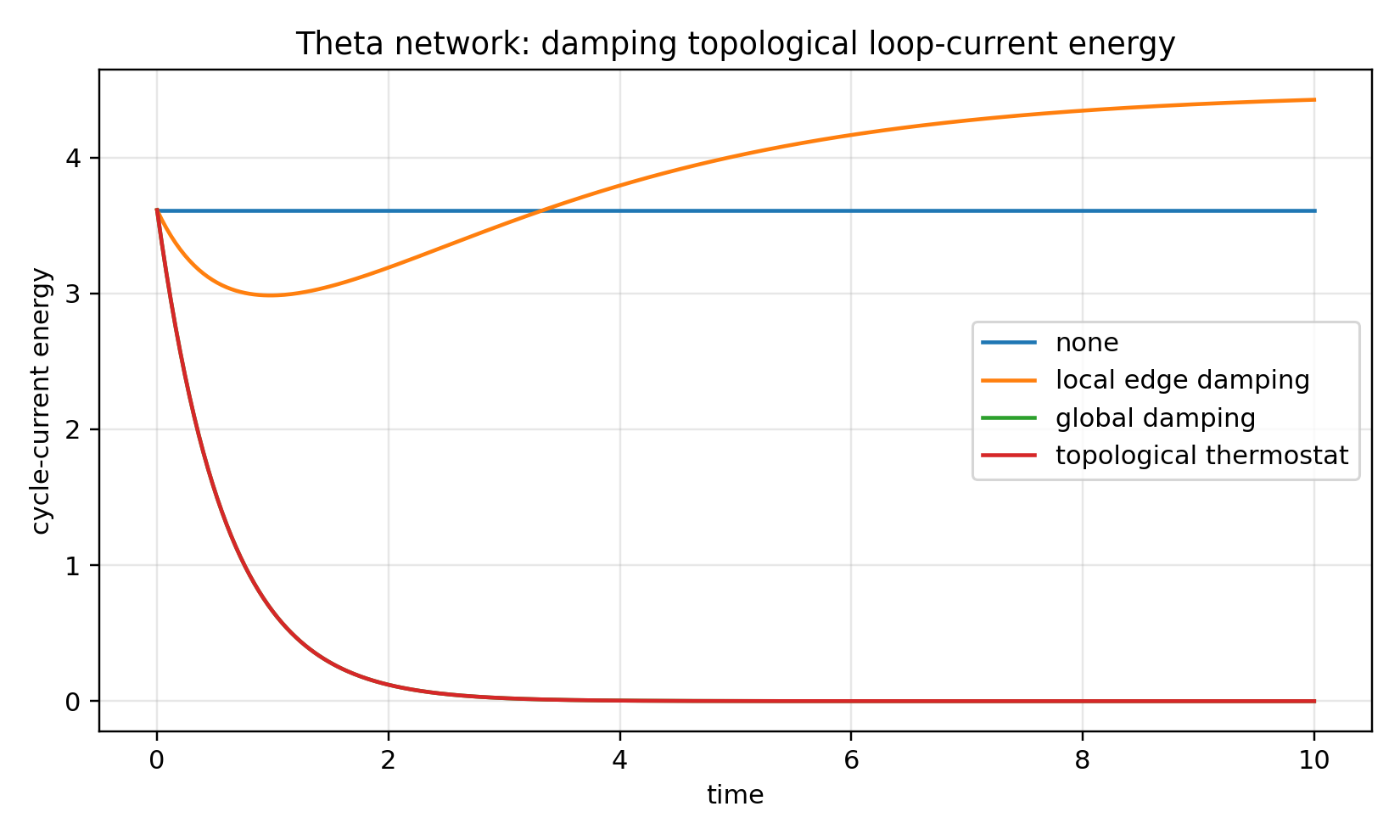}
\caption{Cycle energy.}
\end{subfigure}\hfill
\begin{subfigure}{0.48\textwidth}
\centering
\includegraphics[width=\textwidth]{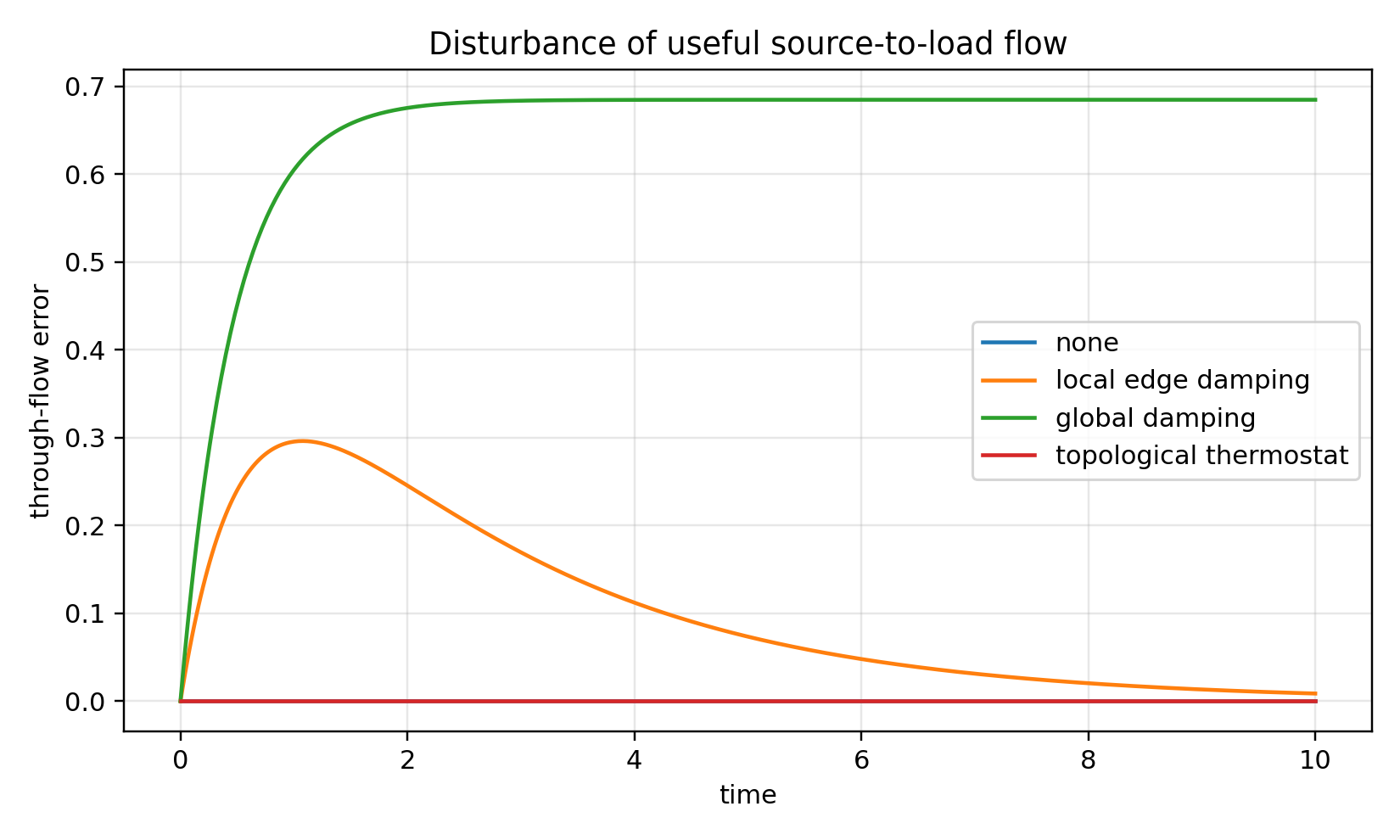}
\caption{Through-flow error.}
\end{subfigure}
\caption{Closed-form theta-network comparison.  Global damping and the topological thermostat have the same harmonic decay, while no loop control and the topological thermostat both preserve the initially correct cut component.  Only the topological thermostat combines full harmonic damping with the prescribed transfer target.}
\label{fig:theta-comparison}
\end{figure}

\begin{figure}[H]
\centering
\includegraphics[width=0.75\textwidth]{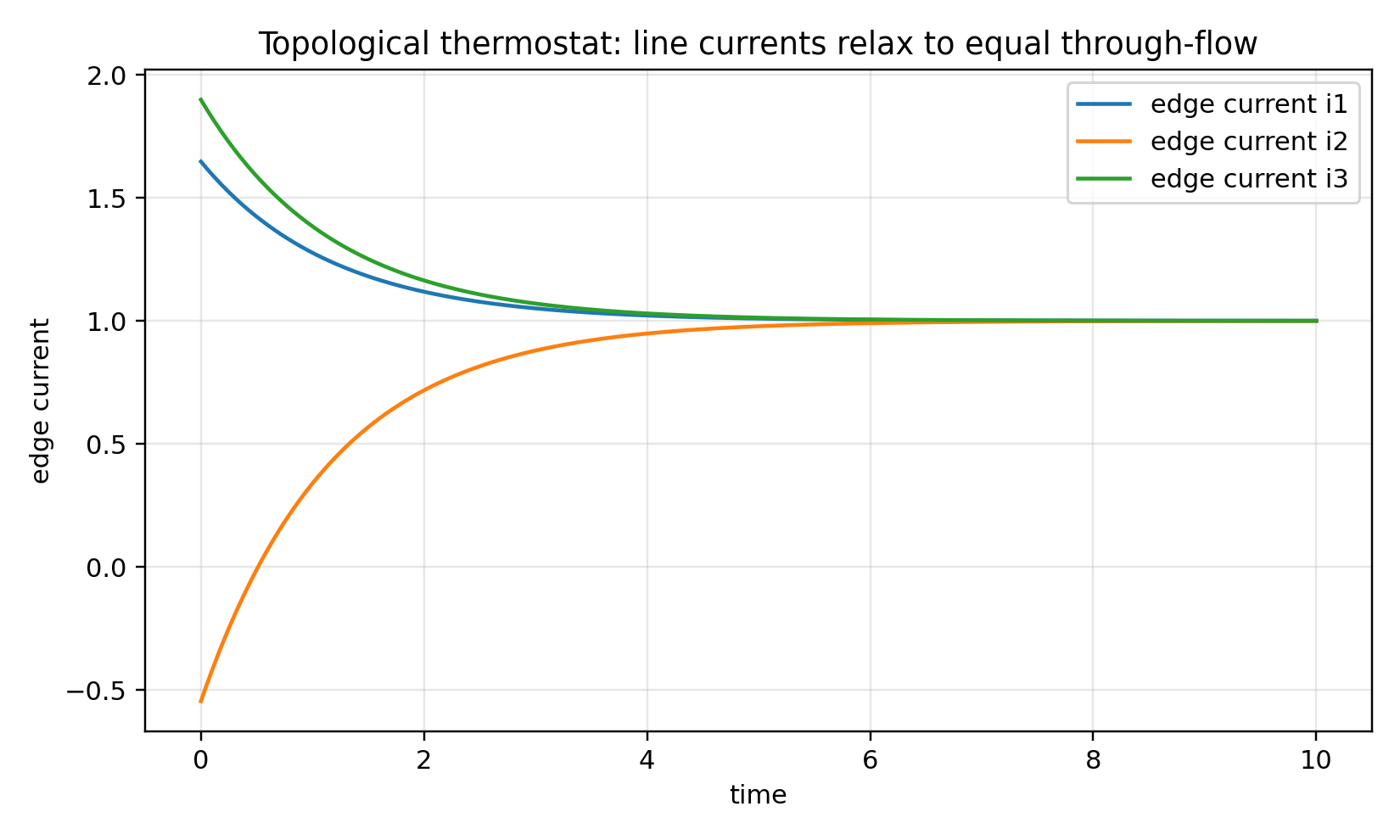}
\caption{Closed-form theta-network trajectory: the cut component remains at the prescribed target while the harmonic component decays exponentially.}
\label{fig:theta-edges}
\end{figure}

For the plots we use the nondimensional gains $\alpha=1.3$ and $\kappa=0.85$, so the ideal law is
\begin{equation}
\dot i=
\begin{pmatrix}
-1.00&-0.15&-0.15\\
-0.15&-1.00&-0.15\\
-0.15&-0.15&-1.00
\end{pmatrix}i+
\begin{pmatrix}
1.3\\
1.3\\
1.3
\end{pmatrix}.
\end{equation}
This normalized system is not calibrated by physical line parameters; a calibrated theta model would replace the Euclidean projection by the corresponding weighted Hodge projection.

\section{Linear edge-current dynamics and circulating-current damping in meshed inverter microgrids}
The abstract topological thermostat of Theorem~\ref{thm:topological-thermostat} and the finite-channel compression theorem, Theorem~\ref{thm:finite-rank-feedback}, are closest to a physical model when the controlled quantity is already an edge variable.  Meshed inverter microgrids provide such a model problem.  In an inverter-rich islanded microgrid, power-electronic converters supply local loads through a network of lines.  A recognized control objective is to share power and regulate voltage while avoiding internal circulating currents between inverters or around loops.  Virtual-impedance methods address this objective by modifying the effective voltage--current relation of an inverter through feedback rather than by inserting a physical resistor or inductor \cite{VirtualImpedanceZhang2023,YangVirtualImpedance2025,KhanVirtualImpedance2024}.

The point of the following discussion is not to model switching electronics.  Instead, we start from a minimal line-current normal form and show how the Hodge decomposition turns circulating-current suppression into a precise feedback objective.  The examples are then ordered by increasing modelling content.  The triangular microgrid displays the single harmonic current coordinate.  The two-triangle graph displays the finite-channel sensor--actuator construction.  The final weighted parallel-inverter example relates the same mechanism to the kind of unequal-capacity circulating-current problem studied in virtual-impedance and droop-control work.

\subsection*{A line-current normal form for virtual-impedance control}
For inverter microgrids, it is more natural to work with line currents than with line powers.  Let
\begin{equation}
        i\in\R^m
\end{equation}
be the vector of line currents, measured for instance in amperes or in per-unit current.  A minimal linear line-current model is
\begin{equation}
        L\dot i=-Ri+B^\top v+u.
        \label{eq:microgrid-current-model}
\end{equation}
This equation is the graph form of Kirchhoff's voltage law for reference-oriented $RL$ lines with an added controllable line-voltage input.  For one reference-oriented line $e$, with current $i_e$, resistance $R_e>0$, inductance $L_e>0$, and control input $u_e$, the scalar relation has the form
\[
        L_e\dot i_e+R_e i_e=\Delta v_e+u_e.
\]
With the incidence convention used in this paper, tail entries are $+1$ and head entries are $-1$, hence
\[
        (B^\top v)_e=v_{\mathrm{tail}(e)}-v_{\mathrm{head}(e)}.
\]
Thus $B^\top v$ is the voltage drop in the reference direction, and stacking the scalar line equations over all lines gives \eqref{eq:microgrid-current-model}.  If one instead uses the opposite convention for voltage drop, the signs of the edge-current coordinate and the control input must be changed consistently.  The Hodge-theoretic content is independent of this bookkeeping choice.

Here $v\in\R^n$ is a nodal voltage vector, $R,L\in\Lin(X)$ are positive diagonal resistance and inductance operators, and $u\in X$ is an idealized line-voltage or virtual-impedance control input.  In physical units, $Ri$, $L\dot i$, $B^\top v$, and $u$ are all voltage-like edge quantities.  The model is deliberately minimal: it retains the line-current state and the incidence geometry, but suppresses switching dynamics, $dq$-coordinate inverter controls, shunt capacitances, saturation, delays, and device limits.  The contribution here is therefore not a new inverter model.  It is the use of this standard current-flow normal form as the carrier of a Hodge-selective feedback law.

Let $H\in\Lin(\R^{b_1},X)$ have orthonormal columns spanning $\Ker B$.  In the normalized case
\begin{equation}
        L=I,
        \qquad
        R=r_0I,
        \qquad r_0>0,
\end{equation}
model \eqref{eq:microgrid-current-model} becomes
\begin{equation}
        \dot i=-r_0 i+B^\top v+u.
        \label{eq:normalized-microgrid-current-model}
\end{equation}
Let $i_*\in\cC=\im B^\top$ be a prescribed load-serving current and put
\begin{equation}
        e=i-i_*=e_C+Hc,
        \qquad
        e_C=P_Ce,
        \qquad
        c=H^\top e=H^\top i.
        \label{eq:microgrid-current-error-splitting}
\end{equation}
Since $H^\top B^\top=(BH)^\top=0$ and $P_Hi_*=0$, projection of \eqref{eq:normalized-microgrid-current-model} gives the exact component equations
\begin{equation}
        \dot e_C
        =
        -r_0e_C+B^\top v-r_0i_*+P_Cu,
        \label{eq:microgrid-cut-current-equation}
\end{equation}
\begin{equation}
        \dot c=-r_0c+H^\top u.
        \label{eq:harmonic-current-reduction}
\end{equation}
Thus nodal voltage differences act only in the cut space and do not directly force the harmonic circulating-current coordinates.

At the reduced control level, the voltage, droop, or secondary-control layer may be represented by the cut-regulation law
\begin{equation}
        -r_0e_C+B^\top v-r_0i_*=-\alpha e_C,
        \qquad \alpha>0.
        \label{eq:microgrid-cut-regulation-closure}
\end{equation}
With this closure, \eqref{eq:normalized-microgrid-current-model} takes the Hodge-coordinate normal form
\begin{equation}
        \dot e
        =
        -\alpha P_Ce-r_0P_He+u.
        \label{eq:microgrid-current-error-normal-form}
\end{equation}
The ideal topological virtual-impedance law is
\begin{equation}
        u=-\kappa P_Hi=-\kappa P_He,
        \qquad \kappa>0.
        \label{eq:topological-virtual-impedance}
\end{equation}
Consequently,
\begin{equation}
        \dot e_C=-\alpha e_C,
        \qquad
        \dot c=-(r_0+\kappa)c.
        \label{eq:ideal-microgrid-hodge-coordinates}
\end{equation}
This is the current-flow specialization of Theorem~\ref{thm:topological-thermostat}; the passive line resistance contributes the additional harmonic decay rate $r_0$, while the topological feedback contributes $\kappa$.

A dynamic virtual-impedance realization introduces an actuator state $w$:
\begin{equation}
        u=w,
        \qquad
        \tau\dot w=-w-\kappa P_Hi,
        \qquad \tau>0.
\end{equation}
With
\begin{equation}
        c_u=H^\top w,
\end{equation}
the harmonic subsystem is
\begin{equation}
        \frac{d}{dt}
        \begin{pmatrix}
        c\\ c_u
        \end{pmatrix}
        =
        \begin{pmatrix}
        -r_0I_{b_1}&I_{b_1}\\
        -\kappa/\tau\,I_{b_1}&-1/\tau\,I_{b_1}
        \end{pmatrix}
        \begin{pmatrix}
        c\\ c_u
        \end{pmatrix}.
        \label{eq:microgrid-harmonic-dynamic-subsystem}
\end{equation}
This formula is the common current-flow subsystem underlying the triangle and two-triangle examples below.

The clean Euclidean splitting in \eqref{eq:microgrid-cut-current-equation}--\eqref{eq:harmonic-current-reduction} uses $L=I$ and $R=r_0I$.  For heterogeneous diagonal inductances and resistances, the operator $L^{-1}R$ need not preserve $\cC$ and $\cH$; an appropriate weighted splitting or explicit cut--harmonic coupling terms are then required.

The finite-channel realization is now an application of Theorem~\ref{thm:finite-rank-feedback}, rather than a separate algebraic construction.  Retain the notation of Section~\ref{sec:finite-edge-channel-realization}: $S$ is interpreted as a line-current error sensor, $C$ as a collection of virtual-impedance or controlled line-voltage directions, and $K$ as the sensor-to-actuator gain.  Put
\begin{equation}
        u=CKSe,
        \qquad
        \mathcal L=CKS\in\Lin(X),
\end{equation}
and denote by $\mathcal L_{CC}$, $\mathcal L_{CH}$, $\mathcal L_{HC}$, and $\mathcal L_{HH}$ the blocks defined as in \eqref{eq:finite-channel-block-definitions}.  Then \eqref{eq:microgrid-current-error-normal-form} is equivalent to
\begin{equation}
        \frac{d}{dt}
        \begin{pmatrix}e_C\\ c\end{pmatrix}
        =
        \begin{pmatrix}
        -\alpha I_{\cC}+\mathcal L_{CC}&\mathcal L_{CH}\\
        \mathcal L_{HC}&-r_0I_{b_1}+\mathcal L_{HH}
        \end{pmatrix}
        \begin{pmatrix}e_C\\ c\end{pmatrix}.
        \label{eq:microgrid-finite-channel-block-system}
\end{equation}
If
\begin{equation}
        \rank(SH)=b_1,
        \qquad
        \rank(H^\top C)=b_1,
\end{equation}
Theorem~\ref{thm:finite-rank-feedback} gives a gain for which
\begin{equation}
        \mathcal L_{HH}=H^\top CKSH=-\kappa I_{b_1}.
\end{equation}
The complete harmonic equation is therefore
\begin{equation}
        \dot c
        =
        \mathcal L_{HC}e_C-(r_0+\kappa)c.
        \label{eq:microgrid-finite-channel-harmonic-equation}
\end{equation}
Thus the actuator contribution has the prescribed harmonic compression, while autonomous harmonic damping along the full trajectory additionally requires the cut-to-harmonic term $\mathcal L_{HC}e_C$ to vanish.  The Betti number counts the independent circulating-current modes that must be seen and influenced; the selected sensors, actuator directions, and gain determine the remaining spillover blocks.

\subsection*{A triangular microgrid: one circulating-current mode}
The smallest genuinely meshed microgrid is a triangle
\begin{equation}
        G=C_3.
\end{equation}
Choose the reference orientation of the edges as
\begin{equation}
        1\to 2,
        \qquad
        2\to 3,
        \qquad
        3\to 1.
\end{equation}
For the line-current vector
\begin{equation}
        i=(i_1,i_2,i_3)^\top,
\end{equation}
the incidence matrix is
\begin{equation}
B=
\begin{pmatrix}
1&0&-1\\
-1&1&0\\
0&-1&1
\end{pmatrix}.
\end{equation}
The first Betti number is
\begin{equation}
        b_1(C_3)=1,
\end{equation}
and the cycle space is
\begin{equation}
        \Ker B=\operatorname{span}\{h\},
        \qquad
        h=\frac1{\sqrt3}(1,1,1)^\top.
\end{equation}
Thus
\begin{equation}
        P_H=hh^\top
        =
        \frac13
        \begin{pmatrix}
        1&1&1\\
        1&1&1\\
        1&1&1
        \end{pmatrix}.
\end{equation}
The scalar
\begin{equation}
        z(t)=h^\top i(t)
\end{equation}
is the average circulating current around the mesh, up to normalization.

For instance, prescribe the balanced nodal injection
\begin{equation}
        p=(1,-1,0)^\top.
\end{equation}
The minimum-norm cycle-free current satisfying $Bi_*=p$ is
\begin{equation}
        i_*=\left(\frac23,-\frac13,-\frac13\right)^\top.
\end{equation}
It satisfies $P_Hi_*=0$.  Thus $i_*$ is the load-serving current with no circulating-current contamination.  An initial condition
\begin{equation}
        i(0)=i_*+z_0h
\end{equation}
has the same nodal injection, because $Bh=0$, but contains a pure circulating-current component.

The ideal normalized current-flow equation is
\begin{equation}
        \dot i
        =
        -\alpha P_C(i-i_*)-(r_0+\kappa)P_Hi.
\end{equation}
It is the one-cycle specialization of \eqref{eq:microgrid-current-error-normal-form}--\eqref{eq:topological-virtual-impedance}.  In Hodge coordinates,
\begin{equation}
        \dot z=-(r_0+\kappa)z,
\end{equation}
while the cut component tracks $i_*$.  If the normalized dynamic virtual-impedance model is used instead, put $z_u:=h^\top w$.  Then \eqref{eq:microgrid-harmonic-dynamic-subsystem} reduces to the two-dimensional subsystem
\begin{equation}
        \frac{d}{dt}
        \begin{pmatrix}
        z\\ z_u
        \end{pmatrix}
        =
        \begin{pmatrix}
        -r_0&1\\
        -\kappa/\tau&-1/\tau
        \end{pmatrix}
        \begin{pmatrix}
        z\\ z_u
        \end{pmatrix}.
        \label{eq:triangle-harmonic-subsystem}
\end{equation}
In the instantaneous virtual-impedance limit $w=-\kappa P_Hi$, this reduces to $\dot z=-(r_0+\kappa)z$.

Since $b_1=1$, one independent sensor and one independent actuator channel are sufficient for the full topological sector.  For example, measuring the current error and actuating on edge $1$ gives
\begin{equation}
        Se=e_1^\top e,
        \qquad
        C\xi=\xi e_1.
\end{equation}
Since
\begin{equation}
        Sh=\frac1{\sqrt3},
        \qquad
        h^\top C=\frac1{\sqrt3},
\end{equation}
the rank condition is satisfied.  Choosing
\begin{equation}
        K=-3\kappa
\end{equation}
gives
\begin{equation}
        h^\top CKS h=-\kappa.
\end{equation}
This is the one-loop compression identity of Theorem~\ref{thm:finite-rank-feedback}.  In the full normalized current model, the actuator adds the rate $\kappa$ to the passive harmonic decay $r_0$; any additional cut-to-harmonic forcing is described by \eqref{eq:microgrid-finite-channel-harmonic-equation}.

\subsection*{Two triangles sharing an edge: two circulating-current modes}
The minimally extended simple graph with two independent circulating-current modes is obtained by taking two triangles sharing one edge.  This example is small enough to display the finite-channel sensor--actuator mechanism of Theorem~\ref{thm:finite-rank-feedback} completely by hand, while still having a non-scalar harmonic sector.

Let the vertices be $1,2,3,4$ and choose the reference orientation of the five edges as
\begin{equation}
        e_1:1\to2,
        \quad
        e_2:2\to3,
        \quad
        e_3:3\to1,
        \quad
        e_4:1\to4,
        \quad
        e_5:4\to2.
\end{equation}
Then
\begin{equation}
        |V|=4,
        \qquad
        |E|=5,
        \qquad
        b_1=5-4+1=2.
\end{equation}
The incidence matrix is
\begin{equation}
B=
\begin{pmatrix}
1&0&-1&1&0\\
-1&1&0&0&-1\\
0&-1&1&0&0\\
0&0&0&-1&1
\end{pmatrix}.
\end{equation}
Two independent cycle vectors are
\begin{equation}
        \ell_1=(1,1,1,0,0)^\top,
        \qquad
        \ell_2=(-1,0,0,1,1)^\top.
\end{equation}
The first vector is the cycle around the triangle $1\to2\to3\to1$.  The second is the cycle around the triangle $1\to4\to2\to1$, with the shared edge $e_1$ traversed in the direction opposite to its reference orientation.  Thus the two vectors represent the two independent circulating-current directions.

An orthonormal basis of $\Ker B$ is obtained, for instance, by taking
\begin{equation}
        h_1=\frac1{\sqrt3}(1,1,1,0,0)^\top,
        \qquad
        h_2=\frac1{\sqrt{24}}(-2,1,1,3,3)^\top.
\end{equation}
Put
\begin{equation}
H=(h_1,h_2)
=
\begin{pmatrix}
1/\sqrt3&-1/\sqrt6\\
1/\sqrt3&1/\sqrt{24}\\
1/\sqrt3&1/\sqrt{24}\\
0&\sqrt6/4\\
0&\sqrt6/4
\end{pmatrix}.
\end{equation}
Then
\begin{equation}
        BH=0,
        \qquad
        H^\top H=I_2,
        \qquad
        P_H=HH^\top.
\end{equation}
Every harmonic current in this two-loop network has the form
\begin{equation}
        i_H=Hz,
        \qquad
        z=(z_1,z_2)^\top\in\R^2.
\end{equation}
Here $z$ is the coordinate vector of the circulating-current component in the chosen harmonic basis $H$.

For the nodal injection
\begin{equation}
        p=(1,-1,0,0)^\top,
\end{equation}
a minimum-norm cycle-free load-serving current is
\begin{equation}
        i_*=
        \left(
        \frac12,-\frac14,-\frac14,\frac14,\frac14
        \right)^\top,
        \qquad
        Bi_*=p,
        \qquad
        P_Hi_*=0.
\end{equation}
An initial condition
\begin{equation}
        i(0)=i_*+a_0h_1+b_0h_2
\end{equation}
therefore has the same nodal balance $p$ as $i_*$, because $Bh_1=Bh_2=0$, but it contains two independent circulating-current components.

\paragraph{What the sensors see.}
Choose two line-current-error sensors, on edges $e_1$ and $e_4$.  Thus
\begin{equation}
        Se=
        \begin{pmatrix}
        (i-i_*)_1\\
        (i-i_*)_4
        \end{pmatrix},
        \qquad
        S=
        \begin{pmatrix}
        1&0&0&0&0\\
        0&0&0&1&0
        \end{pmatrix}.
\end{equation}
For a harmonic current error $e_H=Hz$, the measured vector is
\begin{equation}
        y=Se_H=SHz.
\end{equation}
In the present basis,
\begin{equation}
        SH=
        \begin{pmatrix}
        1/\sqrt3 & -1/\sqrt6\\
        0 & \sqrt6/4
        \end{pmatrix}.
        \label{eq:two-triangle-SH}
\end{equation}
This matrix is invertible.  Hence the two selected sensors do not measure the harmonic coordinates $z$ directly, but they do measure an invertible mixture of them.  No nonzero circulating-current mode is invisible to these two line sensors.

\paragraph{What the actuators can create.}
Use colocated line-actuation channels on the same two edges, so
\begin{equation}
        C=S^\top
        =
        \begin{pmatrix}
        1&0\\
        0&0\\
        0&0\\
        0&1\\
        0&0
        \end{pmatrix}.
\end{equation}
If $\xi=(\xi_1,\xi_2)^\top$ is the actuator command, then $C\xi$ is an edge-space correction supported on $e_1$ and $e_4$.  Its effect on harmonic coordinates is
\begin{equation}
        H^\top C\xi
        =
        (H^\top C)\xi,
\end{equation}
where
\begin{equation}
        H^\top C
        =
        (SH)^\top
        =
        \begin{pmatrix}
        1/\sqrt3 & 0\\
        -1/\sqrt6 & \sqrt6/4
        \end{pmatrix}.
        \label{eq:two-triangle-HTC}
\end{equation}
This matrix is also invertible.  Thus the two selected actuators can produce an arbitrary velocity in the two harmonic coordinates.

\paragraph{Coordinate matching.}
The actuator contribution has the form
\begin{equation}
        u=CKSe.
\end{equation}
On a purely harmonic current error $e=Hz$, its harmonic compression is
\begin{equation}
        H^\top u
        =H^\top CKSHz.
\end{equation}
To contribute the target harmonic damping $-\kappa z$, the gain must satisfy
\begin{equation}
        H^\top CKS H=-\kappa I_2.
        \label{eq:two-triangle-matching-condition}
\end{equation}
Since $SH$ converts harmonic coordinates into sensor readings, and $H^\top C$ converts actuator commands into harmonic-coordinate velocities, the matching gain is
\begin{equation}
        K
        =-
        \kappa
        (H^\top C)^{-1}(SH)^{-1}.
\end{equation}
In the colocated case $C=S^\top$, this can be written as
\begin{equation}
        K
        =-
        \kappa
        \left((SH)(SH)^\top\right)^{-1}.
\end{equation}
Using \eqref{eq:two-triangle-SH}, one obtains
\begin{equation}
        (SH)(SH)^\top
        =
        \begin{pmatrix}
        1/2&-1/4\\
        -1/4&3/8
        \end{pmatrix},
        \qquad
        \left((SH)(SH)^\top\right)^{-1}
        =
        \begin{pmatrix}
        3&2\\
        2&4
        \end{pmatrix}.
\end{equation}
Thus
\begin{equation}
        K=-\kappa
        \begin{pmatrix}
        3&2\\
        2&4
        \end{pmatrix}.
        \label{eq:two-triangle-K-explicit}
\end{equation}
Substitution gives
\begin{equation}
\begin{aligned}
        H^\top CKS H
        &=(SH)^\top
        \left[-\kappa\left((SH)(SH)^\top\right)^{-1}\right]
        (SH) \\
        &=-\kappa I_2.
\end{aligned}
\end{equation}
Thus the actuator compression contributes $-\kappa z$ to the harmonic equation.  In the full normalized current model with zero cut error,
\begin{equation}
        \dot z_1=-(r_0+\kappa)z_1,
        \qquad
        \dot z_2=-(r_0+\kappa)z_2.
\end{equation}
For a general current error the additional term $\mathcal L_{HC}e_C$ is present, as in \eqref{eq:microgrid-finite-channel-harmonic-equation}.  In the dynamic virtual-impedance model, the same two harmonic coordinates enter \eqref{eq:microgrid-harmonic-dynamic-subsystem} with $b_1=2$.

This example displays the finite-channel theorem without hidden notation.  The sensors $S$ read an invertible mixture $SHz$ of the harmonic coordinates.  The actuators $C$ create harmonic-coordinate velocities through $H^\top C$.  The matrix $K$ is exactly the coordinate-matching matrix that decodes the measured harmonic mixture, inserts the desired damping factor $-\kappa$, and re-encodes the result as actuator commands.  This is the two-loop analogue of Theorem~\ref{thm:finite-rank-feedback} and of the rank-seven IEEE construction below.

\subsection*{A weighted parallel-inverter example}
The preceding examples are purely topological current-flow models.  We now give a reduced parallel-inverter example closer to the circulating-current problems studied in the virtual-impedance literature.  Zhang et al. propose an adaptive virtual composite impedance and droop-adjustment strategy for circulating-current suppression in parallel inverter systems \cite{VirtualImpedanceZhang2023}.  The following model is not a reproduction of their PSCAD/EMTDC implementation.  It is a Hodge-coordinate abstraction of the same engineering issue: unequal inverter capacities and line impedances create current-sharing errors and internal circulating-current components, while virtual impedance supplies a voltage-like actuator.

Consider four parallel inverter output paths feeding a common bus.  The reduced graph has two nodes and four parallel edges, all reference-oriented from the inverter side to the common bus:
\begin{equation}
        B=
        \begin{pmatrix}
        1&1&1&1\\
        -1&-1&-1&-1
        \end{pmatrix}.
\end{equation}
The current vector is
\begin{equation}
        i=(i_1,i_2,i_3,i_4)^\top.
\end{equation}
Then
\begin{equation}
        \Ker B
        =
        \left\{
        i\in\R^4:
        i_1+i_2+i_3+i_4=0
        \right\},
        \qquad
        b_1=3.
\end{equation}
Thus there are three independent zero-net-current redistributions among the four inverter paths.  These are the harmonic circulating-current modes in this reduced graph.

For unequal inverter ratings, the desired no-circulation current sharing need not be the equal vector.  Let
\begin{equation}
        \rho=(4,3,2,1)^\top
\end{equation}
represent relative capacity or desired sharing weights.  If the total load current is
\begin{equation}
        I_{\mathrm{load}}=10,
\end{equation}
the desired sharing current is
\begin{equation}
        i_*
        =
        I_{\mathrm{load}}
        \frac{\rho}{\mathbf 1^\top\rho}
        =
        (4,3,2,1)^\top.
\end{equation}
It satisfies
\begin{equation}
        \mathbf 1^\top i_*=10.
\end{equation}
The weighted cut direction is $\operatorname{span}\{\rho\}$.  Equivalently, take the capacity metric matrix $\Sigma_\rho=\operatorname{diag}(\rho)$ and the weighted edge inner product
\begin{equation}
        \langle x,y\rangle_{\Sigma_\rho^{-1}}=x^\top \Sigma_\rho^{-1}y.
\end{equation}
The corresponding weighted harmonic projection onto $\Ker B$ is
\begin{equation}
        P_H^{\Sigma_\rho} i
        =
        i
        -
        \rho\frac{\mathbf 1^\top i}{\mathbf 1^\top\rho}.
        \label{eq:weighted-parallel-harmonic-projection}
\end{equation}
Indeed, $\mathbf 1^\top P_H^{\Sigma_\rho} i=0$, and $P_H^{\Sigma_\rho} i_*=0$.  Thus \eqref{eq:weighted-parallel-harmonic-projection} removes the desired capacity-weighted sharing component and keeps only the circulating-current mismatch.

For example, suppose that the instantaneous current is
\begin{equation}
        i=(5,2,2,1)^\top.
\end{equation}
It carries the same total current as $i_*$, but
\begin{equation}
        P_H^{\Sigma_\rho} i
        =
        i-\rho
        =
        (1,-1,0,0)^\top.
\end{equation}
This is a pure redistribution between the first two inverter paths.  It changes no net load current at the common bus.

A topological virtual-impedance law for this weighted sharing model is
\begin{equation}
        u_{\mathrm{top}}
        =
        -\kappa P_H^{\Sigma_\rho} i.
\end{equation}
Together with a droop or secondary-control layer that maintains the weighted cut target $i_*$, and with passive normalized line resistance $r_0>0$, the ideal current-flow normal form becomes
\begin{equation}
        \dot i
        =
        -\alpha P_C^{\Sigma_\rho}(i-i_*)
        -
        (r_0+\kappa)P_H^{\Sigma_\rho} i,
        \qquad
        P_C^{\Sigma_\rho}=I-P_H^{\Sigma_\rho}.
        \label{eq:weighted-parallel-inverter-ode}
\end{equation}
Consequently
\begin{equation}
        P_C^{\Sigma_\rho} i(t)\to i_*,
        \qquad
        P_H^{\Sigma_\rho} i(t)\to0.
\end{equation}
For the reproducible illustration in Figure~\ref{fig:weighted-parallel-inverter}, we take
\begin{equation}
        h_0=(1.0,-0.8,0.6,-0.8)^\top\in\Ker B,
        \qquad
        i(0)=i_*+h_0,
\end{equation}
and use $r_0=0.15$, $\kappa=1.2$.  Since the initial cut error is zero, the complete trajectory is available in closed form:
\begin{equation}
        i(t)=i_*+e^{-(r_0+\kappa)t}h_0,
        \qquad
        P_H^{\Sigma_\rho} i(t)=e^{-(r_0+\kappa)t}h_0.
\end{equation}
Thus the four inverter currents converge to the unequal sharing target while the circulating-current component decays.  Figure~\ref{fig:weighted-parallel-inverter} is generated by evaluating this formula on the plotting grid; no time-stepping ODE method is used.  Because the selected initial condition has zero cut error, the parameter $\alpha$ does not enter this particular trajectory, although it would regulate a nonzero cut-space deviation in \eqref{eq:weighted-parallel-inverter-ode}.

\begin{figure}[H]
\centering
\includegraphics[width=0.86\textwidth]{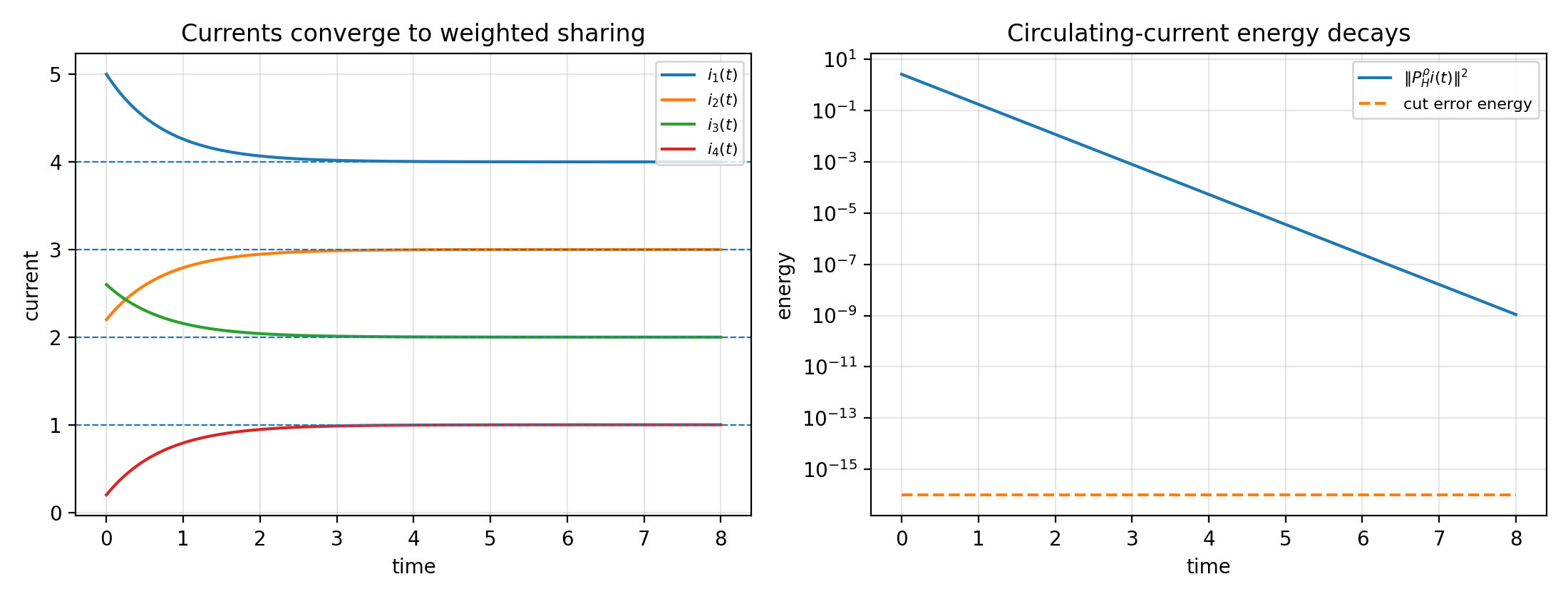}
\caption{Weighted parallel-inverter topological control.  The four line currents converge to the capacity-weighted sharing target while the weighted circulating-current component decays.}
\label{fig:weighted-parallel-inverter}
\end{figure}

In this reduced interpretation, the adaptive droop layer belongs to the weighted sharing or cut component, whereas virtual impedance supplies the line-current actuator used to damp the harmonic redistribution.  The Hodge formulation adds a structural design question: do the selected current measurements and virtual-impedance channels have full rank on the three-dimensional space $\Ker B$?  If not, some circulating-current mode is invisible or uncontrollable, independently of the local gain tuning.

The engineering model in \cite{VirtualImpedanceZhang2023} contains details suppressed here, including active--reactive power coupling, adaptive impedance setting, droop-coefficient adjustment, and electromagnetic-transient simulation.  The purpose of the present example is narrower: it shows how a Zhang-type circulating-current objective can be expressed in the current-flow normal form \eqref{eq:microgrid-current-model}, with the unwanted current identified as a weighted harmonic edge-current component.

\section{Transmission-grid loop flows and the IEEE 14-bus setting}
\label{sec:transmission-preparation}
Loop flows in meshed transmission networks motivate the line-flow application of the abstract thermostat.  A line-flow vector $f\in\R^m$ and a nodal injection vector $p\in\R^n$ are related by
\begin{equation}
        Bf=p,
        \label{eq:transmission-static-balance}
\end{equation}
where positive entries of $p$ represent net generation and negative entries net consumption.  The balance equation fixes the source--load requirement but does not determine the internal redistribution of flow, because any element of $\Ker B$ may be added without changing $p$.  Large circulating or loop flows may therefore occupy line capacity without contributing to net transfer \cite{ColettaLoopFlows,KorabOwczarek2016}.  Cycle-dependent redistribution also underlies Braess-type effects in power networks \cite{SchaeferBraess,TchuisseuBraessControl}.

At the physical level, line actuation may be interpreted through phase-shifting transformers, controllable series compensation, FACTS devices, HVDC links, or power-electronic interfaces \cite{Sadikovic2006,ENTSOEPST,SiemensPST,HitachiTCSC,ENTSOEHVDC}.  The paper uses idealized line-actuation coordinates rather than detailed device models.  A physical device changes the branch constitutive relation; after linearization, its first-order effect is an input direction in edge-flow space.  The finite-channel rank conditions of Section~8 test whether the selected measurements distinguish, and the selected devices excite, all independent harmonic flow modes.  They complement rather than replace classical modal observability, controllability, security, and device-placement criteria.

The principal transmission-grid example uses the MATPOWER IEEE 14-bus topology and branch reactances \cite{MATPOWER,MATPOWERcase14}.  Its twenty branches and fourteen buses give $b_1=7$.  The normalized branch coefficients, synthetic dynamic data, fixed seven-edge architecture, and weighted rank verification are specified together in Section~\ref{sec:ieee14-application}; the intervening swing section first develops the nonlinear closures and their tangent and energy analysis.

\section{Nonlinear line-actuated swing systems, harmonic-output linearization, and dissipation}
\subsection{Purpose and modelling levels}
The theta-network and inverter-microgrid sections take linear edge-flow or edge-current evolution equations as their starting models.  In a transmission grid, by contrast, line flow is a nonlinear output of bus phases and line-actuator variables.  This section therefore formulates two complete nonlinear closed-loop swing systems, one using the ideal weighted harmonic projector and one using finite-channel feedback.

Writing $z=(\theta,\omega,u)$, let $\mathcal F_{\rm id}$ and $\mathcal F_{\rm fc}$ denote the corresponding closed vector fields.  Target centering, together with $p=Bf_*$, makes
$z_*=(\theta_*,0,u_*)$ an equilibrium of both systems.  Their tangent evolution problems are obtained by evaluating the Fr\'echet derivatives of the complete closed vector fields at this equilibrium:
\[
        \dot\xi
        =\mathrm D\mathcal F_{\rm id}(z_*)\xi,
        \qquad
        \dot\xi
        =\mathrm D\mathcal F_{\rm fc}(z_*)\xi,
        \qquad
        \xi=(\vartheta,\nu,w).
\]
They are written explicitly in
\eqref{eq:section15-linearized-ideal-system} and
\eqref{eq:section15-linearized-finite-system}; the latter is expanded
componentwise in
\eqref{eq:section15-linearized-finite-system-expanded}.  Thus the
feedback laws are linearized together with the swing plant, rather
than appended after linearization.  After quotienting the uniform
phase-shift symmetry and restricting the actuator deviation to the
implemented invariant subspace, a Hurwitz reduced derivative invokes
the finite-dimensional principle of linearized stability
\cite{AmannODE} and yields local exponential stability of the
corresponding nonlinear equilibrium class.  The projected output
identities identify particular blocks of this derivative; they do not
replace the stability test for the complete reduced tangent system.

We use the variable terminology consistently throughout this section.  The
\emph{full state variables} are $z=(\theta,\omega,u)$.  The
\emph{effective edge-phase variables} are $(\eta,\omega)$, with
$\eta=B^\top\theta+u$.  The variables $(\nu,r_C,h)$ introduced in the
tangent analysis are called \emph{output-level Hodge variables}.  Finally,
the term \emph{reduced tangent state} is reserved for the quotient and
invariant-subspace restriction used to remove neutral symmetries and
redundant actuator directions before applying the principle of linearized
stability.

The analysis separates exact harmonic compression, the additional condition required for exact prescribed decay, and full reduced-state recovery.  The nonlinear part then exhibits a winding-one passive equilibrium, derives a controller-independent Bregman balance and its ideal and colocated finite-channel dissipation consequences, and ends with an exactly solvable controlled ring trajectory.  The IEEE 14 application is treated separately in the following section.

The use of swing equations and their relation to Kuramoto-type synchronization models is classical; see, for example, D{\"o}rfler and Bullo \cite{DorflerBullo2012}.  A synchronous operating point is a phase-locked state with zero frequency deviation and generally nonzero phase differences encoding active-power transfer.  The topological objective considered here is additional: preserve the balanced transfer while dissipating the harmonic component of the line-flow output.

\subsection{Nonlinear swing dynamics and two line-feedback closures}
\label{subsec:nonlinear-swing-closures}
Throughout this section, recall that $X=\R^E\cong\R^m$ is the edge space, with one coordinate for each reference-oriented transmission line.  Let $\theta\in\R^n$ be the vector of bus-voltage phase angles, $\omega\in\R^n$ the vector of frequency deviations, and $u\in X$ the ideal line-actuator variable.  At the present modelling level, $u_e$ is an additive controllable phase shift in the constitutive law of line $e$.

For a reference-oriented line $e:i\to j$, the component $(B^\top\theta)_e=\theta_i-\theta_j$ is the oriented phase-angle difference generated by the incident buses.  Thus $B^\top\theta\in\im B^\top$ is the exact edge-phase field induced by the bus phases; it is not itself a line-flow vector.  The effective phase difference seen by each line is
\begin{equation}
        \eta=B^\top\theta+u\in X.
        \label{eq:effective-edge-phase}
\end{equation}
The term $B^\top\theta$ is generated by nodal phase angles, whereas $u$ is an independent edge-level control coordinate.

The reduced lossless branch law assigns to $\eta_e$ the signed active-power flow $f_e=a_e\sin\eta_e$ in the chosen reference direction, where $a_e>0$ is the effective line-coupling coefficient.  Hence the nonlinear line-flow map
\begin{equation}
        f:\R^n\times X\longrightarrow X
        \label{eq:nonlinear-line-flow-map-space}
\end{equation}
is
\begin{equation}
        f(\theta,u)
        =
        \operatorname{diag}(a)\sin(B^\top\theta+u)
        =
        f(\eta)
        \in X.
        \label{eq:section15-nonlinear-line-flow-map}
\end{equation}
Although $\eta$ is represented by a vector in $X$, each component is physically an angle modulo $2\pi$.  Let
\begin{equation}
        \operatorname{pr}:\R\longrightarrow[-\pi,\pi)
\end{equation}
denote the principal representative modulo $2\pi$,
\begin{equation}
        \operatorname{pr}(x)
        =x-2\pi\left\lfloor\frac{x+\pi}{2\pi}\right\rfloor,
        \qquad x\in\R,
        \label{eq:principal-representative}
\end{equation}
and apply it componentwise to vectors.  The map is discontinuous at the branch-cut values $(2k+1)\pi$, while the angle-cohesive region below stays uniformly away from them.

Fix a synchronous operating point
\begin{equation}
        (\theta_*,0,u_*).
\end{equation}
Here synchrony means that the bus phases are phase-locked and the frequency deviation is zero.  Put
\begin{equation}
        \eta_*=B^\top\theta_*+u_*,
        \qquad
        f_*:=f(\theta_*,u_*).
\end{equation}
Thus $\eta_*$ is the vector of effective operating-point line phases and $f_*$ is the associated line-flow output.  For some $0<\gamma<\pi/2$, the operating point is called $\gamma$-angle-cohesive if
\begin{equation}
        \left|\operatorname{pr}(\eta_{*,e})\right|\le\gamma,
        \qquad e\in E.
        \label{eq:angle-cohesive-operating-point}
\end{equation}
This condition keeps every line on the strictly increasing branch of the sine law.  Consequently, the differential of the branch constitutive map at the operating point is the positive diagonal operator
\begin{equation}
        W_*
        :=\mathrm D_\eta f(\eta_*)
        =\operatorname{diag}\!\left(a_e\cos(\operatorname{pr}(\eta_{*,e}))\right)
        \in\Lin(X),
        \qquad W_*>0.
        \label{eq:operating-point-edge-sensitivity}
\end{equation}
Thus $W_*$ is the small-signal sensitivity converting effective line-phase variations into line-flow variations.

With the incidence convention fixed above, $Bf$ is the vector of net nodal power outflows.  Hence
\begin{equation}
        p=Bf_*
        \label{eq:section15-target-balance}
\end{equation}
is the steady nodal power-balance condition.  We additionally assume that the target flow has no harmonic loop-flow component in the operating-point weighted geometry:
\begin{equation}
        P_H^{W_*}f_*=0.
        \label{eq:section15-cycle-free-target}
\end{equation}
The weighted cut and harmonic projections associated with $W_*^{-1}$ are frozen at the selected operating point and are defined explicitly below.

Let $M,D\in\Lin(\R^n)$ be positive diagonal operators representing bus inertia and passive frequency damping.  In each closure below, the first equation is the phase--frequency kinematic relation, the second is the nodal swing balance, and the third specifies the line-actuator dynamics.  The first nonlinear closure uses the full weighted harmonic projector:
\begin{subequations}
\label{eq:section15-ideal-nonlinear-closed-loop}
\begin{align}
        \dot\theta
        &=\omega,
        \\
        M\dot\omega
        &=p-D\omega-Bf(\theta,u),
        \\
        \tau\dot u
        &=-\kappa W_*^{-1}P_H^{W_*}
        \bigl(f(\theta,u)-f_*\bigr).
\end{align}
\end{subequations}
Here $\tau>0$ is the actuator time scale and $\kappa>0$ is the harmonic feedback gain.  Since \eqref{eq:section15-cycle-free-target} holds, the subtraction of $f_*$ is redundant in the last equation, but it makes the regulated target explicit.

The second closure replaces the global projector by a finite-channel architecture.  Let
\begin{equation}
        S\in\Lin(X,\R^{n_s}),
        \qquad
        K\in\Lin(\R^{n_s},\R^{n_a}),
        \qquad
        C\in\Lin(\R^{n_a},X),
        \label{eq:section15-finite-map-spaces}
\end{equation}
where $S$ is the line-flow sensor map, $C$ contains the line-flow actuator directions, and $K$ maps sensor signals to actuator commands.  In canonical coordinates these maps have sizes $n_s\times m$, $n_a\times n_s$, and $m\times n_a$, respectively.  The finite-channel nonlinear closed system is
\begin{subequations}
\label{eq:section15-finite-nonlinear-closed-loop}
\begin{align}
        \dot\theta
        &=\omega,
        \\
        M\dot\omega
        &=p-D\omega-Bf(\theta,u),
        \\
        \tau\dot u
        &=W_*^{-1}CKS
        \bigl(f(\theta,u)-f_*\bigr).
\end{align}
\end{subequations}
As in Remark~\ref{rem:finite-channel-target-centering}, the finite-channel controller is applied to the target error rather than to the raw line flow.  The subtraction of $f_*$ guarantees that $(\theta_*,0,u_*)$ is an equilibrium without requiring the additional condition $CKSf_*=0$.

\begin{remark}[Two feedback ports and modelling status]
The passive term $-D\omega$ in \eqref{eq:section15-ideal-nonlinear-closed-loop} and \eqref{eq:section15-finite-nonlinear-closed-loop} acts through the bus-frequency equation.  It dissipates electromechanical motion and vanishes at a synchronous state with $\omega=0$.  The line controllers act through the equation for $u$, which enters the constitutive relation \eqref{eq:section15-nonlinear-line-flow-map}.  They therefore change the effective edge phases and reshape the line-flow output directly.

The ideal closure \eqref{eq:section15-ideal-nonlinear-closed-loop} assumes access to the complete weighted harmonic projection and a distributed actuator pattern.  The finite closure \eqref{eq:section15-finite-nonlinear-closed-loop} uses selected or filtered sensor and actuator channels and need reproduce only the desired harmonic compression after linearization.  Existing phase-shifting transformers, controllable series-compensation devices, FACTS devices, HVDC links, and power-electronic interfaces motivate such line-level actuation, but their detailed dynamics, limits, delays, and protection constraints are not represented here; these belong to the engineering limitations discussed after the main IEEE 14 application.
\end{remark}

\begin{remark}[Vertex and edge Hodge levels]
Geometrically, the swing variables $\theta$ and $\omega$ are $0$-cochains on the network vertices.  At the operating point their tangent restoring operator is the weighted vertex Laplacian
\begin{equation}
        L_{0,*}=BW_*B^\top\in\Lin(\R^n).
\end{equation}
Because $W_*>0$, a connected graph satisfies
\begin{equation}
        \Ker L_{0,*}
        =\Ker B^\top
        =\operatorname{span}\{\mathbf 1_V\}.
\end{equation}
This one-dimensional harmonic $0$-mode is the uniform phase-shift symmetry; it disappears after a reference phase is fixed, or equivalently after quotienting by constant vertex functions.

The branch quantities introduced above are instead $1$-cochains.  More precisely,
\begin{equation}
        B^\top\theta\in\mathcal C=\im B^\top,
        \qquad
        \eta,u,f\in C^1(G)=X.
        \label{eq:edge-cochain-levels}
\end{equation}
Thus the bus-generated part of the effective phase is an exact edge field, whereas the independent actuator $u$ is not restricted to the cut space and may supply both cut and harmonic edge components.  In the linearized model derived below, with the line actuator held fixed, the bus-phase contribution to the line-flow output belongs to the weighted cut space
\begin{equation}
        \mathcal C_*=\im(W_*B^\top),
\end{equation}
and consequently has zero weighted harmonic projection.  For the graph regarded as a one-dimensional complex, the edge Hodge Laplacian is
\begin{equation}
        L_1=B^\top B\in\Lin(X),
\end{equation}
and
\begin{equation}
        \Ker L_1=\Ker B,
        \qquad
        \dim\Ker B=b_1.
\end{equation}
These harmonic $1$-cochains are divergence-free cycle flows rather than gauge directions; the positive operating-point metric changes the orthogonal splitting, but not the harmonic subspace $\Ker B$ itself.

Introducing the independent edge actuator $u$ therefore augments, rather than replaces, the $0$-cochain swing plant by a $1$-cochain control port.  In the ideal model this port makes the $b_1$-dimensional cycle-flow sector directly accessible.  A finite-channel realization retains that authority only under the harmonic rank conditions developed above.  Direct authority over the harmonic output of the linearized model is therefore supplied by the independent edge-actuator component.
\end{remark}

\subsection{Linearization of the two closed-loop systems at the cycle-free target}
Introduce the deviation variables
\begin{equation}
        \vartheta=\theta-\theta_*,
        \qquad
        \nu=\omega,
        \qquad
        w=u-u_*.
        \label{eq:section15-deviation-variables}
\end{equation}
Let
\begin{equation}
        \mathcal L_*
        :=\mathrm D f(\theta_*,u_*)
        \in\Lin(\R^n\times X,X)
        \label{eq:linearized-line-flow-operator-space}
\end{equation}
denote the Fr\'echet derivative of the nonlinear line-flow map at the operating point.  Its action is
\begin{equation}
        \mathcal L_*[\vartheta,w]
        =
        W_*(B^\top\vartheta+w).
\end{equation}
To keep the state coupling visible, we write
\begin{equation}
        r=r(\vartheta,w)
        :=
        \mathcal L_*[\vartheta,w]
        =
        W_*B^\top\vartheta+W_*w
        \label{eq:section15-linearized-line-flow-output}
\end{equation}
for the linearized line-flow output.  When no confusion is possible, we abbreviate $r(\vartheta,w)$ by $r$.  Thus the first-order expansion is
\begin{equation}
        f(\theta_*+\vartheta,u_*+w)-f_*
        =
        r
        +
        O\bigl((\norm{\vartheta}+\norm{w})^2\bigr).
\end{equation}

Linearizing the complete ideal closed system \eqref{eq:section15-ideal-nonlinear-closed-loop}, including its feedback law, gives
\begin{subequations}
\label{eq:section15-linearized-ideal-system}
\begin{align}
        \dot\vartheta
        &=\nu,
        \\
        M\dot\nu
        &=-D\nu-Br(\vartheta,w),
        \\
        \tau\dot w
        &=-\kappa W_*^{-1}P_H^{W_*}r(\vartheta,w).
\end{align}
\end{subequations}
Similarly, the tangent equation of the finite-channel system \eqref{eq:section15-finite-nonlinear-closed-loop} is
\begin{subequations}
\label{eq:section15-linearized-finite-system}
\begin{align}
        \dot\vartheta
        &=\nu,
        \\
        M\dot\nu
        &=-D\nu-Br(\vartheta,w),
        \\
        \tau\dot w
        &=W_*^{-1}CKSr(\vartheta,w).
\end{align}
\end{subequations}
The two tangent systems share the same linearized swing equations.  They differ only in the actuator equation, because the feedback law itself has also been linearized.  The abbreviation $r(\vartheta,w)$ should not hide the state coupling.  For example, the finite-channel tangent system is equivalently
\begin{subequations}
\label{eq:section15-linearized-finite-system-expanded}
\begin{align}
        \dot\vartheta
        &=\nu,
        \\
        M\dot\nu
        &=-D\nu-BW_*B^\top\vartheta-BW_*w,
        \\
        \tau\dot w
        &=W_*^{-1}CKSW_*B^\top\vartheta
        +W_*^{-1}CKSW_*w.
\end{align}
\end{subequations}
Thus $w$ acts directly on the swing equation through $-BW_*w$, while the bus-phase deviation acts on the actuator dynamics through $W_*^{-1}CKSW_*B^\top\vartheta$; the frequency deviation enters the actuator loop indirectly through $\dot\vartheta=\nu$.

The weighted projection identities used below are a direct consequence of the following standard finite-dimensional fact \cite{HornJohnsonMatrixAnalysis}.  We include the proof to fix the metric convention and make the later annihilation and energy identities explicit.

\begin{lemma}[Weighted orthogonal projection]
\label{lem:weighted-orthogonal-projection}
Let $Q\in\Lin(\R^m)$ be symmetric positive definite, and let $H\in\Lin(\R^r,\R^m)$ have full column rank.  Then
\begin{equation}
        P_H^Q
        :=H(H^\top QH)^{-1}H^\top Q
        \label{eq:general-weighted-projector}
\end{equation}
is the $Q$-orthogonal projection onto $\im H$ for the inner product
\begin{equation}
        \langle x,y\rangle_Q=x^\top Qy.
\end{equation}
More precisely,
\begin{equation}
        (P_H^Q)^2=P_H^Q,
        \qquad
        \im P_H^Q=\im H,
        \qquad
        \Ker P_H^Q=(\im H)^{\perp_Q},
        \label{eq:weighted-projection-properties}
\end{equation}
and
\begin{equation}
        (P_H^Q)^\top Q=QP_H^Q.
        \label{eq:weighted-projection-self-adjoint}
\end{equation}
Consequently,
\begin{equation}
        \langle x,P_H^Qx\rangle_Q
        =\norm{P_H^Qx}_Q^2.
        \label{eq:weighted-projection-energy-identity}
\end{equation}
If $H^\top QH=I_r$, then $P_H^Q=HH^\top Q$.
\end{lemma}

\begin{proof}
For $c\in\R^r\setminus\{0\}$,
\begin{equation}
        c^\top H^\top QHc
        =(Hc)^\top Q(Hc)>0,
\end{equation}
because $Q>0$ and $H$ has full column rank.  Thus $H^\top QH$ is invertible.  Direct multiplication gives
\begin{align}
        (P_H^Q)^2
        &=H(H^\top QH)^{-1}H^\top QH
          (H^\top QH)^{-1}H^\top Q
        \notag\\
        &=P_H^Q.
\end{align}
The formula \eqref{eq:general-weighted-projector} shows that $\im P_H^Q\subseteq\im H$, while
\begin{equation}
        P_H^QH=H
\end{equation}
shows the reverse inclusion.  Hence $\im P_H^Q=\im H$.

For every $x\in\R^m$,
\begin{equation}
        H^\top Q(I-P_H^Q)x=0,
\end{equation}
so $(I-P_H^Q)x\in(\im H)^{\perp_Q}$.  If $x\in(\im H)^{\perp_Q}$, then $H^\top Qx=0$ and therefore $P_H^Qx=0$.  Conversely, if $P_H^Qx=0$, multiplication by $H^\top Q$ gives
\begin{equation}
        H^\top Qx
        =H^\top QP_H^Qx
        =0,
\end{equation}
where $H^\top QP_H^Q=H^\top Q$ follows from the defining formula.  This proves the kernel identity.

Taking transposes in \eqref{eq:general-weighted-projector} and using symmetry of $Q$ and $H^\top QH$ yields
\begin{equation}
        (P_H^Q)^\top Q
        =QH(H^\top QH)^{-1}H^\top Q
        =QP_H^Q.
\end{equation}
Finally, idempotence and $Q$-self-adjointness give
\begin{equation}
        \langle x,P_H^Qx\rangle_Q
        =x^\top QP_H^Qx
        =x^\top(P_H^Q)^\top QP_H^Qx
        =\norm{P_H^Qx}_Q^2.
\end{equation}
The orthonormal-basis formula is immediate.
\end{proof}

\subsubsection*{Weighted Hodge decomposition of the linearized line-flow output}

We now specialize Lemma~\ref{lem:weighted-orthogonal-projection} to the operating-point metric $W_*^{-1}$ and introduce the weighted Hodge coordinates common to the ideal and finite-channel output analyses.

Let
\begin{equation}
        H_*\in\Lin(\R^{b_1},X)
\end{equation}
be a coordinate map whose columns form a basis of $\Ker B$ satisfying
\begin{equation}
        H_*^\top W_*^{-1}H_*=I_{b_1}.
\end{equation}
Then Lemma~\ref{lem:weighted-orthogonal-projection}, with $Q=W_*^{-1}$, gives
\begin{equation}
        P_H^{W_*}=H_*H_*^\top W_*^{-1},
        \qquad
        P_C^{W_*}=I-P_H^{W_*}.
        \label{eq:section15-weighted-projectors}
\end{equation}
Moreover,
\begin{equation}
        (\Ker B)^{\perp_{W_*^{-1}}}
        =\im(W_*B^\top).
        \label{eq:weighted-cut-complement}
\end{equation}
Indeed, $H_*^\top W_*^{-1}(W_*B^\top y)=H_*^\top B^\top y=0$, and the two spaces have the same dimension.  In particular,
\begin{equation}
        P_H^{W_*}W_*B^\top=0.
        \label{eq:weighted-projector-annihilation}
\end{equation}
For the linearized line-flow output, introduce the weighted Hodge coordinates
\begin{equation}
        r=r_C+H_*h,
        \qquad
        r_C=P_C^{W_*}r,
        \qquad
        h=H_*^\top W_*^{-1}r.
        \label{eq:section15-harmonic-coordinates}
\end{equation}
In these common coordinates, the ideal and finite-channel feedback laws lead to the distinct output dynamics stated below.

\begin{theorem}[Ideal Hodge separation and harmonic-output damping for the linearized swing system]
\label{thm:section15-ideal-harmonic-output}
Let $(\theta_*,0,u_*)$ satisfy \eqref{eq:section15-target-balance} and \eqref{eq:section15-cycle-free-target}, and let the operating point be angle-cohesive.  Every solution of the ideal tangent system \eqref{eq:section15-linearized-ideal-system} satisfies
\begin{equation}
        \tau\frac{\dd}{\dd t}P_H^{W_*}r
        =
        -\kappa P_H^{W_*}r.
        \label{eq:section15-ideal-harmonic-output}
\end{equation}
Consequently,
\begin{equation}
        P_H^{W_*}r(t)
        =
        e^{-(\kappa/\tau)t}P_H^{W_*}r(0)
        \qquad (t\ge0).
        \label{eq:section15-ideal-harmonic-solution}
\end{equation}
Moreover, on the effective output state
\begin{equation}
        z_{\rm eff}:=(\nu,r_C,h)
        \in\R^n\times\mathcal C_*\times\R^{b_1},
        \label{eq:section15-ideal-effective-state}
\end{equation}
the ideal tangent system has the block-separated form
\begin{subequations}
\label{eq:section15-ideal-effective-output-system}
\begin{align}
        M\dot\nu
        &=-D\nu-Br_C,
        \\
        \dot r_C
        &=W_*B^\top\nu,
        \\
        \tau\dot h
        &=-\kappa h.
\end{align}
\end{subequations}
Let $\mathcal A_{\rm sw}$ denote the generator of the first two equations on $\R^n\times\mathcal C_*$, and let $\mathcal A_{\rm id}^{\rm eff}$ denote the generator of \eqref{eq:section15-ideal-effective-output-system}.  Then
\begin{equation}
        \mathcal A_{\rm id}^{\rm eff}
        =
        \mathcal A_{\rm sw}
        \oplus
        \left(-\frac{\kappa}{\tau}I_{b_1}\right),
        \label{eq:section15-ideal-generator-direct-sum}
\end{equation}
and hence its characteristic polynomial factors as
\begin{equation}
        \chi_{\mathcal A_{\rm id}^{\rm eff}}(\lambda)
        =
        \chi_{\mathcal A_{\rm sw}}(\lambda)
        \left(\lambda+\frac{\kappa}{\tau}\right)^{b_1}.
        \label{eq:section15-ideal-characteristic-factorization}
\end{equation}
Thus the ideal generator preserves the passive swing spectrum and appends the eigenvalue $-\kappa/\tau$ with algebraic multiplicity $b_1$.  In particular, $\mathcal A_{\rm id}^{\rm eff}$ is Hurwitz if and only if $\mathcal A_{\rm sw}$ is Hurwitz.  Under the standing assumptions that the graph is connected and $M,D,W_*>0$, both generators are Hurwitz.
\end{theorem}

\begin{proof}
By \eqref{eq:section15-linearized-line-flow-output} and $\dot\vartheta=\nu$,
\begin{equation}
        \dot r
        =
        \mathcal L_*[\nu,\dot w]
        =
        W_*B^\top\nu+W_*\dot w.
        \label{eq:section15-output-derivative}
\end{equation}
By \eqref{eq:weighted-projector-annihilation},
\begin{equation}
        P_H^{W_*}W_*B^\top\nu=0.
        \label{eq:section15-projection-annihilates-swing-term}
\end{equation}
Multiplying \eqref{eq:section15-output-derivative} by $\tau P_H^{W_*}$ and using the third equation of \eqref{eq:section15-linearized-ideal-system} yields
\begin{align}
        \tau\frac{\dd}{\dd t}P_H^{W_*}r
        &=
        \tau P_H^{W_*}W_*B^\top\nu
        +
        \tau P_H^{W_*}W_*\dot w
        \notag\\
        &=
        -\kappa(P_H^{W_*})^2r
        =
        -\kappa P_H^{W_*}r.
\end{align}
Equation \eqref{eq:section15-ideal-harmonic-solution} follows.

Since $BH_*=0$, one has $Br=Br_C$, and the swing equation gives the first equation of \eqref{eq:section15-ideal-effective-output-system}.  Projecting \eqref{eq:section15-output-derivative} by $P_C^{W_*}$ gives
\begin{equation}
        \dot r_C
        =
        W_*B^\top\nu,
\end{equation}
because $W_*B^\top\nu\in\mathcal C_*$ and $W_*\dot w=-(\kappa/\tau)P_H^{W_*}r$ is harmonic.  Multiplication by $H_*^\top W_*^{-1}$ gives $\tau\dot h=-\kappa h$.  This proves the block separation and therefore \eqref{eq:section15-ideal-generator-direct-sum}--\eqref{eq:section15-ideal-characteristic-factorization}.

It remains to verify the standing stability assertion for $\mathcal A_{\rm sw}$.  On $\R^n\times\mathcal C_*$, consider
\begin{equation}
        \mathcal E_{\rm sw}(\nu,r_C)
        :=
        \frac12\nu^\top M\nu
        +
        \frac12 r_C^\top W_*^{-1}r_C.
        \label{eq:section15-passive-swing-output-energy}
\end{equation}
Along the first two equations of \eqref{eq:section15-ideal-effective-output-system},
\begin{equation}
        \dot{\mathcal E}_{\rm sw}
        =
        -\nu^\top D\nu.
        \label{eq:section15-passive-swing-output-dissipation}
\end{equation}
Because $D>0$, the largest invariant subset of $\{\dot{\mathcal E}_{\rm sw}=0\}$ has $\nu=0$ and hence $Br_C=0$.  Since $r_C\in\mathcal C_*=(\Ker B)^{\perp_{W_*^{-1}}}$, this implies $r_C=0$.  Thus the reduced passive swing system is asymptotically stable and, being finite-dimensional and linear, exponentially stable.  Hence $\mathcal A_{\rm sw}$ is Hurwitz, and the direct-sum formula gives the same conclusion for $\mathcal A_{\rm id}^{\rm eff}$.
\end{proof}

\begin{remark}[Harmonic disturbances and passive swing damping]
\label{rem:passive-swing-harmonic-invariance}
The preceding theorem replaces a conservation law of the passive tangent swing system by prescribed decay.  If the line actuator is held fixed, so that $\dot w=0$, then \eqref{eq:section15-output-derivative} reduces to
\begin{equation}
        \dot r=W_*B^\top\nu.
\end{equation}
Since $W_*B^\top\nu\in\mathcal C_*=\im(W_*B^\top)$, \eqref{eq:weighted-projector-annihilation} gives
\begin{equation}
        \frac{\dd}{\dd t}P_H^{W_*}r=0,
        \qquad
        P_H^{W_*}r(t)=P_H^{W_*}r(0).
        \label{eq:passive-swing-harmonic-conservation}
\end{equation}
Thus passive frequency damping may stabilize bus motion and the cut-output component while leaving an already present harmonic line-flow deviation unchanged.  In particular, a purely harmonic output perturbation $r_H\in\Ker B$ satisfies $Br_H=0$ and creates no incremental nodal torque.

The ideal line controller changes \eqref{eq:passive-swing-harmonic-conservation} into the exponential law \eqref{eq:section15-ideal-harmonic-solution}.  This identifies precisely the additional disturbance class targeted by topological line actuation.

Within the passive linearized system the cut-space restriction is exact.  If the actuator deviation is held at zero, then
\begin{equation}
        r(t)=W_*B^\top\vartheta(t)\in\mathcal C_*
        \qquad (t\ge0),
\end{equation}
so bus-angle dynamics cannot generate a harmonic line-flow output.  More generally, if the actuator deviation is held fixed, any harmonic component already contained in $W_*w$ is conserved by \eqref{eq:passive-swing-harmonic-conservation}.  A nonzero harmonic output in the linearized model must therefore be supplied by an independent line degree of freedom, such as a line-device or converter state, a switching or reclosing disturbance, or unequal line-current states.  The full nonlinear system has a different possibility: a finite bus-angle configuration can itself carry harmonic line flow, as Subsection~\ref{subsec:winding-one-ring-equilibrium} shows explicitly.

The block separation in Theorem~\ref{thm:section15-ideal-harmonic-output} also gives a stability interpretation.  The ideal thermostat is non-invasive with respect to the reduced passive swing dynamics: it leaves every nodal--cut eigenvalue unchanged and appends only the stable harmonic block $-(\kappa/\tau)I_{b_1}$.  It therefore cannot reverse the linearized stability of a stable passive swing equilibrium.  Conversely, it does not repair an unstable nodal--cut mode.  Its benchmark role is to test topological selectivity while preserving the native linearized swing dynamics, rather than to maximize the full closed-loop stability margin.
\end{remark}

\begin{lemma}[Weighted cut--harmonic block form of the finite-channel tangent system]
\label{lem:section15-finite-output-block-form}
Consider the finite-channel tangent system \eqref{eq:section15-linearized-finite-system}, and put
\begin{equation}
        L:=CKS\in\Lin(X).
\end{equation}
With the weighted output decomposition $r=r_C+H_*h$ from \eqref{eq:section15-harmonic-coordinates}, define the four blocks of $L$ relative to
\begin{equation}
        X=\mathcal C_*\oplus^{\perp_{W_*^{-1}}}\Ker B,
        \qquad
        \mathcal C_*:=\im(W_*B^\top),
\end{equation}
by
\begin{equation}
\begin{aligned}
        L_{CC}
        &:={P_C^{W_*}L}\big|_{\mathcal C_*}
          \in\Lin(\mathcal C_*),
        &
        L_{CH}
        &:=P_C^{W_*}LH_*
          \in\Lin(\R^{b_1},\mathcal C_*),
        \\
        L_{HC}
        &:={H_*^\top W_*^{-1}L}\big|_{\mathcal C_*}
          \in\Lin(\mathcal C_*,\R^{b_1}),
        &
        L_{HH}
        &:=H_*^\top W_*^{-1}LH_*
          \in\Lin(\R^{b_1}).
\end{aligned}
        \label{eq:section15-controller-blocks}
\end{equation}
The first index denotes the receiving output sector and the second the originating sector.  For every solution of \eqref{eq:section15-linearized-finite-system}, the variables $(\nu,r_C,h)$ satisfy the closed output system
\begin{subequations}
\label{eq:section15-coupled-output-system}
\begin{align}
        M\dot\nu
        &=-D\nu-Br_C,
        \\
        \tau\dot r_C
        &=\tau W_*B^\top\nu+L_{CC}r_C+L_{CH}h,
        \\
        \tau\dot h
        &=L_{HC}r_C+L_{HH}h.
        \label{eq:section15-general-harmonic-block-equation}
\end{align}
\end{subequations}
In particular,
\begin{equation}
\begin{aligned}
        h(t)
        ={}&e^{(t/\tau)L_{HH}}h(0)
        \\
        &+\frac1\tau\int_0^t
        e^{((t-s)/\tau)L_{HH}}L_{HC}r_C(s)\,\dd s.
\end{aligned}
        \label{eq:section15-harmonic-variation-of-constants}
\end{equation}
\end{lemma}

\begin{proof}
From \eqref{eq:section15-output-derivative} and the actuator equation in \eqref{eq:section15-linearized-finite-system},
\begin{equation}
        \tau\dot r
        =
        \tau W_*B^\top\nu+Lr.
        \label{eq:section15-full-output-derivative}
\end{equation}
Since $r=r_C+H_*h$, projection by $P_C^{W_*}$ gives
\begin{equation}
        \tau\dot r_C
        =
        \tau W_*B^\top\nu+L_{CC}r_C+L_{CH}h,
\end{equation}
where $W_*B^\top\nu\in\mathcal C_*$.  Multiplication by $H_*^\top W_*^{-1}$ and the identity
\begin{equation}
        H_*^\top W_*^{-1}W_*B^\top
        =H_*^\top B^\top
        =0
\end{equation}
give
\begin{equation}
        \tau\dot h=L_{HC}r_C+L_{HH}h.
\end{equation}
Finally, $BH_*=0$ implies $Br=Br_C$, which yields the first equation of \eqref{eq:section15-coupled-output-system}.  Formula \eqref{eq:section15-harmonic-variation-of-constants} follows from variation of constants.
\end{proof}

The block form separates four controller effects.  The block $L_{HH}$ governs harmonic input acting on harmonic output and is the block prescribed by finite-channel compression.  The block $L_{HC}$ transfers cut-output motion into the harmonic coordinate and therefore appears as the direct forcing term in the harmonic equation.  Conversely, $L_{CH}$ transfers harmonic motion into the cut output, from which it enters the nodal swing dynamics through $-Br_C$, while $L_{CC}$ modifies the controller action within the weighted cut space.  Thus harmonic compression fixes only one of the four blocks; exact harmonic autonomy and stability of the complete tangent system require additional information.  Although $(\nu,r_C,h)$ is a closed output-level system, it need not be a state-equivalent realization of the complete tangent dynamics, because phase--actuator combinations may be invisible in $r$.  This is why the full reduced tangent generator still enters the recovery statement below.

\begin{corollary}[Finite-channel compression, exact harmonic decay, and full reduced-state recovery]
\label{cor:section15-harmonic-output}
Assume
\begin{equation}
        \rank(SH_*)=b_1,
        \qquad
        \rank(H_*^\top W_*^{-1}C)=b_1.
        \label{eq:section15-finite-channel-ranks}
\end{equation}
Then, for every $\kappa>0$, there exists $K\in\Lin(\R^{n_s},\R^{n_a})$ such that the harmonic block of Lemma~\ref{lem:section15-finite-output-block-form} satisfies
\begin{equation}
        L_{HH}
        =H_*^\top W_*^{-1}CKSH_*
        =-\kappa I_{b_1}.
        \label{eq:cor15-weighted-compression}
\end{equation}
In the minimal square case $n_a=n_s=b_1$, this gain is unique for fixed $S$ and $C$ and is given by
\begin{equation}
        K
        =
        -\kappa
        (H_*^\top W_*^{-1}C)^{-1}
        (SH_*)^{-1}.
        \label{eq:section15-square-gain}
\end{equation}
For every gain satisfying \eqref{eq:cor15-weighted-compression}, the harmonic coordinate obeys
\begin{equation}
        \tau\dot h
        =
        -\kappa h+L_{HC}r_C.
        \label{eq:cor15-harmonic-coordinate-with-coupling}
\end{equation}
Consequently, the following statements hold.

\begin{enumerate}[label=\textup{(\roman*)}]
\item The prescribed autonomous law
\begin{equation}
        h(t)=e^{-(\kappa/\tau)t}h(0)
        \label{eq:section15-exact-finite-harmonic-decay}
\end{equation}
holds for every tangent solution if and only if
\begin{equation}
        L_{HC}=0,
        \qquad\text{equivalently}\qquad
        H_*^\top W_*^{-1}CKSP_C^{W_*}=0.
        \label{eq:cor15-no-cut-to-harmonic}
\end{equation}
A sufficient condition is
\begin{equation}
        SP_C^{W_*}=0.
        \label{eq:section15-hodge-filtered-sensor}
\end{equation}

\item Quotient the phase deviation by the uniform phase-shift direction and restrict the actuator deviation to the invariant implemented subspace $\im(W_*^{-1}C)$.  Denote the resulting reduced tangent state by
\begin{equation}
        z_{\mathrm{red}}
        :=\bigl([\vartheta],\nu,w\bigr)
        \in
        \bigl(\R^n/\operatorname{span}\{\mathbf 1_V\}\bigr)
        \times\R^n\times\im(W_*^{-1}C),
        \label{eq:section15-reduced-tangent-state}
\end{equation}
and let $\mathcal A_{\mathrm{red}}(S,C,K)$ denote the matrix of its linear dynamics in any fixed coordinates.  If
\begin{equation}
        s\bigl(\mathcal A_{\mathrm{red}}(S,C,K)\bigr)<0,
        \label{eq:section15-hurwitz-condition}
\end{equation}
where $s(\cdot)$ is the spectral abscissa, then there exist constants $M_0\ge1$ and $\mu>0$ such that
\begin{equation}
        \norm{z_{\mathrm{red}}(t)}
        \le
        M_0e^{-\mu t}\norm{z_{\mathrm{red}}(0)}.
        \label{eq:section15-full-reduced-state-recovery}
\end{equation}
This is the precise meaning of \emph{full reduced-state recovery}: the phase returns to the operating point modulo a uniform shift, the frequency deviation vanishes, and the implemented actuator deviation returns to zero,
\begin{equation}
        \inf_{c\in\R}\norm{\vartheta(t)-c\mathbf 1_V}\longrightarrow0,
        \qquad
        \nu(t)\longrightarrow0,
        \qquad
        w(t)\longrightarrow0.
        \label{eq:section15-full-recovery-components}
\end{equation}
Consequently,
\begin{equation}
        r(t)\longrightarrow0,
        \qquad
        r_C(t)\longrightarrow0,
        \qquad
        h(t)\longrightarrow0.
        \label{eq:section15-stable-recovery}
\end{equation}
The harmonic output may nevertheless contain the cut-forced transient described by \eqref{eq:section15-harmonic-variation-of-constants} and need not follow \eqref{eq:section15-exact-finite-harmonic-decay} during the transient.
\end{enumerate}
\end{corollary}

\begin{proof}
Apply Lemma~\ref{lem:matrix-compression} with
\begin{equation}
        M_H:=H_*^\top W_*^{-1}C,
        \qquad
        N_H:=SH_*,
        \qquad
        Q=-\kappa I_{b_1}.
\end{equation}
The rank assumptions \eqref{eq:section15-finite-channel-ranks} therefore give a gain satisfying \eqref{eq:cor15-weighted-compression}; in the square case the same lemma gives the unique formula \eqref{eq:section15-square-gain}.  Equation \eqref{eq:cor15-harmonic-coordinate-with-coupling} follows immediately from Lemma~\ref{lem:section15-finite-output-block-form}.

If $L_{HC}=0$, equation \eqref{eq:section15-exact-finite-harmonic-decay} follows immediately.  Conversely, suppose $L_{HC}\ne0$.  Choose $r_C^0\in\im(W_*B^\top)$ with $L_{HC}r_C^0\ne0$ and an initial tangent state satisfying $r_C(0)=r_C^0$ and $h(0)=0$; for example, take $\vartheta(0)=0$, $\nu(0)=0$, and $w(0)=W_*^{-1}r_C^0$.  Then
\begin{equation}
        \tau\dot h(0)=L_{HC}r_C^0\ne0,
\end{equation}
whereas \eqref{eq:section15-exact-finite-harmonic-decay} would imply $h\equiv0$.  Thus \eqref{eq:cor15-no-cut-to-harmonic} is necessary.  Condition \eqref{eq:section15-hodge-filtered-sensor} is sufficient because it implies $CKSP_C^{W_*}=0$.

Finally, the implemented actuator subspace is invariant because the third equation of \eqref{eq:section15-linearized-finite-system} has range in $\im(W_*^{-1}C)$.  By the standard exponential-stability theorem for finite-dimensional linear systems \cite{AmannODE}, condition \eqref{eq:section15-hurwitz-condition} gives the estimate \eqref{eq:section15-full-reduced-state-recovery}.  The componentwise interpretation \eqref{eq:section15-full-recovery-components} follows from the definition of the quotient state, and the decay of $r$, $r_C$, and $h$ follows because they are linear outputs of that state.  The final assertion follows from \eqref{eq:section15-harmonic-variation-of-constants} with $L_{HH}=-\kappa I_{b_1}$.
\end{proof}

\begin{remark}[Diagnostic meaning of the finite-channel blocks]
\label{rem:section15-finite-channel-block-diagnostics}
Lemma~\ref{lem:section15-finite-output-block-form} provides a diagnostic template for the finite-channel examples below.  The block $L_{HH}$ tests whether the prescribed harmonic damping has been realized; $L_{HC}$ measures cut-to-harmonic forcing; $L_{CH}$ measures harmonic-to-cut spillover, which can enter the nodal swing dynamics; and $L_{CC}$ records how the controller modifies the transfer sector itself.  The ideal output operator $L_{\rm id}=-\kappa P_H^{W_*}$ has
\begin{equation}
        L_{CC}=L_{CH}=L_{HC}=0,
        \qquad
        L_{HH}=-\kappa I_{b_1}.
\end{equation}
A localized finite-channel realization may reproduce $L_{HH}$ exactly while retaining nonzero cross-blocks, explaining transfer, frequency, or harmonic transients despite exact harmonic compression.

Restricting the initial perturbation to the harmonic sector is not an adequate remedy.  Even if $r_C(0)=0$, the second equation in \eqref{eq:section15-coupled-output-system} may generate a cut component through $W_*B^\top\nu$ or $L_{CH}h$, which can then return to the harmonic coordinate through $L_{HC}$.  Exact prescribed harmonic decay for every perturbation therefore requires \eqref{eq:cor15-no-cut-to-harmonic}.
\end{remark}

\begin{remark}[Exact decoupling and practical stability design]
Condition \eqref{eq:section15-hodge-filtered-sensor} means that the sensor filters out the complete weighted cut component before the feedback acts.  A canonical choice is the distributed harmonic-coordinate sensor
\begin{equation}
        S_H=H_*^\top W_*^{-1}
        \in\Lin(X,\R^{b_1}),
\end{equation}
for which $S_HP_C^{W_*}=0$ and $S_HH_*=I$.  With respect to the $W_*^{-1}$ edge inner product and the Euclidean channel inner product, its metric adjoint is
\begin{equation}
        C=S_H^*=H_*\in\Lin(\R^{b_1},X).
\end{equation}
Together with $K=-\kappa I$, this gives $CKS_H=-\kappa P_H^{W_*}$.  Since multiplication by the positive diagonal matrix $W_*^{-1}$ does not change which rows of $H_*$ vanish, the physical edge support of both distributed channel maps is exactly the cyclic support
\begin{equation}
        E_{\rm cyc}=\operatorname{esupp}(\Ker B).
\end{equation}
They are therefore global Hodge patterns rather than local coordinate channels on only $b_1$ physical edges.  Complete cut--harmonic decoupling would additionally require $L_{CH}=0$.

For parsimonious local channels, the natural engineering objective is therefore not exact reproduction of the ideal trajectory, but the combination of the compression constraint \eqref{eq:cor15-weighted-compression} with a Hurwitz condition such as \eqref{eq:section15-hurwitz-condition}.  The latter gives full reduced-state recovery in the precise sense of \eqref{eq:section15-full-recovery-components}.  Any remaining placement or gain freedom may then be used to reduce $\norm{L_{HC}}$, $\norm{L_{CH}}$, gain size, and non-normal transient amplification.  The no-coupling conditions identify the theoretical limit and provide quantitative design targets even when they cannot be attained exactly.
\end{remark}

\begin{remark}[Meaning of the linearized statements and nonlinear stability]
Theorem~\ref{thm:section15-ideal-harmonic-output} establishes exponential stability of the output-level Hodge variables $(\nu,r_C,h)$ and shows that the ideal closure preserves the passive swing spectrum.  This is a statement about the closed output dynamics, not convergence of the full state variables to one distinguished representative $(\theta_*,0,u_*)$.  The full state retains the uniform phase-shift symmetry and, in the ideal model, neutral phase--actuator combinations that leave $B^\top\vartheta+w$ unchanged.

To invoke the finite-dimensional principle of linearized stability for a particular nonlinear representative, one must remove these neutral directions by fixing an appropriate local slice or quotient and specify the admissible actuator state space.  In such reduced coordinates, the Hurwitz derivative yields local exponential asymptotic stability of the corresponding isolated equilibrium by the standard principle of linearized stability \cite{AmannODE}.  For a finite-channel realization, the complete reduced Jacobian must still be verified to be Hurwitz; exact harmonic compression alone does not provide this conclusion.

The nonlinear energy theorem below addresses a different question.  With the design metric $\Sigma_\rho$ introduced there, and conditional on a trajectory remaining in a compact forward-invariant subset of the angle-cohesive region, it yields convergence to the largest invariant set on which
\begin{equation}
        \omega=0,
        \qquad
        P_H^{\Sigma_\rho}f(\eta)=0,
        \qquad
        Bf(\eta)=p.
\end{equation}
This is a regional convergence result in a prescribed cohesive branch, not a global stability assertion on the entire phase space.  Periodicity of the sine law, different winding sectors, and possible multiplicity of equilibria prevent such a global conclusion without substantially stronger hypotheses.
\end{remark}

\subsection{A nonlinear winding-one ring equilibrium with nonzero loop flow}
\label{subsec:winding-one-ring-equilibrium}
The ring is introduced not as a realistic transmission-grid benchmark,
but as the simplest nonlinear cycle-bearing configuration in which the
limitation of purely nodal swing stabilization can be exhibited exactly.
A winding-one phase profile can be angle-cohesive, synchronous, and
nodally balanced while carrying a nonzero harmonic line flow in
$\Ker B$.  Since this circulation satisfies $Bf=0$, it produces no
nodal accelerating imbalance; passive frequency damping is therefore
inactive and cannot remove it.  The example thus provides an explicit
nonlinear realization of the harmonic flow errors targeted by the
topological thermostat and prepares the exactly solvable controlled
trajectory of Subsection~\ref{subsec:exact-ring-trajectory}, where the
harmonic line-flow output is driven asymptotically to zero while the
bus-angle winding remains one.  The relevant contrast with the
preceding subsection is between the linearized and full nonlinear
systems.  In the passive linearized system with zero actuator deviation,
the weighted cut space $\mathcal C_*$ is an invariant output space, and
bus-angle dynamics cannot generate a harmonic line-flow component.  The
full nonlinear constitutive map
\begin{equation}
        \theta\longmapsto
        \operatorname{diag}(a)\sin(B^\top\theta)
\end{equation}
has no corresponding global cut-space invariance.  A finite bus-angle
configuration can therefore carry a nonzero harmonic line flow even
when $u=0$.  The winding-one equilibrium below gives the simplest exact
realization of this nonlinear phenomenon.

Accordingly, consider a single ring
\begin{equation}
        G=C_N.
\end{equation}
Its first Betti number is
\begin{equation}
        b_1(C_N)=1,
\end{equation}
so the harmonic edge-flow space is one-dimensional.  With all edges reference-oriented consistently around the ring, let
\begin{equation}
        h=\frac{1}{\sqrt N}(1,\ldots,1)\in\R^E
\end{equation}
be its normalized generator.  The scalar harmonic line-flow readout is
\begin{equation}
        \zeta(t)=h^\top f(\theta(t),u(t)).
\end{equation}

The graph ring should be distinguished from the winding of a phase state.  Let $\operatorname{pr}:\R\to[-\pi,\pi)$ be the principal representative defined in \eqref{eq:principal-representative}.  For a consistently oriented ring, define
\begin{equation}
        q(\theta)
        =
        \frac{1}{2\pi}
        \sum_{e\in E}
        \operatorname{pr}\bigl((B^\top\theta)_e\bigr).
        \label{eq:ring-winding-number}
\end{equation}
For phase configurations whose principal edge increments are chosen continuously around the ring, $q(\theta)$ is an integer and records how often the phases wind around the phase circle when the network ring is traversed once.

Consider zero injections,
\begin{equation}
        p=0,
\end{equation}
and the phase profile
\begin{equation}
        \theta_j^{(1)}=-\frac{2\pi j}{N},
        \qquad j=0,\ldots,N-1,
        \label{eq:ring-winding-one-profile}
\end{equation}
with
\begin{equation}
        \omega=0,
        \qquad
        u=0.
\end{equation}
For every edge, the principal representative of the oriented phase increment equals
\begin{equation}
        \operatorname{pr}\bigl((B^\top\theta^{(1)})_e\bigr)
        =\frac{2\pi}{N}.
\end{equation}
Thus the principal increments sum to $2\pi$ and
\begin{equation}
        q(\theta^{(1)})=1.
\end{equation}
The unreduced real differences still telescope to zero around the closed graph; the nonzero integer in \eqref{eq:ring-winding-number} arises from taking the phases modulo $2\pi$.

Assume that all line coefficients are equal to $a>0$.  Then every edge carries the same nonlinear active flow,
\begin{equation}
        f_e(\theta^{(1)},0)
        =a\sin\frac{2\pi}{N},
\end{equation}
and hence
\begin{equation}
        f(\theta^{(1)},0)
        =a\sin\frac{2\pi}{N}\,\mathbf 1_E
        \in\Ker B.
\end{equation}
This is the nonlinear mechanism that is absent from the linearized
model.  Although $B^\top\theta^{(1)}$ is an exact edge field in
$\im B^\top$ and its unreduced components telescope to zero, one
component differs from the others by $-2\pi$.  Periodicity of the
constitutive law identifies their sine values:
\begin{equation}
        \sin\left(\frac{2\pi}{N}-2\pi\right)
        =\sin\frac{2\pi}{N}.
\end{equation}
The componentwise nonlinear map therefore sends the exact edge-angle
field to the equal-flow vector
$a\sin(2\pi/N)\mathbf 1_E\in\Ker B$.  Thus the nonlinear line-flow map
need not preserve the cut space, whereas the linearized angle-to-flow
map $\vartheta\mapsto W_*B^\top\vartheta$ takes every bus-angle
deviation into the weighted cut space exactly.

Consequently,
\begin{equation}
        Bf(\theta^{(1)},0)=0.
\end{equation}
Substitution into the common swing equations of Subsection~\ref{subsec:nonlinear-swing-closures}, with the line setting held fixed by $\dot u=0$, gives
\begin{equation}
        \dot\theta=0,
        \qquad
        M\dot\omega=0.
\end{equation}
Therefore
\begin{equation}
        (\theta^{(1)},0,0)
\end{equation}
is a stationary solution of the nonlinear passive swing equations.  It is synchronized and nodally balanced, yet
\begin{equation}
        P_Hf(\theta^{(1)},0)
        =f(\theta^{(1)},0)\ne0.
        \label{eq:ring-protected-harmonic-flow}
\end{equation}
For $N>4$, the local phase increment $2\pi/N$ is smaller than $\pi/2$, so this winding-one state is also angle-cohesive in the sense of \eqref{eq:angle-cohesive-operating-point}.  Nonzero global winding is therefore compatible with locally small phase differences.

The preceding calculation makes this missing observable explicit.  Passive swing damping and any controller whose error signal is built only from frequency deviation or nodal imbalance see
\begin{equation}
        \omega=0,
        \qquad
        p-Bf(\theta^{(1)},0)=0.
\end{equation}
They therefore have no direct error signal identifying the stationary circulation.  The topological thermostat instead measures
\begin{equation}
        \zeta(0)=h^\top f(\theta^{(1)},0)\ne0.
\end{equation}
Since $b_1=1$, a single line-actuation direction with nonzero harmonic projection can reach the full harmonic sector.  If $\delta_{e_0}$ is the canonical coordinate vector of one edge, then
\begin{equation}
        h^\top\delta_{e_0}=\frac{1}{\sqrt N}\ne0.
\end{equation}
A schematic one-edge actuator may therefore use the measured signal $\zeta(t)$ to command an opposing phase-shift or line-flow bias.  This reachability observation is not the controller integrated in Subsection~\ref{subsec:exact-ring-trajectory}; the exactly solvable experiment there uses the full ideal projected line actuator.

\begin{remark}[Ring topology, winding, and harmonic flow]
Three circle-related notions occur in this example and should not be identified.  The ring $C_N$ is the topology of the physical graph and implies $b_1=1$, while every ring edge lies in the harmonic edge support:
\begin{equation}
        \operatorname{esupp}(\Ker B)=E.
\end{equation}
Thus the harmonic coordinate space is one-dimensional even though its ideal sensor and actuator patterns are distributed over the whole ring.  The integer $q(\theta)$ in \eqref{eq:ring-winding-number} is a property of a bus-phase configuration.  The vector $P_Hf(\theta,u)$ is the harmonic component of the nonlinear line-flow output.  The constitutive law
\begin{equation}
        f(\theta,u)=\operatorname{diag}(a)\sin(B^\top\theta+u)
\end{equation}
relates these objects, but they are mathematically distinct.

The special choice $p=0$ makes the cycle-free transfer target equal to zero:
\begin{equation}
        f_*=B^\top(BB^\top)^+p=0.
        \label{eq:ring-zero-injection-target}
\end{equation}
Hence every nonzero balanced flow in the present experiment is pure circulation, and suppressing its harmonic component removes the entire line flow.

For a balanced but nonzero injection pattern,
\begin{equation}
        p\ne0,
        \qquad
        \mathbf 1_V^\top p=0,
\end{equation}
the cycle-free representative
\begin{equation}
        f_*=B^\top(BB^\top)^+p
\end{equation}
is generally nonzero and supplies the generators and loads.  Since $\Ker B=\operatorname{span}\{h\}$, every edge flow with the same nodal balance has the form
\begin{equation}
        f=f_*+c h,
        \qquad c\in\R.
        \label{eq:ring-flow-affine-decomposition}
\end{equation}
The scalar $c$ is the additional uniform circulation.  At the graph-flow level, the selective thermostat objective is
\begin{equation}
        f_*+c(0)h\longrightarrow f_*,
\end{equation}
that is, to remove the harmonic contamination while preserving the nonzero source--load transfer.  In the nonlinear swing model this interpretation additionally requires that the target $f_*$ be compatible with the sinusoidal line law in an angle-cohesive branch.  The present zero-injection ring is chosen only to isolate the protected circulation mechanism in its simplest form.
\end{remark}

Coletta, Delabays, Adagideli and Jacquod describe topologically protected loop flows in high-voltage AC grids, and Delabays, Coletta and Jacquod analyze winding-number-labelled phase-locked states in single-loop oscillator networks \cite{ColettaLoopFlows,DelabaysColettaJacquodSingleLoop}.  The example above is read in that spirit: synchronization and nodal balance do not by themselves imply a cycle-free line-flow output.

\subsection{Universal Bregman balance and nonlinear dissipation}
\label{subsec:universal-bregman-dissipation}
\subsubsection*{Effective energy variables and proof strategy}

The ideal and finite-channel closures above are nonlinear ODEs in the full
state
\begin{equation}
        z=(\theta,\omega,u).
\end{equation}
Their branch constitutive law and potential depend on $\theta$ and $u$ only
through the effective edge phase $\eta=B^\top\theta+u$.  Define the
full-to-effective map
\begin{equation}
        \Pi_{\rm eff}(\theta,\omega,u)
        :=(B^\top\theta+u,\omega)
        =(\eta,\omega).
        \label{eq:full-to-effective-state-map}
\end{equation}
If $z(\cdot)$ is a classical solution of either nonlinear closure, then its
effective trajectory is $(\eta(t),\omega(t))=\Pi_{\rm eff}z(t)$, and
differentiating the definition of $\eta$ gives
\begin{equation}
        \dot\eta=B^\top\omega+\dot u.
        \label{eq:effective-edge-phase-kinematics}
\end{equation}
This is a kinematic identity inherited from $\eta=B^\top\theta+u$ and
$\dot\theta=\omega$, not an additional plant equation.

To obtain one energy identity for both feedback laws, temporarily open the
actuator loop and write
\begin{equation}
        v:=\dot u
\end{equation}
for an externally supplied edge-port input.  The corresponding effective
open-port system is
\begin{subequations}\label{eq:universal-line-actuated-swing-system}
\begin{align}
        \dot\eta&=B^\top\omega+v,\label{eq:universal-effective-kinematics}\\
        M\dot\omega+D\omega&=p-Bf(\eta).
\end{align}
\end{subequations}
Every trajectory of either complete nonlinear closure projects to a solution
of this system after its particular actuator law is substituted for $v$.

Let $(\eta_*,0)$ be a synchronous target with
\begin{equation}
        f_*:=f(\eta_*),
        \qquad
        p=Bf_*.
\end{equation}
Introduce the nonlinear potential
\begin{equation}
        U(\eta)=\sum_{e\in E}a_e(1-\cos\eta_e)
\end{equation}
and its Bregman distance from $\eta_*$,
\begin{equation}
        D_U(\eta,\eta_*)
        :=U(\eta)-U(\eta_*)
          -f_*^\top(\eta-\eta_*).
\end{equation}
The target-centred swing storage is
\begin{equation}\label{eq:nonlinear-bregman-energy}
        V(\eta,\omega)
        :=\frac12\omega^\top M\omega+D_U(\eta,\eta_*).
\end{equation}
It is defined on the effective state space.  Its pullback to the full state is
\begin{equation}
        \mathcal V(\theta,\omega,u)
        :=(V\circ\Pi_{\rm eff})(\theta,\omega,u)
        =V(B^\top\theta+u,\omega).
        \label{eq:full-state-bregman-pullback}
\end{equation}
Hence, along a full trajectory $z(t)$,
\begin{equation}
        \frac{\dd}{\dd t}\mathcal V(z(t))
        =\frac{\dd}{\dd t}V(\eta(t),\omega(t)).
        \label{eq:full-effective-storage-derivative}
\end{equation}
At this stage $V$ is a storage function.  It becomes a Lyapunov function for
a particular closure once convexity makes it positive definite in the
effective variables and the chosen feedback law makes its derivative
nonpositive.  Throughout this subsection, $\dot V$ denotes the derivative of
the scalar function $t\mapsto V(\eta(t),\omega(t))$ along a classical
trajectory.  We also use $\preceq$ for the L\"owner partial order on
self-adjoint linear operators: $A\preceq B$ means
$\langle z,(B-A)z\rangle\geq0$ for every $z$.  This notation is used below to
express strong-convexity Hessian bounds and negative semidefiniteness of
channel gains.

The argument now has three steps.  First, derive a controller-independent
balance for arbitrary $v$.  Second, close the edge port with the ideal or
finite-channel law, turning the port supply into a nonpositive quadratic
form.  Third, use convexity on the cohesive branch and LaSalle's invariance
principle to obtain convergence of the effective edge phase, frequency, and
line-flow output.  Since $\Pi_{\rm eff}$ is many-to-one, the corresponding
full-state conclusion concerns an equilibrium fibre rather than a uniquely
determined pair $(\theta,u)$.

\begin{lemma}[Universal Bregman balance]
\label{lem:universal-bregman-balance}
Let $I\subset\R$ be an interval, let $v:I\to X$ be continuous, and let
$(\eta,\omega)$ be a classical solution of
\eqref{eq:universal-line-actuated-swing-system}.  Then
\begin{equation}
        \boxed{
        \dot V
        =-\omega^\top D\omega
         +\bigl(f(\eta)-f_*\bigr)^\top v.}
        \label{eq:universal-bregman-balance}
\end{equation}
\end{lemma}

\begin{proof}
Since $\nabla U(\eta)=f(\eta)$ and $\eta_*$ is fixed, the first equation of
\eqref{eq:universal-line-actuated-swing-system} gives
\begin{equation}
        \frac{\dd}{\dd t}D_U(\eta,\eta_*)
        =\bigl(f(\eta)-f_*\bigr)^\top
         \bigl(B^\top\omega+v\bigr).
\end{equation}
Moreover, using $p=Bf_*$,
\begin{align}
        \frac{\dd}{\dd t}\frac12\omega^\top M\omega
        &=\omega^\top\bigl(p-D\omega-Bf(\eta)\bigr)\notag\\
        &=-\omega^\top D\omega
          -\omega^\top B\bigl(f(\eta)-f_*\bigr).
\end{align}
Adding the two identities cancels the incidence terms because
\begin{equation}
        \bigl(f(\eta)-f_*\bigr)^\top B^\top\omega
        =\omega^\top B\bigl(f(\eta)-f_*\bigr),
\end{equation}
which proves \eqref{eq:universal-bregman-balance}.
\end{proof}

\begin{remark}[Equilibrium-relative passivity interpretation]
On a cohesive region where the storage is nonnegative,
\eqref{eq:universal-bregman-balance} shows that the effective swing system is
passive relative to the fixed target from the edge-port input $v$ to the
line-flow-error output
\begin{equation}
        e_f:=f(\eta)-f_*,
\end{equation}
with supply rate
\begin{equation}
        s(v,e_f)=e_f^\top v.
\end{equation}
For either complete closed system, $v=\dot u$.  This is an
equilibrium-relative statement, not a claim of a general two-trajectory
incremental-passivity property.  A negative feedback law is dissipative when
the actuator direction is the metric dual of the measured flow signal,
because the supply term then becomes a negative quadratic form.  The Hodge
decomposition selects which component of the flow error is regulated, while
colocation or an equivalent metric-duality condition supplies the sign.  This
explains why the ideal projected controller and the canonical colocated finite
realization admit exact dissipation identities, whereas harmonic compression
alone does not determine the sign of the complete Bregman supply term.
\end{remark}

\subsubsection*{Ideal metric-compatible closure}

Let $\Sigma_\rho\in\Lin(X)$ be a positive diagonal edge-weight operator and let $P_H^{\Sigma_\rho}\in\Lin(X)$ denote the corresponding weighted Hodge projection onto $\Ker B$, orthogonal with respect to
\begin{equation}
        \langle x,y\rangle_{\Sigma_\rho^{-1}}
        =x^\top\Sigma_\rho^{-1}y.
\end{equation}
This capacity or design metric should be distinguished from the operating-point sensitivity matrix $W_*$ introduced in the swing linearization: $W_*$ comes from differentiating the nonlinear line-flow law, whereas $\Sigma_\rho$ specifies the weighted edge metric used to define $P_H^{\Sigma_\rho}$.  Assume that the target is harmonic-free in this metric,
\begin{equation}
        P_H^{\Sigma_\rho}f_*=0.
\end{equation}
Closing the edge port with the ideal metric-compatible feedback gives
\begin{equation}
        \tau\dot u
        =-\kappa\Sigma_\rho^{-1}P_H^{\Sigma_\rho}
          \bigl(f(\eta)-f_*\bigr),
        \qquad \kappa,\tau>0.
        \label{eq:ideal-nonlinear-bregman-actuator}
\end{equation}
Since $P_H^{\Sigma_\rho}f_*=0$, the target-centred law is equivalently
\begin{equation}
        \tau\dot u
        =-\kappa\Sigma_\rho^{-1}P_H^{\Sigma_\rho}f(\eta).
\end{equation}
For $\Sigma_\rho=W_*$, it coincides with the ideal nonlinear closure introduced above.  In the effective edge-phase variables, the closed system is
\begin{subequations}\label{eq:nonlinear-lyapunov-system}
\begin{align}
        \dot\eta
        &=B^\top\omega-\frac{\kappa}{\tau}
          \Sigma_\rho^{-1}P_H^{\Sigma_\rho}\bigl(f(\eta)-f_*\bigr),\\
        M\dot\omega+D\omega
        &=p-Bf(\eta).
\end{align}
\end{subequations}

The positive-definiteness property required below is another standard consequence of convexity; we record the quantitative form used in the proof.

\begin{lemma}[Bregman bounds on a cohesive set]
\label{lem:bregman-coercivity}
Let $\mathcal O\subset\R^m$ be open and convex, let $U\in C^2(\mathcal O)$, and let $K\subset\mathcal O$ be compact and convex.  Assume that there exist constants $0<\mu\le L$ such that
\begin{equation}
        \mu I
        \preceq
        \nabla^2U(\xi)
        \preceq
        LI,
        \qquad \xi\in K.
        \label{eq:hessian-bounds}
\end{equation}
Then, for all $x,y\in K$, the Bregman distance
\begin{equation}
        D_U(x,y)
        :=U(x)-U(y)-\nabla U(y)^\top(x-y)
\end{equation}
satisfies
\begin{equation}
        \frac\mu2\norm{x-y}^2
        \le
        D_U(x,y)
        \le
        \frac L2\norm{x-y}^2.
        \label{eq:bregman-quadratic-bounds}
\end{equation}
\end{lemma}

\begin{proof}
Put $d=x-y$ and define $g:[0,1]\to\R$ by
\begin{equation}
        g(s)=U(y+sd).
\end{equation}
Since $K$ is convex, the segment $y+sd$ lies in $K$.  The fundamental theorem of calculus gives
\begin{align}
        D_U(x,y)
        &=g(1)-g(0)-g'(0)
        \notag\\
        &=\int_0^1\bigl(g'(s)-g'(0)\bigr)\,\dd s
        \notag\\
        &=\int_0^1\int_0^s g''(t)\,\dd t\,\dd s
        \notag\\
        &=\int_0^1(1-t)g''(t)\,\dd t.
\end{align}
Moreover,
\begin{equation}
        g''(t)
        =d^\top\nabla^2U(y+td)d.
\end{equation}
The Hessian bounds \eqref{eq:hessian-bounds} therefore imply
\begin{equation}
        \mu\norm d^2
        \le g''(t)
        \le L\norm d^2.
\end{equation}
Integration and $\int_0^1(1-t)\,\dd t=1/2$ yield \eqref{eq:bregman-quadratic-bounds}.
\end{proof}

For the network potential
\begin{equation}
        U(\eta)=\sum_{e\in E}a_e(1-\cos\eta_e),
\end{equation}
one has
\begin{equation}
        \nabla^2U(\eta)
        =\operatorname{diag}(a_e\cos\eta_e).
\end{equation}
Hence every compact convex subset of the cohesive region $\mathcal D$ satisfies the hypotheses of Lemma~\ref{lem:bregman-coercivity}.

\begin{theorem}[Nonlinear edge-energy dissipation]
\label{thm:nonlinear-edge-energy-dissipation}
Assume that $(\eta_*,0)$ is angle-cohesive and that a solution of \eqref{eq:nonlinear-lyapunov-system} remains in a compact forward-invariant subset of
\begin{equation}
        \mathcal D\times\R^n,
        \qquad
        \mathcal D
        =\{\eta\in\R^m:\ |\eta_e|<\pi/2\text{ for all }e\in E\}.
\end{equation}
Then the Bregman swing energy \eqref{eq:nonlinear-bregman-energy} satisfies
\begin{equation}\label{eq:nonlinear-lyapunov-identity}
        \dot V
        =-\omega^\top D\omega
         -\frac{\kappa}{\tau}
          \norm{P_H^{\Sigma_\rho}\bigl(f(\eta)-f_*\bigr)}_{\Sigma_\rho^{-1}}^2
        \le0.
\end{equation}
Consequently, every such trajectory converges to the largest invariant set on the effective state space, on which
\begin{subequations}\label{eq:nonlinear-limit-set}
\begin{align}
        \omega&=0,
        \label{eq:nonlinear-limit-frequency}\\
        P_H^{\Sigma_\rho}\bigl(f(\eta)-f_*\bigr)&=0.
        \label{eq:nonlinear-limit-harmonic}
\end{align}
\end{subequations}
On this invariant set one also has
\begin{equation}
        Bf(\eta)=p.
        \label{eq:nonlinear-limit-balance}
\end{equation}
Equations~\eqref{eq:nonlinear-limit-harmonic} and \eqref{eq:nonlinear-limit-balance}, together with $Bf_*=p$, imply $f(\eta)=f_*$.  Since the componentwise sine law is injective on $\mathcal D$, the invariant set in $(\eta,\omega)$-coordinates is the single point $(\eta_*,0)$.  Hence
\begin{subequations}\label{eq:nonlinear-output-convergence}
\begin{align}
        f(\eta(t))&\longrightarrow f_*,
        \label{eq:nonlinear-flow-convergence}\\
        \eta(t)&\longrightarrow\eta_*,
        \label{eq:nonlinear-edge-phase-convergence}\\
        \omega(t)&\longrightarrow0.
        \label{eq:nonlinear-frequency-convergence}
\end{align}
\end{subequations}
These conclusions concern the effective edge phase and the line-flow output.  Without additional constraints they do not determine $\theta$ and $u$ separately: limiting pairs lie in the fibre $B^\top\theta+u=\eta_*$, and the uniform phase shift remains a gauge symmetry.
\end{theorem}

\begin{proof}
The edge-phase projection of the assumed forward-invariant set is compact and stays a positive distance from the boundary of the convex box $\mathcal D$.  Hence its convex hull together with $\eta_*$ is contained in a compact convex set $K\subset\mathcal D$.  Lemma~\ref{lem:bregman-coercivity}, together with positive definiteness of $M$, shows that $V$ is positive definite with respect to $(\eta,\omega)=(\eta_*,0)$ on $K\times\R^n$.  Closing the port by setting $v=\dot u$ according to \eqref{eq:ideal-nonlinear-bregman-actuator} in the universal balance \eqref{eq:universal-bregman-balance} gives
\begin{align}
        \dot V
        &=-\omega^\top D\omega
          -\frac{\kappa}{\tau}
          \bigl(f(\eta)-f_*\bigr)^\top
          \Sigma_\rho^{-1}P_H^{\Sigma_\rho}
          \bigl(f(\eta)-f_*\bigr)\notag\\
        &=-\omega^\top D\omega
          -\frac{\kappa}{\tau}
          \norm{P_H^{\Sigma_\rho}\bigl(f(\eta)-f_*\bigr)}_{\Sigma_\rho^{-1}}^2,
\end{align}
where the weighted projection identity \eqref{eq:weighted-projection-energy-identity} was used in the last step.  This proves \eqref{eq:nonlinear-lyapunov-identity}.

LaSalle's invariance principle \cite{KhalilNonlinear} gives convergence to the largest invariant subset of $\{\dot V=0\}$.  Equations~\eqref{eq:nonlinear-limit-frequency} and \eqref{eq:nonlinear-limit-harmonic} hold there.  Invariance also requires $\dot\omega=0$, and hence \eqref{eq:nonlinear-limit-balance}.  Put $e_f=f(\eta)-f_*$.  Equation~\eqref{eq:nonlinear-limit-balance} and the target identity $Bf_*=p$ give $Be_f=0$, so $e_f\in\Ker B$.  Since $P_H^{\Sigma_\rho}$ is the projection onto $\Ker B$, one has $e_f=P_H^{\Sigma_\rho}e_f$.  Equation~\eqref{eq:nonlinear-limit-harmonic} therefore yields $e_f=0$.  Finally, $f_e(\eta_e)=a_e\sin\eta_e$ is strictly increasing on $(-\pi/2,\pi/2)$, so $f(\eta)=f_*$ implies $\eta=\eta_*$.  This proves \eqref{eq:nonlinear-output-convergence}.
\end{proof}

\begin{remark}[A sufficient cohesive energy bound]
The compact forward-invariance hypothesis can be verified by the standard Lyapunov sublevel-set argument.  Let $\Omega\subset\mathcal D$ be bounded and open, with $\eta_*\in\Omega$ and $\overline\Omega\subset\mathcal D$, and put
\begin{equation}
        c_\Omega
        :=\min_{\eta\in\partial\Omega}D_U(\eta,\eta_*).
        \label{eq:cohesive-boundary-energy}
\end{equation}
Strict convexity of $U$ on $\mathcal D$ gives $c_\Omega>0$.  If $\eta(0)\in\Omega$ and
\begin{equation}
        V(\eta(0),\omega(0))<c_\Omega,
        \label{eq:cohesive-energy-certificate}
\end{equation}
then monotonicity of $V$ prevents the trajectory from reaching $\partial\Omega$.  Hence the sublevel set
\begin{equation}
        \left\{
        (\eta,\omega)\in\overline\Omega\times\R^n:
        V(\eta,\omega)\le V(\eta(0),\omega(0))
        \right\}
        \label{eq:cohesive-invariant-sublevel}
\end{equation}
is forward invariant and compact: $\eta$ is confined to $\overline\Omega$, while positive definiteness of $M$ bounds $\omega$.  Thus \eqref{eq:cohesive-energy-certificate} is a directly checkable sufficient condition for the hypothesis of Theorem~\ref{thm:nonlinear-edge-energy-dissipation}.  The same certificate applies to the colocated finite-channel law below, for which $V$ is also nonincreasing.  Failure of the inequality makes this particular certificate inconclusive; it does not imply loss of cohesiveness.
\end{remark}

\begin{remark}[Synchronization and topological selectivity]
The theorem is not a phase-consensus statement.  It describes a phase-locked transfer equilibrium: the frequencies synchronize, the nodal transfer balance is satisfied, and the weighted harmonic line-flow error vanishes.  The actuator contributes a negative square only through the selected harmonic output; the useful cut-space transfer does not enter the additional ideal dissipation term.
\end{remark}

\subsubsection*{Colocated finite-channel closure}

\begin{corollary}[Dissipation of colocated finite-channel feedback]
\label{cor:finite-channel-bregman-dissipation}
Let $S\in\Lin(X,\R^{n_s})$ be an edge-flow sensor map and use its dual colocated actuator
\begin{equation}
        C=S^*\in\Lin(\R^{n_s},X).
        \label{eq:colocated-adjoint-actuator}
\end{equation}
With the Euclidean pairings used here, the coordinate matrix of $S^*$ is $S^\top$.  If $K=K^*\preceq0$ and
\begin{equation}
        \tau\dot u=S^*KS\bigl(f(\eta)-f_*\bigr),
        \label{eq:general-colocated-finite-bregman-law}
\end{equation}
then
\begin{equation}
        \dot V
        =-\omega^\top D\omega
         +\frac1\tau
          \bigl(S(f(\eta)-f_*)\bigr)^\top
          K\bigl(S(f(\eta)-f_*)\bigr)
        \le0.
        \label{eq:general-colocated-finite-bregman-dissipation}
\end{equation}

More specifically, let $H_{\Sigma_\rho}$ have $\Sigma_\rho^{-1}$-orthonormal columns spanning $\Ker B$, assume $n_s=b_1$ and that
\begin{equation}
        G:=SH_{\Sigma_\rho}
\end{equation}
is invertible, and choose the canonical gain
\begin{equation}
        K=-\kappa(GG^\top)^{-1}.
\end{equation}
Then
\begin{equation}
        \boxed{
        \dot V
        =-\omega^\top D\omega
         -\frac{\kappa}{\tau}
          \norm{G^{-1}S\bigl(f(\eta)-f_*\bigr)}^2
        \le0.}
        \label{eq:finite-channel-nonlinear-dissipation}
\end{equation}
If the trajectory remains in a compact forward-invariant subset of the cohesive region, then its line-flow output converges to $f_*$.
\end{corollary}

\begin{proof}
Equation~\eqref{eq:general-colocated-finite-bregman-dissipation} follows by closing the port with $v=\dot u$ from \eqref{eq:general-colocated-finite-bregman-law} in \eqref{eq:universal-bregman-balance} and using the adjoint identity
\begin{equation}
        \langle e_f,S^*KSe_f\rangle_X
        =\langle Se_f,KSe_f\rangle_{\R^{n_s}}.
\end{equation}  For the canonical gain,
\begin{equation}
        (Se_f)^\top K(Se_f)
        =-\kappa\norm{G^{-1}Se_f}^2,
        \qquad e_f=f(\eta)-f_*,
\end{equation}
which gives \eqref{eq:finite-channel-nonlinear-dissipation}.  By LaSalle's invariance principle \cite{KhalilNonlinear}, on the largest invariant subset of $\{\dot V=0\}$ one has $\omega=0$ and $Se_f=0$.  Invariance gives $Be_f=0$, so $e_f\in\Ker B$ and may be written $e_f=H_{\Sigma_\rho} c$.  Since $0=Se_f=Gc$ and $G$ is invertible, $c=0$ and hence $f(\eta)=f_*$.  The conclusion concerns the line-flow output and does not, without further assumptions, determine $\theta$ and $u$ separately.
\end{proof}

\subsubsection*{General finite-channel dissipativity and noncolocation}

\begin{remark}[Dissipativity, selectivity, and noncolocation]
The following observation is standard in finite-dimensional control and matrix analysis; we record it because the same algebra is used here in the Hodge, sensor--actuator, and power-system formulations.  For $L\in\Lin(X)$, define
\begin{equation}
        \operatorname{sym}(L)
        :=\frac12(L+L^*).
\end{equation}
The skew-adjoint part contributes nothing to a real quadratic form, and hence
\begin{equation}
        \langle x,Lx\rangle_X
        =\langle x,\operatorname{sym}(L)x\rangle_X.
        \label{eq:symmetric-part-quadratic-form}
\end{equation}
Consequently, $L$ contributes nonpositive power for every $x$ precisely when $\operatorname{sym}(L)\preceq0$.  In the colocated case $C=S^*$,
\begin{equation}
        \operatorname{sym}(CKS)
        =S^*\operatorname{sym}(K)S,
        \label{eq:colocated-symmetric-part}
\end{equation}
so $K=K^*\preceq0$ implies $CKS=S^*KS\preceq0$.  The resulting edge-space operator is generally only semidefinite because vectors in $\Ker S$ are left undamped.  For noncolocated $C$, the sign of $K$ alone gives no corresponding conclusion.

Closing the edge port with a general finite realization
\begin{equation}
        \tau\dot u=CKS\bigl(f(\eta)-f_*\bigr)
\end{equation}
gives
\begin{equation}
        \dot V
        =-\omega^\top D\omega
         +\frac1\tau e_f^\top\operatorname{sym}(CKS)e_f.
        \label{eq:general-finite-bregman-supply}
\end{equation}
Thus exact harmonic compression, which constrains only the harmonic-to-harmonic block of $CKS$, does not imply dissipation of the complete Bregman storage.  A noncolocated realization may still be dissipative if $\operatorname{sym}(CKS)\preceq0$, but this is an additional full-operator design condition rather than a consequence of the Betti-number rank criterion.

Dissipativity and Hodge selectivity are also distinct.  Target-subtracted global damping is dissipative but acts on both cut and harmonic errors.  The ideal projector is both dissipative and perfectly selective.  The canonical colocated finite controller has exact harmonic compression and a decreasing Bregman energy, but its selected edge signals generally contain cut components; transfer and frequency transients may therefore occur while the total target-centred storage decreases.  Bregman monotonicity does not imply that each Hodge component is monotone.
\end{remark}

\subsection{An exactly solvable nonlinear controlled ring trajectory}
\label{subsec:exact-ring-trajectory}

\subsubsection*{Physical initial state and control objective}

We return to the winding-one equilibrium of
Subsection~\ref{subsec:winding-one-ring-equilibrium} and apply the full
ideal projected line actuator.  The condition $p=0$ means that no bus
injects or withdraws external power; it does not imply that every line
flow vanishes.  At a stationary state with $\omega=0$, the nodal swing
balance requires only
\begin{equation}
        Bf=0.
\end{equation}
A nonzero flow in $\Ker B$ may therefore circulate around the ring,
entering and leaving every bus in equal amounts.

The bus phases in $\theta^{(1)}$ are not all equal.  Rather, the
principal phase difference on every reference-oriented edge is
\begin{equation}
        \delta=\frac{2\pi}{N},
\end{equation}
so their sum around the ring is $2\pi$ and the bus-phase winding is one.
With zero initial actuator offset and equal unit line coefficients, the
corresponding line flow is
\begin{equation}
        f(\theta^{(1)},0)
        =\sin\delta\,\mathbf 1_E
        \in\Ker B\setminus\{0\}.
\end{equation}
Thus the initial state is a phase-locked winding-one configuration
carrying a uniform pure active-power circulation.  Since the prescribed
target is also required to be cycle-free,
\begin{equation}
        Bf_*=p=0,
        \qquad
        P_Hf_*=0
\end{equation}
imply
\begin{equation}
        f_*\in\Ker B\cap\im B^\top=\{0\},
        \qquad
        f_*=0,
\end{equation}
consistently with \eqref{eq:ring-zero-injection-target}.  The control
objective is therefore to remove the circulating line flow while
leaving the bus-phase configuration, and hence its winding number,
unchanged.

The ring has
\begin{equation}
        N=8,
        \qquad
        p=0,
        \qquad
        M=I_8,
        \qquad
        D=0.8I_8,
\end{equation}
and all line coefficients are equal to one.  The common nonlinear swing
equations are
\begin{subequations}\label{eq:ring-full-nonlinear-system}
\begin{align}
        \dot\theta&=\omega,\\
        \dot\omega&=-0.8\omega-Bf(\theta,u),\\
        f(\theta,u)&=\sin(B^\top\theta+u),
\end{align}
\end{subequations}
where the sine is componentwise.  The initial state is
\begin{equation}
        \theta(0)=\theta^{(1)},
        \qquad
        \omega(0)=0,
        \qquad
        u(0)=0,
        \label{eq:ring-controlled-initial-state}
\end{equation}
with $\theta^{(1)}$ given by \eqref{eq:ring-winding-one-profile}.

We compare three line-actuator laws.  Passive swing damping leaves the
line setting fixed,
\begin{equation}
        \dot u=0.
\end{equation}
The global and topological line controllers are, respectively,
\begin{equation}
        \dot u
        =
        -\frac{\kappa}{\tau}W_0^{-1}f(\theta,u),
        \label{eq:ring-global-line-law}
\end{equation}
\begin{equation}
        \dot u
        =
        -\frac{\kappa}{\tau}W_0^{-1}P_Hf(\theta,u),
        \label{eq:ring-topological-line-law}
\end{equation}
where
\begin{equation}
        \delta=\frac{2\pi}{N}=\frac{\pi}{4},
        \qquad
        W_0=\cos\delta\,I_8,
        \qquad
        \kappa=1.3,
        \qquad
        \tau=0.7.
\end{equation}
Equation~\eqref{eq:ring-topological-line-law} is the full ideal projected
actuator, not the schematic one-edge reachability construction discussed
in Subsection~\ref{subsec:winding-one-ring-equilibrium}.  Since
$P_H=hh^\top$ on the ring, it factors through one scalar harmonic sensor
and one scalar actuator command, but the sensor $h^\top$ and actuator
pattern $h$ are distributed over all edges; see
Remark~\ref{rem:minimal-rank-local-realization}.

\subsubsection*{Invariant scalar dynamics}

The trajectories can be obtained without a numerical ODE solver.  In
the full state space of $(\theta,\omega,u)$, define the one-dimensional
affine line
\begin{equation}
        \mathcal M_1
        :=
        \left\{
        \bigl(\theta^{(1)},0,s\mathbf 1_E\bigr):s\in\R
        \right\}.
        \label{eq:ring-winding-one-invariant-manifold}
\end{equation}
For a state in $\mathcal M_1$, put
\begin{equation}
        \psi:=\delta+s,
\end{equation}
and call $\psi$ the common branch phase, understood modulo $2\pi$.
Every component of
\begin{equation}
        \eta=B^\top\theta^{(1)}+s\mathbf 1_E
\end{equation}
is congruent to $\psi$ modulo $2\pi$.  Therefore
\begin{equation}
        f(\theta^{(1)},s\mathbf 1_E)
        =
        \sin\psi\,\mathbf 1_E
        \in\Ker B.
        \label{eq:ring-symmetric-flow}
\end{equation}
At every point of $\mathcal M_1$, one has $\omega=0$ and $Bf=0$, so
\eqref{eq:ring-full-nonlinear-system} gives $\dot\theta=0$ and
$\dot\omega=0$.  The passive actuator law gives $\dot u=0$, whereas
each active actuator law produces a vector proportional to
$\mathbf 1_E$.  Hence all three vector fields are tangent to
$\mathcal M_1$, and this affine line is invariant.  Since the initial
state \eqref{eq:ring-controlled-initial-state} belongs to
$\mathcal M_1$, uniqueness yields
\begin{equation}
        \theta(t)=\theta^{(1)},
        \qquad
        \omega(t)=0,
        \qquad
        u(t)=s(t)\mathbf 1_E,
        \qquad
        s(0)=0.
        \label{eq:ring-symmetric-invariant-trajectory}
\end{equation}
The nodal inertia remains part of the plant; it is simply not excited
on this affine line because the line flow is divergence-free and
produces no nodal accelerating imbalance.  This exact freezing relies
on the symmetric ring, equal line coefficients, zero injections, zero
initial frequency, and, in the active cases, the distributed actuator
pattern; it is not a generic property of local or asymmetric line
control.

The passive and active nonlinear systems have different dynamics on the
same invariant affine line.  Under passive swing damping, $\dot s=0$,
so every point of $\mathcal M_1$ is stationary.  Under either active
line controller, $P_Hf=f$ on $\mathcal M_1$, so the global and
topological laws
\eqref{eq:ring-global-line-law}--\eqref{eq:ring-topological-line-law}
coincide exactly.  Writing
\begin{equation}
        \psi(t)=\delta+s(t),
\end{equation}
the actuator equation reduces to
\begin{equation}
        \dot\psi=-\gamma\sin\psi,
        \qquad
        \gamma=\frac{\kappa}{\tau\cos\delta},
        \qquad
        \psi(0)=\delta.
        \label{eq:ring-scalar-controlled-equation}
\end{equation}
Separating variables gives the closed formula
\begin{equation}
        \tan\frac{\psi(t)}{2}
        =
        \tan\frac{\delta}{2}\,e^{-\gamma t}.
        \label{eq:ring-scalar-exact-solution}
\end{equation}
Thus
\begin{equation}
        \psi(t)\longrightarrow0,
        \qquad
        s(t)\longrightarrow-\delta,
        \qquad
        f(\theta(t),u(t))\longrightarrow0.
\end{equation}
The limiting controlled state is therefore
\begin{equation}
        \theta_\infty=\theta^{(1)},
        \qquad
        \omega_\infty=0,
        \qquad
        u_\infty=-\frac{\pi}{4}\mathbf 1_E,
        \label{eq:ring-controlled-limit-state}
\end{equation}
up to a uniform bus-phase shift and equivalent $2\pi$ representatives
of the edge phases.  The controller does not unwind the bus phases.
Instead, the uniform limiting actuator offset cancels the bus-generated
phase difference inside the branch constitutive law.  Accordingly,
\begin{equation}
        q(\theta(t))=1
\end{equation}
for all $t$, while the effective line-flow circulation is driven to
zero.  This illustrates why a cycle-free line-flow target need not have
zero bus-angle winding when an independent line-actuator offset is
present.

The same mechanism has a direct storage interpretation.  Along
$\mathcal M_1$, the kinetic term in the target-centred storage of
Subsection~\ref{subsec:universal-bregman-dissipation} vanishes.  With unit line coefficients and the
zero-flow target, the line-potential contribution is
\begin{equation}
        V(t)=N\bigl(1-\cos\psi(t)\bigr).
\end{equation}
Consequently,
\begin{equation}
        \dot V(t)
        =f(\theta(t),u(t))^\top\dot u(t)
        =-N\gamma\sin^2\psi(t)
        \leq0.
\end{equation}
Thus, in the idealized model, the line-actuator port removes the stored
target-centred line-potential energy while the nodal kinetic energy
remains identically zero.

The quadratic cycle-flow diagnostic plotted below is
\begin{equation}
        E_H(t)
        =
        \norm{P_Hf(\theta(t),u(t))}^2
        =
        N\sin^2\psi(t).
        \label{eq:ring-exact-cycle-energy}
\end{equation}
For the active laws it decreases monotonically, because
\begin{equation}
        \dot E_H(t)
        =
        -2N\gamma\sin^2\psi(t)\cos\psi(t)<0
\end{equation}
for $0<\psi(t)\le\delta<\pi/2$.  The passive trajectory has
$\psi(t)=\delta$ and $E_H(t)=4$ for all time.

\begin{figure}[htbp]
\centering
\begin{subfigure}{0.48\textwidth}
\centering
\includegraphics[width=\textwidth]{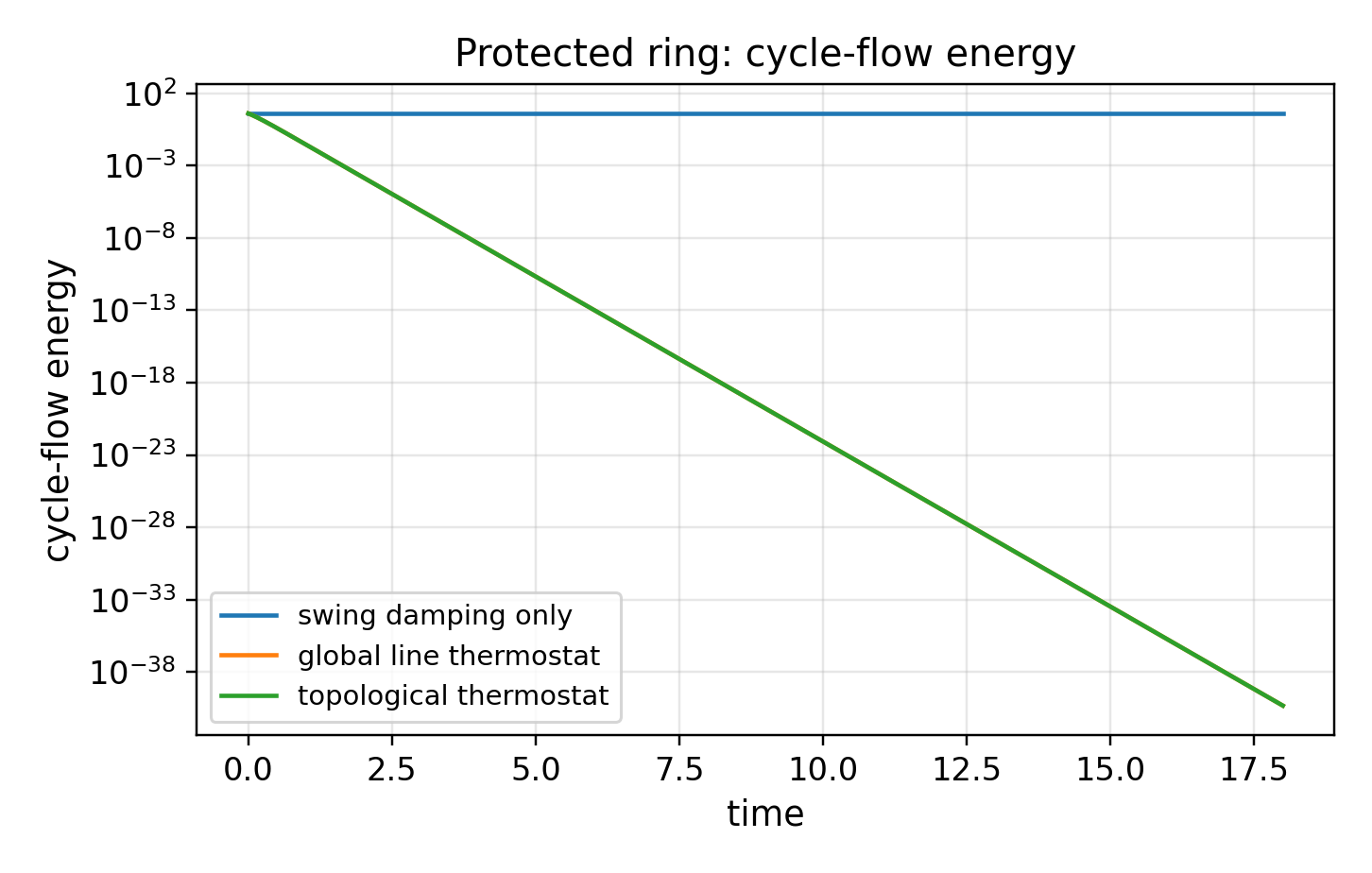}
\caption{Cycle-flow energy.}
\end{subfigure}\hfill
\begin{subfigure}{0.48\textwidth}
\centering
\includegraphics[width=\textwidth]{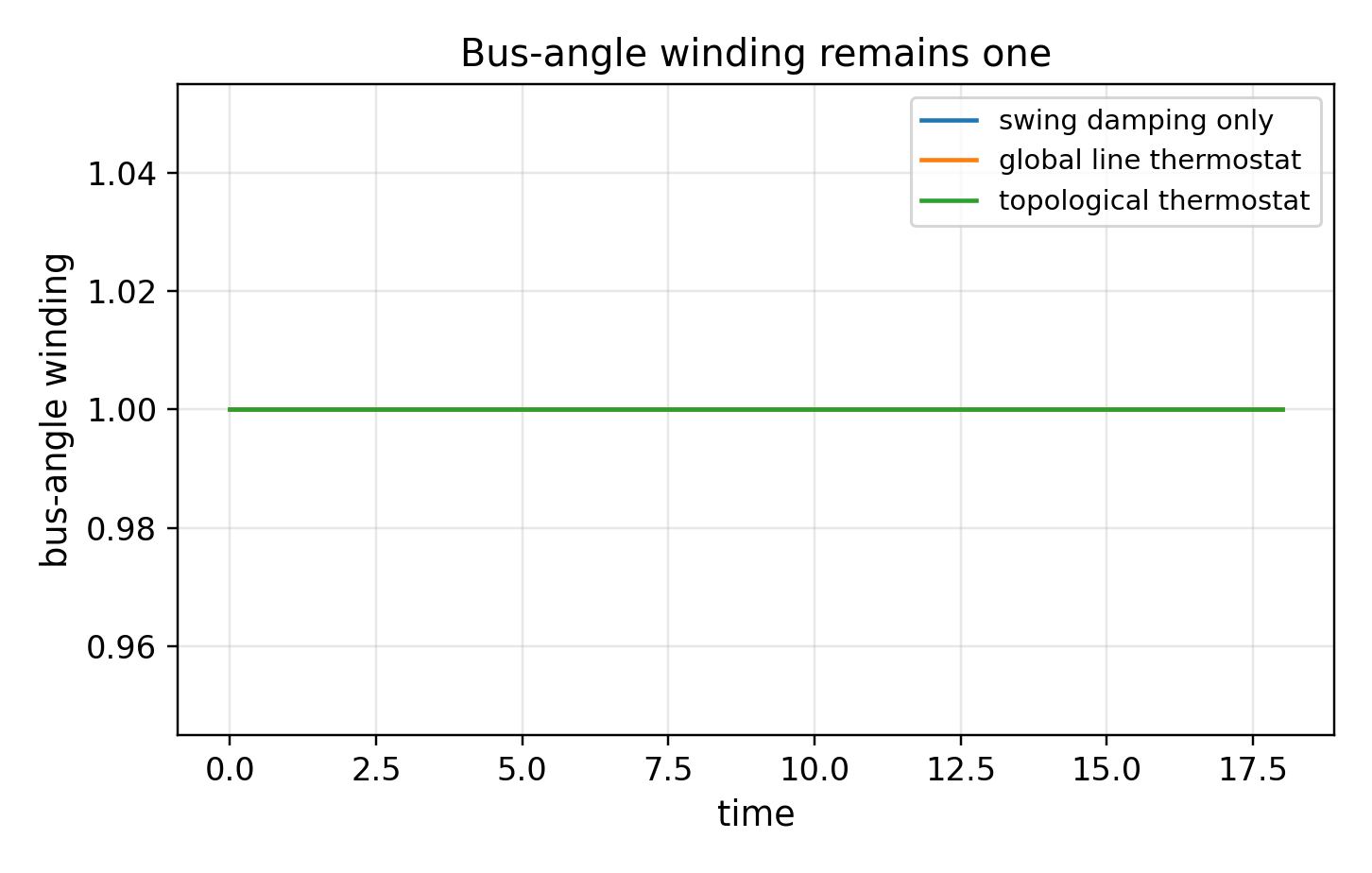}
\caption{Bus-angle winding.}
\end{subfigure}
\caption{Exact nonlinear controlled trajectory on an eight-bus ring.  The global and topological controllers coincide on the symmetric invariant trajectory: cycle-flow energy decays while the bus-angle winding remains one.}
\label{fig:swing-ring}
\end{figure}

\begin{table}[H]
\centering
\small
\resizebox{\textwidth}{!}{%
\begin{tabular}{lrrrr}
\toprule
Controller & Final cycle energy & Integrated cycle energy & Reduction vs. passive & Max frequency norm\\
\midrule
Swing damping only & 4.000 & 72.000 & 0.0\% & 0.000\\
Global line thermostat & $4.75\cdot10^{-41}$ & 0.892 & 98.76\% & 0.000\\
Topological thermostat & $4.75\cdot10^{-41}$ & 0.892 & 98.76\% & 0.000\\
\bottomrule
\end{tabular}%
}
\caption{Metrics for the exact nonlinear ring trajectory over $0\le t\le18$.  The global and topological line controllers coincide on the symmetric invariant trajectory, while the bus-frequency state remains identically zero.  The tiny terminal cycle energy is the value of the analytic formula, not a numerical solver residual.}
\label{tab:swing-ring-metrics}
\end{table}

\section{Topological line-flow control on the IEEE 14-bus network}
\label{sec:ieee14-application}
The preparatory discussion in Section~\ref{sec:transmission-preparation} motivates the transmission-grid setting.  We now specify the IEEE 14 topology, branch data, operating point, and fixed seven-channel architecture, verify exact harmonic compression of the localized controller, and compare the ideal and finite-channel nonlinear swing responses.

The graph has $n=14$ buses, $m=20$ branches, and $b_1=7$.  For the localized comparisons we use the ordered seven-edge architecture
\begin{equation}
        R=\{1,3,7,8,11,12,20\}.
        \label{eq:ieee14-fixed-channel-list}
\end{equation}

\begin{figure}[htbp]
\centering
\includegraphics[width=0.78\textwidth]{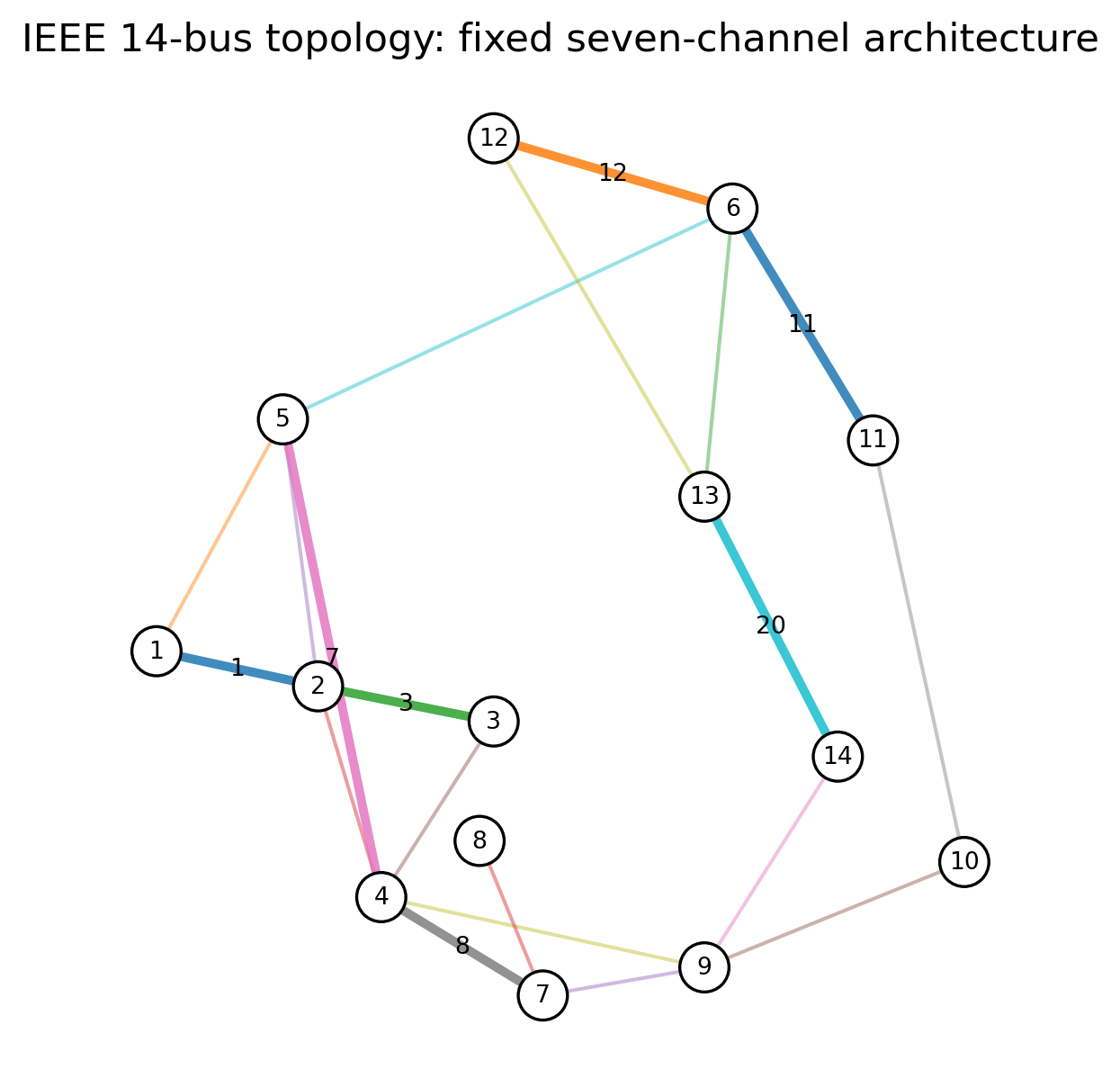}
\caption{IEEE 14-bus topology and the fixed ordered seven-edge sensor--actuator architecture $R=\{1,3,7,8,11,12,20\}$.  The same physical channels are used for the uncoupled selected-line comparison and for the rank-seven Hodge-aligned controller.}
\label{fig:ieee14-topology}
\end{figure}

\subsection{Operating point and finite-channel architecture}
The main transmission-grid application uses the IEEE 14-bus topology, but not a complete dynamic IEEE data set.  MATPOWER provides static power-flow data; a full engineering swing study would additionally require machine inertias, damping constants, exciters, governors, converter models, protection logic, and device limits.  We therefore retain the IEEE 14 branch topology and reactances, while specifying the dynamic and controller data explicitly.

Let $B\in\Lin(\R^{20},\R^{14})$ be the incidence operator of the IEEE 14 topology and let
\begin{equation}
        a_e
        =
        \frac{x_e^{-1}}{\operatorname{median}_{\ell}(x_\ell^{-1})}
\end{equation}
be the normalized line coefficient obtained from the branch reactance $x_e$.  The common nonlinear state is
\begin{equation}
        (\theta,\omega,u)\in\R^{14}\times\R^{14}\times\R^{20},
\end{equation}
and all five cases solve
\begin{subequations}\label{eq:ieee14-full-nonlinear-system}
\begin{align}
        \dot\theta
        &=\omega,\\
        \dot\omega
        &=p-1.2\omega-Bf(\theta,u),\\
        f(\theta,u)
        &=\operatorname{diag}(a)\sin(B^\top\theta+u),
\end{align}
\end{subequations}
so $M=I_{14}$ and $D=1.2I_{14}$.

The operating point is constructed so that its line flow is exactly cycle-free in the metric of the local sensitivity matrix.  Choose a small phase profile $\theta_*$, put
\begin{equation}
        d=B^\top\theta_*,
        \qquad
        \eta_*=\arctan d,
        \qquad
        u_*=\eta_*-d,
        \label{eq:ieee14-operating-point-construction}
\end{equation}
where the arctangent is componentwise, and define
\begin{equation}
        W_*=\operatorname{diag}\bigl(a_e\cos\eta_{*,e}\bigr)
        \in\Lin(\R^{20}).
\end{equation}
Then
\begin{equation}
        f_*=f(\theta_*,u_*)
        =W_*B^\top\theta_*,
        \qquad
        P_H^{W_*}f_*=0,
        \label{eq:ieee14-cycle-free-operating-flow}
\end{equation}
and the injection vector is fixed consistently by
\begin{equation}
        p=Bf_*.
\end{equation}
The initial state is
\begin{equation}
        \theta(0)=\theta_*,
        \qquad
        \omega(0)=0,
        \qquad
        u(0)=u_*+W_*^{-1}h_0,
        \label{eq:ieee14-nonlinear-initial-state}
\end{equation}
where
\begin{equation}
        h_0\in\Ker B,
        \qquad
        \norm{h_0}_{W_*^{-1}}=0.35.
\end{equation}
The difference
\begin{equation}
        d_0:=u(0)-u_*=W_*^{-1}h_0
        \label{eq:ieee14-initial-edge-phase-mismatch}
\end{equation}
should be read as a pre-existing edge-phase mismatch relative to the desired operating setting, not as actuator movement generated after $t=0$.  It may represent stale or unequal line-device settings, converter offsets, switching history, or another reduced line-level disturbance.  Since $\mathrm D_u f(\theta_*,u_*)=W_*$,
\begin{equation}
        f(\theta_*,u(0))-f_*
        =h_0+O(\norm{d_0}^2).
        \label{eq:ieee14-initial-harmonic-first-order}
\end{equation}
Thus the initial line-flow error is harmonic to first order.  The subsequent quantity
\begin{equation}
        \Delta u(t):=u(t)-u(0)
        \label{eq:ieee14-cumulative-actuator-movement}
\end{equation}
is the correction actually produced by the selected devices.

The five actuator laws are stated explicitly.  Passive swing damping uses
\begin{equation}
        \dot u=0.
\end{equation}
For the localized comparisons, use the fixed ordered architecture $R$ from \eqref{eq:ieee14-fixed-channel-list}.  Let
\begin{equation}
        S_R\in\Lin(\R^{20},\R^7)
\end{equation}
be coordinate restriction to these edges.  Its Euclidean adjoint
\begin{equation}
        A_R=S_R^*\in\Lin(\R^7,\R^{20})
\end{equation}
is the colocated coordinate injection and is represented in canonical coordinates by $S_R^\top$.  Put
\begin{equation}
        E_R=A_RS_R\in\Lin(\R^{20}),
\end{equation}
and use the same physical channels for selected-line damping and for the finite Hodge-aligned controller.  The selected-line and global phase controllers are
\begin{equation}
        \dot u
        =
        -\frac{\kappa}{\tau}W_*^{-1}E_Rf(\theta,u),
        \label{eq:ieee14-selected-line-nonlinear-law}
\end{equation}
\begin{equation}
        \dot u
        =
        -\frac{\kappa}{\tau}W_*^{-1}f(\theta,u),
        \label{eq:ieee14-global-nonlinear-law}
\end{equation}
while the metric-compatible ideal topological actuator is
\begin{equation}
        \dot u
        =
        -\frac{\kappa}{\tau}W_*^{-1}P_H^{W_*}f(\theta,u).
        \label{eq:ieee14-topological-nonlinear-law}
\end{equation}

Let $H_*\in\Lin(\R^7,\R^{20})$ have $W_*^{-1}$-orthonormal columns in the sense of \eqref{eq:section15-weighted-projectors}.  For the minimal finite-channel realization, set
\begin{equation}
        G_R:=S_RH_*\in\Lin(\R^7).
\end{equation}
For the fixed architecture \eqref{eq:ieee14-fixed-channel-list}, $G_R$ is invertible.  The canonical colocated gain is
\begin{equation}
        K_R
        =
        -\kappa(G_RG_R^\top)^{-1},
        \label{eq:ieee14-finite-hodge-gain}
\end{equation}
and the fifth controller is
\begin{equation}
        \tau\dot u
        =
        A_RK_RS_R\bigl(f(\theta,u)-f_*\bigr).
        \label{eq:ieee14-finite-hodge-nonlinear-law}
\end{equation}
This is the specialization of \eqref{eq:section15-finite-nonlinear-closed-loop} with $C=W_*A_R$.  Moreover,
\begin{equation}
        H_*^\top A_RK_RS_RH_*
        =
        -\kappa I_7,
        \label{eq:ieee14-finite-harmonic-compression}
\end{equation}
so the finite controller has exact weighted harmonic compression at the operating point, although its full edge-space action need not coincide with the ideal projector.

The fixed seven-edge architecture is illustrative.  By
Remark~\ref{rem:minimal-colocated-cotrees}, the full-rank minimal
colocated architectures are exactly the co-trees of the IEEE 14-bus
graph.  An exhaustive search may therefore enumerate spanning trees
$T$ and take
\[
        R=E\setminus T,
\]
rather than testing all
\[
        \binom{20}{7}=77\,520
\]
seven-edge subsets.  For the IEEE 14 topology used here, the
Matrix--Tree Theorem gives
\[
        N_{\mathrm{st}}(G)=3909,
\]
so only $3909$ admissible full-rank configurations need be compared.
For each such set $R$, the exact-compression gain is fixed by
\eqref{eq:ieee14-finite-hodge-gain}; with
\begin{equation}
        L_R:=W_*S_R^\top K_RS_R,
\end{equation}
one may compare, for example, the weighted worst-case distance
\begin{equation}
        d(R)
        :=
        \norm{
        W_*^{-1/2}
        \bigl(L_R+\kappa P_H^{W_*}\bigr)
        W_*^{1/2}
        }_2.
        \label{eq:ieee14-placement-benchmark}
\end{equation}
Because every admissible co-tree already has exact harmonic
compression, this distance measures only residual cut action and
cut--harmonic spillover.  Carrying out and engineering this placement
search is left to future work; full tangent stability, actuator effort,
thermal limits, and robustness remain additional design requirements.
No optimized placement is used in the experiment below.

\subsection{Nonlinear experiment and results}
The experiment uses
\begin{equation}
        \kappa=1,
        \qquad
        \tau=0.6,
        \qquad
        0\le t\le25.
\end{equation}
Equations~\eqref{eq:ieee14-full-nonlinear-system} together with the five actuator laws above are integrated numerically by an adaptive eighth-order Dormand--Prince method.  The state dimension is $14+14+20=48$.  Cycle energy and transfer error are measured in the $W_*^{-1}$ edge metric.

The experiment instantiates the general closed-loop theory of the preceding section: the useful transfer component is anchored by the cycle-free operating point $f_*$, while the initial perturbation is placed in the harmonic line-flow sector.  The global controller acts on the complete line-flow output and therefore disturbs transfer.  The ideal topological controller acts on the weighted harmonic output.  The rank-seven controller uses only the fixed channels in \eqref{eq:ieee14-fixed-channel-list}; it retains exact harmonic compression but may generate cut-space spillover because \eqref{eq:ieee14-finite-harmonic-compression} does not determine its complete edge-space action.

\begin{figure}[htbp]
\centering
\begin{subfigure}{0.48\textwidth}
\centering
\includegraphics[width=\textwidth]{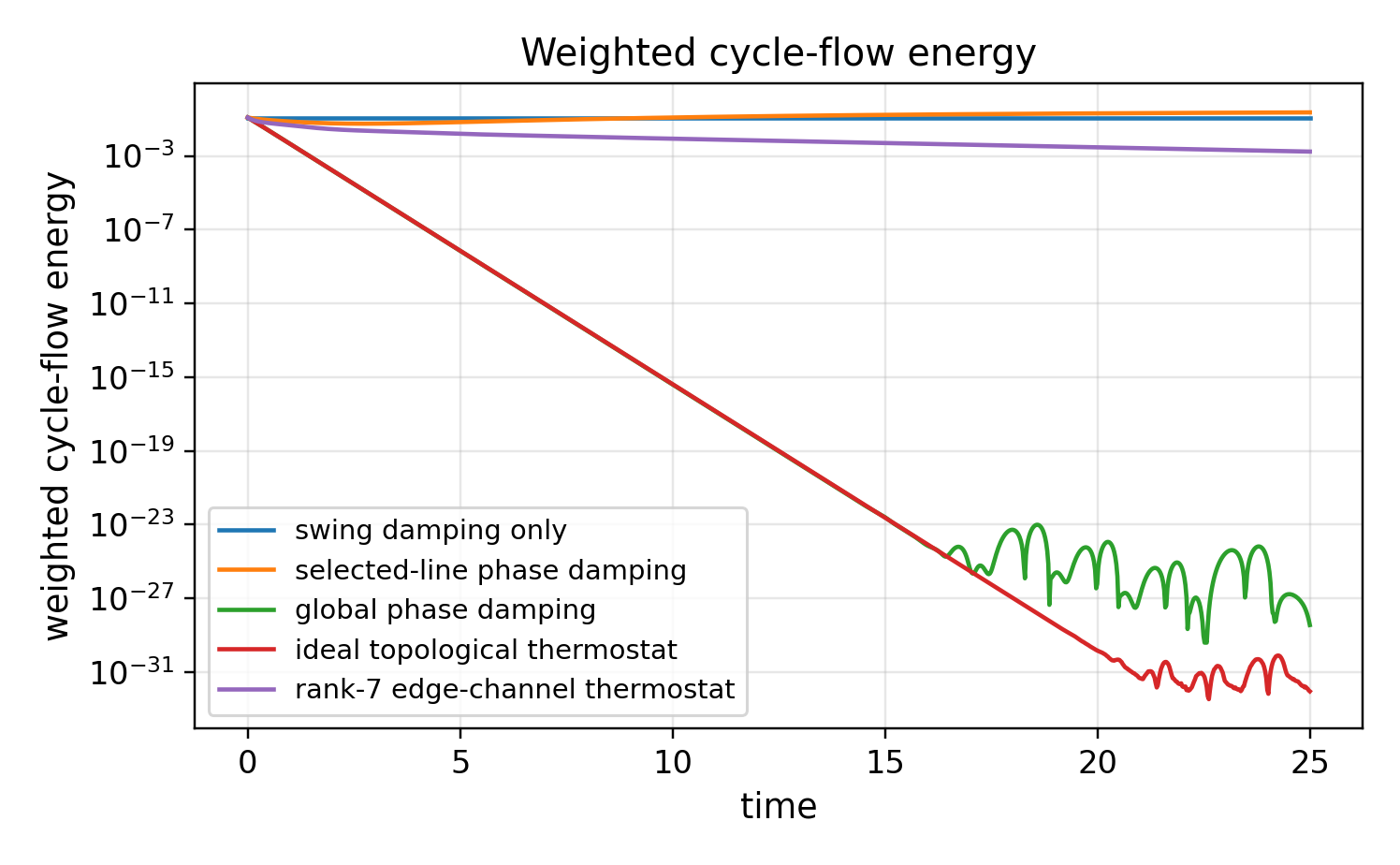}
\caption{Weighted cycle-flow energy.}
\end{subfigure}\hfill
\begin{subfigure}{0.48\textwidth}
\centering
\includegraphics[width=\textwidth]{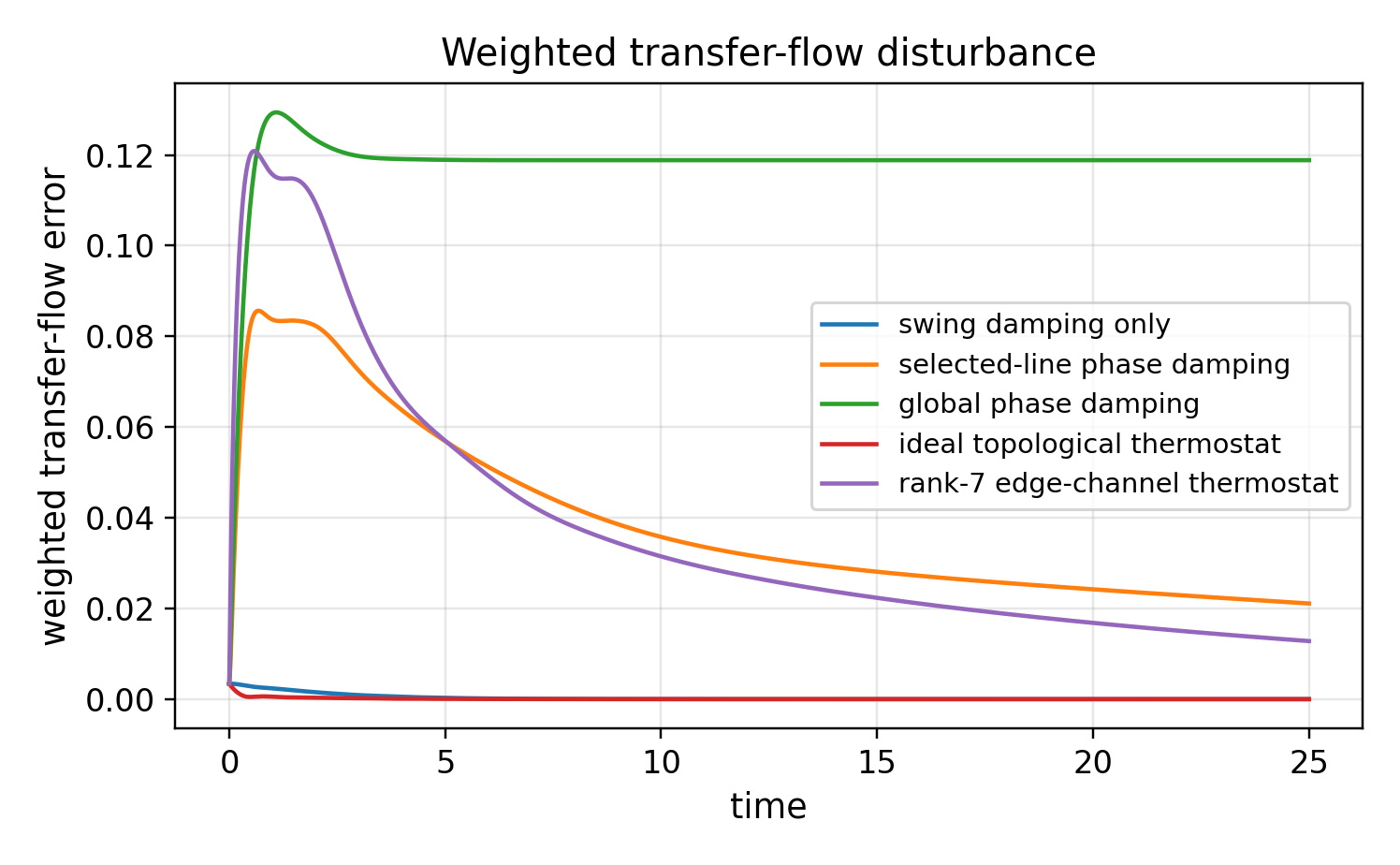}
\caption{Weighted transfer-flow error.}
\end{subfigure}
\caption{Synthetic nonlinear IEEE 14 swing experiment.  The ideal thermostat is Hodge-selective; the rank-seven realization removes most accumulated cycle energy but produces a transient transfer error through cut-space spillover.}
\label{fig:swing-ieee14-flow}
\end{figure}

\begin{figure}[htbp]
\centering
\begin{subfigure}{0.48\textwidth}
\centering
\includegraphics[width=\textwidth]{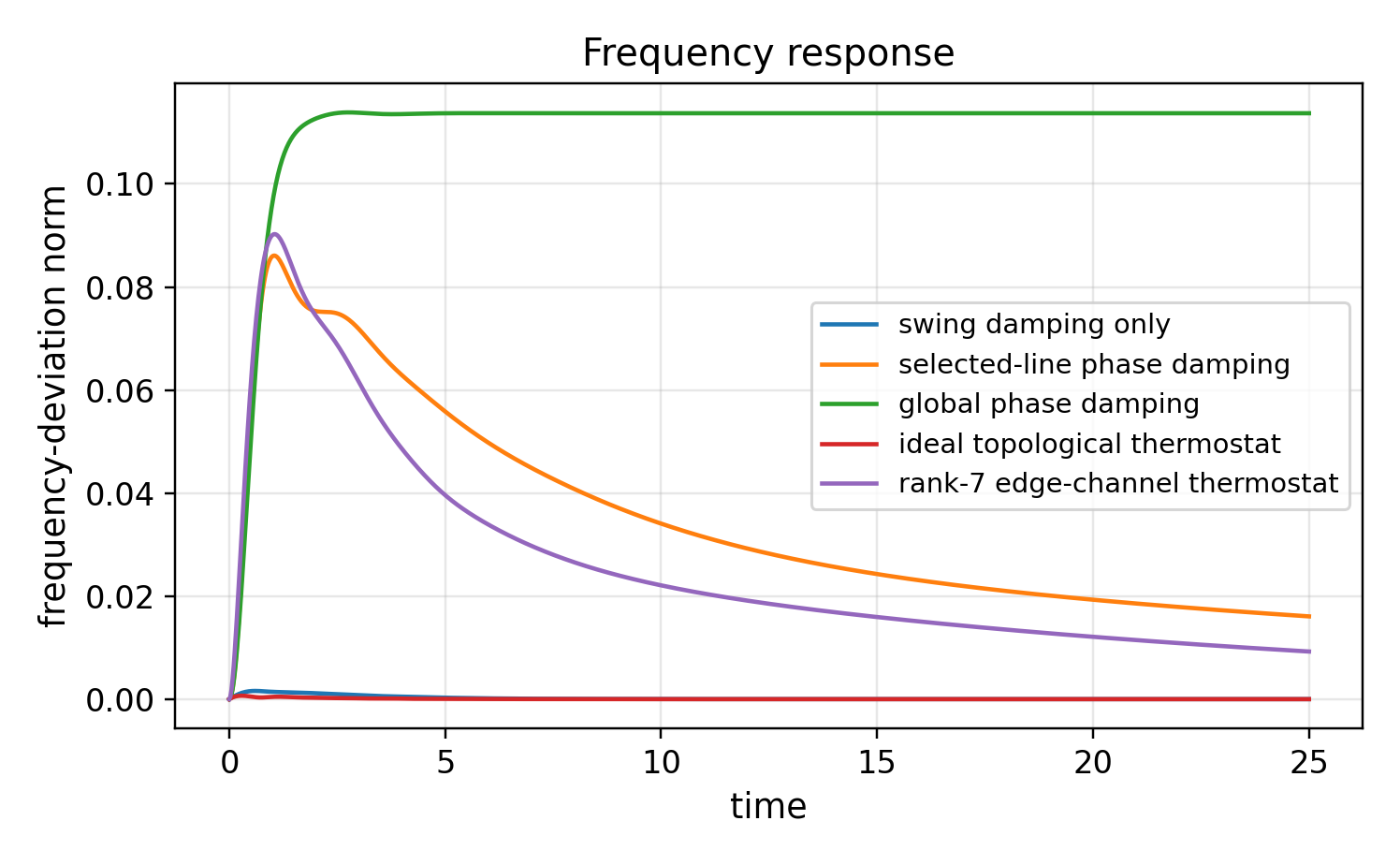}
\caption{Frequency-deviation norm.}
\end{subfigure}\hfill
\begin{subfigure}{0.48\textwidth}
\centering
\includegraphics[width=\textwidth]{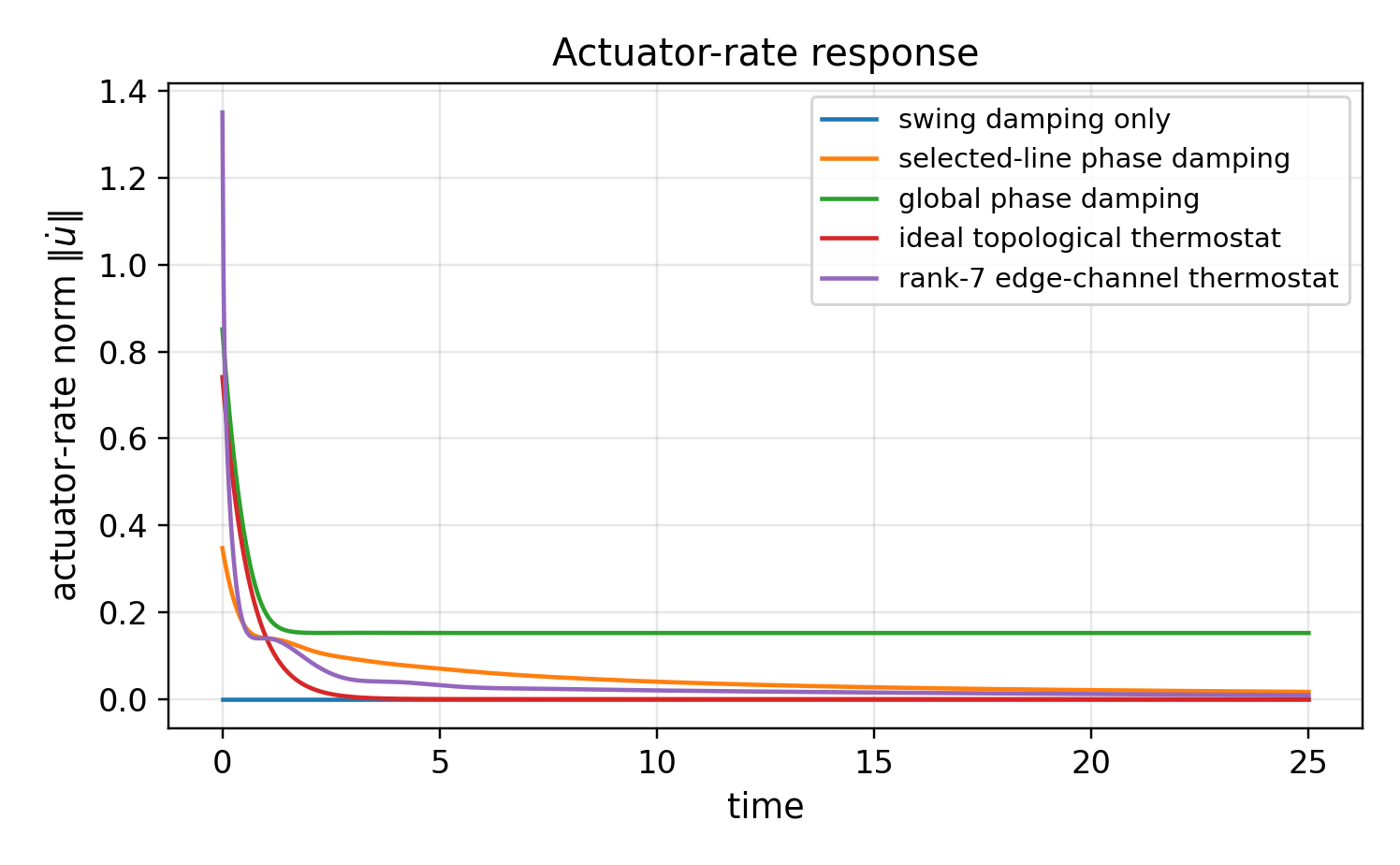}
\caption{Actuator-rate norm $\norm{\dot u}$.}
\end{subfigure}
\caption{Frequency excursion and actuator-rate norm in the synthetic IEEE 14 experiment.  Nonselective and localized controllers excite the electromechanical state because their full edge-space action is not confined to the harmonic sector.}
\label{fig:swing-ieee14-dynamics}
\end{figure}

\begin{table}[H]
\centering
\small
\resizebox{\textwidth}{!}{%
\begin{tabular}{lrrrr}
\toprule
Controller & Integrated cycle energy & Reduction vs. passive & Max transfer error & Max frequency norm\\
\midrule
Swing damping only & 2.990 & 0.0\% & 0.00342 & 0.00160\\
Selected-line phase damping & 3.548 & $-18.7$\% & 0.08565 & 0.08608\\
Global phase damping & 0.0366 & 98.8\% & 0.12938 & 0.11388\\
Ideal topological thermostat & 0.0365 & 98.8\% & 0.00342 & 0.00066\\
Rank-seven Hodge-aligned thermostat & 0.285 & 90.5\% & 0.12087 & 0.09026\\
\bottomrule
\end{tabular}%
}
\caption{Synthetic IEEE 14 metrics in the $W_*^{-1}$ edge metric.  The comparison distinguishes exact harmonic compression from exact ideal edge-space action.}
\label{tab:swing-ieee14-metrics}
\end{table}

\begin{remark}[Energy interpretation of the localized nonlinear experiment]
Corollary~\ref{cor:finite-channel-bregman-dissipation} applies directly to the rank-seven controller used above, with $S=S_R$, $S^*=A_R$, and $K=K_R$.  Its Bregman energy is therefore nonincreasing even though Figures~\ref{fig:swing-ieee14-flow} and~\ref{fig:swing-ieee14-dynamics} show substantial transfer and frequency transients.  There is no contradiction: the localized controller may exchange energy among the harmonic line-flow error, the cut-space error, and the bus kinetic energy while the combined target-centred storage decreases.  The numerical script verifies both the compression identity and the sign of the Bregman dissipation along the computed trajectory.
\end{remark}

\begin{remark}[Limiting actuator setting]
Recall from \eqref{eq:ieee14-cumulative-actuator-movement} that
\[
        \Delta u(t)=u(t)-u(0)
\]
is the cumulative actuator movement generated after the initial time.  For the rank-seven controller \eqref{eq:ieee14-finite-hodge-nonlinear-law}, one has $\dot u(t)\in\im A_R$ and hence $\Delta u(t)\in\im A_R$ for every $t$.  Define the selected-channel actuator movement by
\begin{equation}
        (\Delta u)_R(t)
        :=S_R\Delta u(t)
        \in\R^7.
        \label{eq:ieee14-selected-channel-movement}
\end{equation}
Since $S_RA_R=I_7$ and $A_RS_R$ is the coordinate projector onto $\im A_R$,
\begin{equation}
        \Delta u(t)=A_R(\Delta u)_R(t),
        \qquad
        u(t)=u(0)+A_R(\Delta u)_R(t).
        \label{eq:ieee14-actuator-coordinate-path}
\end{equation}
Applying $S_R$ to \eqref{eq:ieee14-finite-hodge-nonlinear-law} also gives
\begin{equation}
        \tau\frac{\dd}{\dd t}(\Delta u)_R
        =K_RS_R\bigl(f(\theta,u)-f_*\bigr),
        \qquad
        (\Delta u)_R(0)=0.
        \label{eq:ieee14-selected-channel-movement-ode}
\end{equation}

For every $t$,
\begin{equation}
        \eta(t)-\eta_*
        =B^\top\bigl(\theta(t)-\theta_*\bigr)
        +d_0+A_R(\Delta u)_R(t).
        \label{eq:ieee14-effective-phase-limit-identity}
\end{equation}
Multiplication by $H_*^\top$ eliminates the bus-phase term.  Since $G_R=S_RH_*$ and $H_*^\top A_R=G_R^\top$,
\begin{equation}
        H_*^\top\bigl(\eta(t)-\eta_*\bigr)
        =H_*^\top d_0+G_R^\top(\Delta u)_R(t).
        \label{eq:ieee14-selected-channel-limit-identity}
\end{equation}
Under the compact cohesive forward-invariance hypothesis of Corollary~\ref{cor:finite-channel-bregman-dissipation}, one has $f(\eta(t))\to f_*$.  Injectivity of the componentwise sine law on the cohesive branch then gives $\eta(t)\to\eta_*$.  Since $G_R$ is invertible,
\begin{equation}
        \boxed{
        (\Delta u)_{R,\infty}
        :=\lim_{t\to\infty}(\Delta u)_R(t)
        =-(G_R^\top)^{-1}H_*^\top d_0.}
        \label{eq:ieee14-limiting-actuator-coordinate}
\end{equation}
Consequently,
\begin{equation}
        \boxed{
        \Delta u_\infty
        :=\lim_{t\to\infty}\Delta u(t)
        =-A_R(G_R^\top)^{-1}H_*^\top
          \bigl(u(0)-u_*\bigr).}
        \label{eq:ieee14-limiting-actuator-setting}
\end{equation}
Thus the initial mismatch and the actuator architecture determine the final setting; $\kappa$ and $\tau$ determine the transient path and time scale, provided the same target is reached.  In the present experiment,
\begin{equation}
        \norm{\Delta u(25)}_2\approx0.6996,
        \qquad
        \norm{\Delta u_\infty}_2\approx0.8094.
\end{equation}
The continued growth of $\norm{\Delta u(t)}_2$ at the plotted horizon therefore represents convergence toward a finite nonzero setting, not instability.

Mathematically, $u_e$ is an ideal continuous branch phase-shift coordinate.  For a phase-shifting transformer, a nonzero $\Delta u_{\infty,e}$ corresponds qualitatively to a new tap position and phase-shift angle that is retained after the flow redistribution has settled.  A mechanical device would realize this motion through bounded discrete tap changes, deadbands, delays, and rate limits; the present model captures the constitutive and setpoint role of such a device, not its detailed dynamics \cite{ENTSOEPST,SiemensPST}.
\end{remark}

\section{Conclusions and outlook}
Passive Hodge dynamics damps positive spectral modes while leaving the harmonic kernel invariant.  The ideal topological thermostat converts this persistence into prescribed damping while regulating a cycle-free transfer target.  Finite-channel realizations separate a topology-only feasibility question from metric- and dynamics-dependent quality: the first Betti number fixes the minimal sensing and actuation ranks for full harmonic compression, but exact compression may coexist with cut action, cross-sector spillover, and non-normal transients.  The distributed Hodge projector therefore provides a canonical selectivity benchmark rather than a claim of practical realizability.

For nonlinear swing systems, linearizing the complete target-compatible closures at an angle-cohesive equilibrium shows that the ideal controller appends a stable harmonic block without changing the reduced passive nodal--cut spectrum.  The target-centred Bregman balance identifies the corresponding nonlinear mechanism: metric-compatible ideal feedback and canonical colocated feedback make the actuator port dissipative, although only the ideal law is perfectly Hodge-selective.  Harmonic compression, Bregman dissipation, and stability of the complete reduced closed loop are distinct requirements.

The examples progress from exact edge-flow models to a synthetic IEEE 14-bus experiment.  They show how useful transfer and internal circulation separate, how a winding-one state can carry divergence-free power flow, and how localization produces transfer and frequency transients even when harmonic compression is exact.  In the IEEE 14 setting, co-trees characterize the minimal full-rank colocated architectures; comparing them by conditioning, spillover, or distance from the ideal thermostat is a natural future placement problem.

Practical assessment requires explicit models of PST, FACTS, HVDC, or inverter devices, including lag, delay, saturation, rate limits, losses, measurement error, protection constraints, and topology changes.  Larger calibrated networks should combine harmonic-rank criteria with thermal, contingency, robustness, and cost objectives.  On the mathematical side, a natural continuation is feedback-perturbed Hodge heat in which the positive spectrum remains parabolically stable while the controlled harmonic compression undergoes bifurcation \cite{AmannHopf}.

\section*{Conflict of interest}
The author declares no conflict of interest.

\section*{Funding information}
No external funding was received for this work.

\section*{Data and code availability}
The companion source package contains the LaTeX source, all manuscript figures, and three Python scripts reproducing the numerical figures, tables, and verification checks.  Further implementation details are given in the accompanying README.

\section*{Declaration of generative AI and AI-assisted technologies}
During preparation of this manuscript, OpenAI ChatGPT was used for editorial, organizational, computational, and literature-search assistance.  The author verified and assumes full responsibility for all mathematical content, code, figures, and references.

\end{document}